\documentclass{article}

\usepackage[utf8]{inputenc}
\usepackage[english]{babel}

\usepackage{geometry}
\usepackage{fancyhdr}
\usepackage{titlesec}
\usepackage{indentfirst}
\usepackage{float}

\usepackage{amsmath,comment}
\usepackage{amssymb}
\usepackage{mathrsfs}
\usepackage{amsthm}
\usepackage{amsfonts}
\usepackage{ifthen}

\usepackage[linktoc=all]{hyperref}
\usepackage{cleveref}

\usepackage{tikz}

\theoremstyle{plain}
\newtheorem{thm}{Theorem}[section]
\newtheorem{prop}[thm]{Proposition}
\newtheorem{lem}[thm]{Lemma}
\newtheorem{cor}[thm]{Corollary}
\newtheorem{defn}[thm]{Definition}
\newtheorem{question}[thm]{Question}

\newtheorem{introthm}{Theorem}

\newtheorem{introcor}[introthm]{Corollary}

\theoremstyle{remark}

\newtheorem*{remark*}{Remark}
\newtheorem{remark}[thm]{Remark}

\DeclareMathOperator{\rank}{rk}

\newcommand{\bA}{\mathbf{A}}

\newcommand{\ba}{\mathbf{a}}
\newcommand{\bb}{\mathbf{b}}
\newcommand{\bc}{\mathbf{c}}
\newcommand{\bd}{\mathbf{d}}
\newcommand{\bt}{\mathbf{t}}
\newcommand{\bu}{\mathbf{u}}

\newcommand{\bw}{\mathbf{w}}
\newcommand{\bx}{\mathbf{x}}

\newcommand{\bz}{\mathbf{z}}

\newcommand{\bbN}{\mathbb{N}}
\newcommand{\bbZ}{\mathbb{Z}}
\newcommand{\bbQ}{\mathbb{Q}}

\newcommand{\cA}{\mathcal{A}}

\newcommand{\cG}{\mathcal{G}}
\newcommand{\cH}{\mathcal{H}}

\newcommand{\cP}{\mathcal{P}}
\newcommand{\cR}{\mathcal{R}}

\newcommand{\fD}{\mathfrak{D}}
\newcommand{\fF}{\mathfrak{F}}

\newcommand{\rar}{\rightarrow}

\newcommand{\ol}[1]{\overline{#1}}
\newcommand{\ot}[1]{\widetilde{#1}}
\newcommand{\oc}[1]{\widehat{#1}}
\newcommand{\abs}[1]{\left|#1\right|}
\newcommand{\sgr}{\le}

\newcommand{\gen}[1]{\langle #1 \rangle}
\newcommand{\pres}[2]{\langle #1 | #2 \rangle}

\newcommand{\supp}[1]{\mathrm{supp}(#1)}
\newcommand{\divides}{\bigm|}

\newcommand{\GL}[2]{\mathrm{GL}_{#1}(#2)}

\newcommand{\pA}{\mathbf{A}^+}

\newcommand{\FG}[1]{\mathrm{FG}(#1)}
\newcommand{\cnj}{\sim_{\mathrm{c}}}
\newcommand{\cnjclass}[1]{[#1]_{\mathrm{c}}}

\newcommand{\xdash}[1][3em]{\rule[0.5ex]{#1}{0.55pt}}
\newcommand{\edgedash}[1]{\ \xdash[#1]\ }
\newcommand{\edge}{\edgedash{0.6cm}}
\newcommand{\cyclicV}[1]{V_{=\bbZ}(#1)}
\newcommand{\notcyclicV}[1]{V_{\not=\bbZ}(#1)}

\newcommand{\notcyclicR}[1]{\cR_{\not=\bbZ}(#1)}
\newcommand{\out}[1]{\mathrm{Out}(#1)}

\title{On the isomorphism problem for cyclic JSJ decompositions: vertex elimination}

\author{
Dario Ascari\\
{\small \textit{Department of Mathematics, University of the Basque Country,}}\\
{\small \textit{Barrio Sarriena, Leioa, 48940, Spain}}\\
{\small e-mail: \texttt{ascari.maths@gmail.com}}\\
\and
Montserrat Casals-Ruiz\\
{\small \textit{Ikerbasque - Basque Foundation for Science and Matematika Saila,}}\\
{\small \textit{UPV/EHU,  Sarriena s/n, 48940, Leioa - Bizkaia, Spain}}\\
{\small e-mail: \texttt{montsecasals@gmail.com}}\\
\and
Ilya Kazachkov\\
{\small \textit{Ikerbasque - Basque Foundation for Science and Matematika Saila,}}\\
{\small \textit{UPV/EHU,  Sarriena s/n, 48940, Leioa - Bizkaia, Spain}}\\
{\small e-mail: \texttt{ilya.kazachkov@gmail.com}}
}

\date{}

\begin{document}

\maketitle

\begin{abstract}
    We introduce two new moves on graphs of groups with cyclic edge groups that preserve the fundamental group. These moves allow us to address the isomorphism problem without the use of expansions, therefore keeping the number of vertices and edges constant along sequences of moves witnessing an isomorphism between two groups. We further show that the isomorphism problem for a large family of cyclic JSJ decompositions reduces to the case of generalized Baumslag-Solitar groups (GBS), and that among GBSs, it suffices to consider one-vertex graphs. As an application of the methods, we solve the isomorphism problem for a broad class of flexible GBSs. Finally, we discuss potential further applications of our techniques.
\end{abstract}

%\keywords{GBS group, isomorphism problem} keywords = {Generalized Baumslag-Solitar groups, The isomorphism problem, Deformation space of trees}

% 21 figures

%20E08; 20F10; 20F28
%Groups acting on trees
%Word problems, other decision problems
%Automorphism groups of groups

\section{Introduction}

The isomorphism problem, formulated by Dehn around the turn of the 20th century, is a fundamental decision problem in group theory. It asks whether there exists an algorithm that, given two groups defined by finite presentations, can decide whether the groups are isomorphic. Although this problem is undecidable in general, it remains of fundamental interest for specific, relevant classes of groups. Even when restricted to a particular family of groups, the isomorphism problem is generally considered to be one of the hardest algorithmic problems in group theory. Only a few instances of successful solutions are known, and they usually rely on some rigidity assumption, either algebraic or geometric, on the family of groups taken into consideration.

On the algebraically rigid side, the most notable result is the decidability of the isomorphism problem for polycyclic groups established by Segal in \cite{Segal}. We note that the isomorphism problem for solvable groups is undecidable, see \cite{KirRe}, and it is an open question for finitely presented metabelian groups.

On the geometrically rigid side, one of the most important decidability results was established in the context of 3-manifolds. Given two (triangulations of) closed, oriented 3-manifolds, there is an algorithm to decide whether or not they are homeomorphic. This algorithm is based on several steps of \emph{surgery}: the 3-manifold is first cut along non-trivial spheres (irreducible decomposition) and then along a canonical family of incompressible tori (JSJ decomposition) in order to obtain a decomposition into simpler pieces that can be classified and endowed with exactly one of the eight Thurston's geometries, according to the Geometrization Conjecture. While performing the surgery, by van Kampen's theorem, the fundamental group $G$ of the manifold decomposes as a free product $G = G_1 \ast G_2$ (corresponding to a cut along a non-trivial sphere) or as an amalgamated product $G = G_1 \ast_{\mathbb Z^2} G_2$ (corresponding to a cut along a separating incompressible torus), or, more generally, as a graph of groups. Motivated by this approach, it is natural, for an abstract group $G$, to study the family of all possible splittings of $G$ as a graph of groups (with edge groups lying in some prescribed family). This idea leads to a rich and fruitful theory, known as the JSJ decomposition for groups \cite{RS97,DS99,FP06,GL17}. One of the most striking applications of the theory of JSJ decomposition is the solution of the isomorphism problem for hyperbolic groups \cite{Sel95,DG11,DT19}, in analogy with the case of $3$-manifolds.

For irreducible closed, oriented $3$-manifolds, the JSJ decomposition is unique in the strongest possible way: the JSJ tori are uniquely determined up to isotopy, see \cite{Hat}. For one-ended hyperbolic groups, a similar strong uniqueness holds, even though the construction is more delicate, based on the tree of cylinders, see \cite{GL11}. Therefore, any isomorphism essentially preserves the unique JSJ decomposition and the problem reduces to the isomorphism of the pieces (relative to the edge groups). In contrast, for other families of groups, the JSJ decomposition can exhibit significantly greater flexibility: not necessarily in the pieces obtained through cutting, but at least in the way in which they are glued together. In \cite{For02}, Forester explained how one can go from one gluing pattern of a group to any other using a certain set of moves, called \textit{elementary deformations}. In this way, given a splitting $T$ of a group $G$ (i.e. an action of $G$ on a tree $T$), the \textit{deformation space} $\fD_G(T)$ is the family of all splittings of $G$ that are related to $T$ by elementary deformations. The structure of the deformation spaces represents one of the main obstructions to solving the isomorphism problem (and to describing the automorphism group) for large families of groups. In this regard, the most basic and important example is the class of generalized Baumslag-Solitar groups.

A \textbf{generalized Baumslag-Solitar group} (GBS) is the fundamental group of a graph of groups where all vertex groups and edge groups are isomorphic to $\bbZ$. Every GBS can be represented as a finite graph labeled with two numbers at each edge determined by the index of the inclusion of the edge group into an adjacent vertex group. In this representation, Forester's elementary deformations translate into combinatorial moves that change the graph and the labels. The isomorphism problem is thus reduced to determining, given two labeled graphs, whether there is a sequence of elementary deformations going from one to the other. However, in general, there is no known (computable) bound on the number of vertices or edges in the graphs, or on the edge labels throughout the sequence. Providing such bounds would immediately resolve the isomorphism problem for GBSs, as one could then perform a brute-force search for a sequence of moves connecting any two given labeled graphs.

In this paper, we provide an explicit (computable) bound on the number of vertices and edges for the graphs along the sequence. This deals with the first of the two obstructions to solving the isomorphism problem. At the same time, our results also allow us to produce new isomorphism invariants, providing deeper insight into the structure of these groups. With these new tools, we are able to classify a large family of very flexible GBSs.

\paragraph{Vertex elimination.} One of the key observations, due to Forester \cite{For06}, is that some vertices in a GBS graph of groups have a stabilizer which is contained (up to conjugation) in the stabilizers of vertices in a different orbit; Forester refers to such vertices as \textit{horizontal}, while we prefer to call them \textit{redundant}, to emphasize the fact that they contain a repetition of information. Forester explains how to remove redundant vertices, in order to obtain a \textit{fully reduced} labeled graph. However, the vertex-elimination procedure described in \cite{For06} is not canonical. Studying different vertex-elimination procedures and the non-uniqueness of their outcomes leads us to the definition of two completely new moves - the \textbf{swap move} and the \textbf{connection move} - that can be applied to change a labeled graph into another one (preserving, up to isomorphism, the fundamental group of the associated graphs of groups).

The main result of the paper is the following \Cref{introthm:sequence-new-moves}: in order to decide if two GBSs are isomorphic, one only has to deal with sequences of slides, swaps, and connections (plus possibly sign-changes and inductions, which can be performed at the beginning). The key feature is that none of these moves changes the number of vertices or edges in the graph.

\begin{introthm}[\Cref{thm:sequence-new-moves}]\label{introthm:sequence-new-moves}
Let $(\Gamma,\psi),(\Delta,\phi)$ be totally reduced GBS graphs, with isomorphic fundamental groups {\rm(}and different from $\bbZ,\bbZ^2$ and the Klein bottle group{\rm)}. Then $\abs{V(\Gamma)}=\abs{V(\Delta)}$ and there is a sequence of slides, swaps, connections, sign-changes, and inductions going from $(\Delta,\phi)$ to $(\Gamma,\psi)$. Moreover, all the sign-changes and inductions can be performed at the beginning of the sequence.
\end{introthm}

\begin{remark*}
In the above \Cref{introthm:sequence-new-moves}, the assumption that a GBS graph is totally reduced is not restrictive, as one can algorithmically compute a sequence of moves that transforms any GBS graph into a totally reduced one (see \Cref{prop:compute-totally-reduced}). This notion is a variant of Forester’s concept of fully reduced GBS graph (see \cite{For06}), and is introduced for technical reasons. 
\end{remark*}

The proof of \Cref{introthm:sequence-new-moves} relies on systematically eliminating redundant vertices from a sequence of moves. Starting with a sequence of Forester’s elementary deformations, we identify a redundant vertex in one of the intermediate graphs and apply a vertex-elimination procedure to remove it. Although the outcome of this procedure is not unique, we show that any two resulting sequences differ only by slides, swaps, and connections. By repeating this process and progressively eliminating all redundant vertices from the graphs in the sequence, we arrive at the conclusion of \Cref{introthm:sequence-new-moves}. Using the description of the new moves in terms of Forester's elementary deformations, \Cref{introthm:sequence-new-moves} yields the following

\begin{introcor}[\Cref{cor:sequence-now-moves}]\label{introcor:sequence-Forester-moves}
Let $(\Gamma,\psi),(\Delta,\phi)$ be totally reduced GBS graphs, with isomorphic fundamental groups {\rm(}and different from $\bbZ,\bbZ^2$ and the Klein bottle group{\rm)}. Then there is a sequence of Forester's elementary deformations going from $(\Delta,\phi)$ to $(\Gamma,\psi)$ in such a way that every graph along the sequence has at most $\abs{V(\Gamma)}+2$ vertices.
\end{introcor}

\begin{remark*}
    Note that in both \Cref{introthm:sequence-new-moves} and \Cref{introcor:sequence-Forester-moves}, the sequences of moves happen at the level of finite labeled graphs, and not at the level of Bass-Serre trees.
\end{remark*}

\paragraph{Relating the isomorphism problem for GBS groups to other families of groups.}

We further investigate how the elementary deformations defined for GBS groups can be extended to other families of groups, thereby reducing the isomorphism problem for these families to the isomorphism problem for GBSs.

Firstly, we apply these ideas to compare the isomorphism problem within the family of GBSs, showing that it (essentially) suffices to solve the problem for graphs with one vertex and only positive labels different from 1 (and using sequences of moves that involve only such graphs).

\begin{introthm}[Theorem \ref{thm:one-vertex} and Proposition \ref{prop:one-vertex-positive}] \label{introthm:one-vertex-positive}
    Let $(\Gamma_0,\psi_0)$ and $(\Gamma_1,\psi_1)$ be totally reduced GBS graphs and let $b:V(\Gamma_1)\rar V(\Gamma_0)$ be a bijection. Then there are one-vertex GBS graphs $(\Delta_0,\phi_0),(\Delta_1,\phi_1)$, with labels which are positive and $\not=1$, such that the following are equivalent:
    \begin{enumerate}
        \item There is a sequence of slides, swaps, and connections going from $(\Gamma_0,\psi_0)$ to $(\Gamma_1,\psi_1)$ and such that the vertex $v\in V(\Gamma_1)$ corresponds to the vertex $b(v)\in V(\Gamma_0)$.
        \item There is a sequence of slides, swaps, and connections going from $(\Delta_0,\phi_0)$ to $(\Delta_1,\phi_1)$.
    \end{enumerate}
    Moreover, $(\Delta_0,\phi_0),(\Delta_1,\phi_1)$ can be constructed algorithmically from $(\Gamma_0,\psi_0),(\Gamma_1,\psi_1)$, and $b$.
\end{introthm}

Secondly, we show that the isomorphism problem for many cyclic JSJ decompositions can be reduced to the case of GBSs. This means that generalized Baumslag-Solitar groups are not only the easiest possible example, but they actually encode the complexity of the isomorphism problem for large families of graphs of groups.

\begin{introthm}[Theorem \ref{thm:isomorphism-other-families}]\label{introthm:isomorphism-other-families}
    Let $\fF$ be a family of finitely presented torsion-free groups. Suppose that we have the following conditions:
    \begin{enumerate}
        \item $\fF$ has algorithmic conjugacy problem {\rm(Definition \ref{def:algorithmic-conjugacy})}.
        \item $\fF$ has algorithmic isomorphism problem with cyclic peripherals {\rm(Definition \ref{def:algorithmic-isomorphism})}.
        \item $\fF$ is $\bbZ$-algorithmic {\rm(Definition \ref{def:Zalgorithmic})}.
        \item $\fF$ has algorithmic unique roots {\rm(Definition \ref{def:algorithmic-unique-roots})}.
    \end{enumerate}
    If there is an algorithm to decide the isomorphism problem for GBSs, then there is an algorithm to decide the isomorphism problem for graphs of groups with vertex groups in $\fF$ and cyclic edge groups.
\end{introthm}

If we want to solve the isomorphism problem for graphs of groups with vertex groups in $\fF$ and edge groups isomorphic to $\bbZ$, then requiring that $\fF$ has solvable conjugacy and isomorphism problem is very natural. The condition of being $\bbZ$-algorithmic is the most delicate one, and it means that we need to be able to compute the cyclic JSJ decomposition for groups in $\fF$, relative to any prescribed peripheral structure. Finally, the condition of having algorithmic unique roots is needed for technical reasons; however, we believe this assumption could be weakened, encoding the isomorphism problem for even broader families of groups into the combinatorics of GBSs. In particular, we have the following.

\begin{introcor}[\Cref{{cor:reduction_iso_relative_hyp_nilpotent}}]\label{introthm:reduction_iso_relatively_hyp}
If there is an algorithm to decide the isomorphism problem for GBSs, then there is an algorithm to decide the isomorphism problem for graphs of groups, whose vertex groups are finitely presented, torsion-free, relatively hyperbolic with nilpotent parabolic subgroups, and whose edge groups are cyclic.
\end{introcor}

\begin{remark*}
In the subsequent paper \cite{ACK-out} we further explore this relationship and, using different techniques, we demonstrate how the (outer) automorphism groups are related.
\end{remark*}

\paragraph{New isomorphism invariants for flexible configurations.} The isomorphism problem for GBS groups has received considerable attention, leading to several results that resolve it for certain restricted subclasses of GBSs \cite{For06, Lev07, CF08, Dud17, CRKZ21, Wan25}. However, these results rely on strong rigidity assumptions (such as the group being 2-generated, admitting only finitely many reduced graphs, or having at most one mobile edge), which significantly limit the potential interactions between edges. In contrast, the new moves we introduce allow us to handle configurations with rich flexibility, involving numerous ascending and virtually ascending loops arranged in a way that permit a lot of interaction with each other.

In view of \Cref{introthm:one-vertex-positive}, it is natural to focus on the case of one-vertex GBS graphs. A one-vertex GBS graph $(\Gamma,\psi)$ is called \textit{controlled} if there is an edge with labels $n$ and $n\ell$ (for some $n,\ell\in\bbZ\setminus\{0\}$) such that every other label is a multiple of $n$ and divisor of $n\ell^k$ for some $k\in\bbN$. We prove the following:

\begin{introthm}[\Cref{cor:iso-controlled}]\label{introthm:iso-controlled}
    There is an algorithm that, given as input a one-vertex controlled GBS graph $(\Gamma,\psi)$ and a GBS graph $(\Delta,\phi)$, decides whether the corresponding GBS groups are isomorphic.
\end{introthm}

\begin{remark*}
    The above \Cref{introthm:iso-controlled} gives as a particular case the result of \cite{CRKZ21}. They consider the family of finite index subgroups $H$ of the GBS groups $G_{1,p}^d$ for $p,d\ge1$ integers (see \cite{CRKZ21} for the definitions). But each such subgroup $H$ can be represented by a one-vertex controlled GBS graph.
\end{remark*}

For the classification result of \Cref{introthm:iso-controlled}, we introduce two new invariants: the set of minimal points and a linear algebra invariant. The minimal points are the smallest integers (in the divisibility sense) that appear as edge labels in a fully reduced GBS graph; these depend only on the isomorphism class of the group, not on the particular graph. The linear algebra invariant is given by the subgroup generated by a $k$-tuple of vectors, one for each edge, inside a certain abelian group. Despite the $k$-tuple of vectors being derived from the GBS graph, the subgroup generated ultimately depends only on its isomorphism type.

\begin{remark*}
We emphasize that the techniques and invariants developed here have applications well beyond the class of controlled one-vertex GBS graphs. In the forthcoming papers \cite{ACK-iso2, ACK-iso3}, we introduce two additional invariants, \emph{rigid vectors} and \emph{angles}, and demonstrate how these can be used to classify broader families of GBSs. These developments bring us much closer to a strategy for solving the isomorphism problem in the general case.
\end{remark*}

\paragraph{Potential further developments.} We conclude by briefly discussing two natural questions that arise from the results of this paper.

The first question concerns the nature of our sequences of moves. Unlike Forester's deformations, our sequences of moves (as in \Cref{introthm:sequence-new-moves} and \Cref{introthm:one-vertex-positive}) take place at the level of finite labeled graphs rather than at the level of trees. This means that we are not describing the deformation space $\fD(G)$ of a GBS group $G$, but only its quotient $\fD(G)/\out{G}$ by the action of the outer automorphism group. 

\begin{question}
    Can our results be extended to sequenes of moves happening at the level of trees? If so, what are the consequences on the structure of $\out{G}$ for a GBS group $G$?
\end{question}

The second question concerns the potential for generalizing (some of) the results of this paper to graphs with bigger edge-stabilizers instead of GBSs. A first step in this direction would be to define a vertex-elimination procedure in this broader context: if there are two different orbits of vertices, and the stabilizers from the first orbit are contained in the stabilizers from the second orbit, then we would like to eliminate the first orbit of vertices by means of a sequence of moves. However, we present an example of a graph of $\bbZ^2$ where the inclusion between stabilizers holds, yet we conjecture that no sequence of moves can eliminate the corresponding orbit of vertices.

\begin{question}
    Can the vertex elimination technique be extended to graphs of $\bbZ^2$? Can it be extended to graphs of virtually cyclic groups?
\end{question}

\paragraph{Guide for the reader.} Section \ref{sec:preliminaries} sets up the notation and summarizes the background literature. In Section \ref{sec:GBSs}, we introduce the \textit{affine representation}, a graphical way of representing GBSs. In Section \ref{sec:moves}, we define the new swap and connection moves. In Section \ref{sec:controlled-linear-algebra}, we show how these moves can be used for significant manipulations of GBS graphs. In Section \ref{sec:coarse-projection-map}, we define the \textit{coarse projection map}, which plays a central role in the vertex-elimination arguments. In Section \ref{sec:sequences-of-moves}, we prove the main \Cref{introthm:sequence-new-moves}, and we deduce \Cref{introcor:sequence-Forester-moves}, \Cref{introthm:one-vertex-positive} and \Cref{introthm:iso-controlled}. In Section \ref{sec:comparing-deformation-spaces}, we address graphs of groups with arbitrary vertex groups and we prove \Cref{introthm:isomorphism-other-families} and \Cref{introthm:reduction_iso_relatively_hyp}. Finally, in Section \ref{sec:questions}, we discuss open questions and directions for future research.

\subsection*{Acknowledgements}

This work was supported by the Basque Government grant IT1483-22. The second author was supported by the Spanish Government grant PID2020-117281GB-I00, partly by the European Regional Development Fund (ERDF), the MICIU /AEI /10.13039/501100011033 / UE grant PCI2024-155053-2.

\section{Preliminaries}\label{sec:preliminaries}

In this section, we recall the fundamental concepts and set up the notation that we are going to use on graphs of groups and JSJ decompositions. We point out that Section \ref{sec:preliminaries} is completely expository, based on the already existing literature, and the (few) proofs provided are well-known standard arguments. For the Bass-Serre correspondence, we refer the reader to \cite{Ser77}. For a general and in-depth treatment of the theory of JSJ decomposition, we refer the reader to \cite{GL17}. We also remark that we are mostly going to apply the results of this section to finitely presented torsion-free groups and to splittings over the trivial group or over infinite cyclic groups.

\subsection{Graphs}

We consider graphs as combinatorial objects, following the notation of \cite{Ser77}. A \textbf{graph} is a quadruple $\Gamma=(V,E,\ol{\cdot},\iota)$ consisting of a set $V=V(\Gamma)$ of \textit{vertices}, a set $E=E(\Gamma)$ of \textit{edges}, a map $\ol{\cdot}:E\rar E$ called \textit{reverse} and a map $\iota:E\rar V$ called \textit{initial vertex}; we require that for every edge $e\in E$, we have $\ol{e}\not=e$ and $\ol{\ol{e}}=e$. For an edge $e\in E$, we denote by $\tau(e)=\iota(\ol{e})$ the \textit{terminal vertex} of $e$. A \textbf{path} in a graph $\Gamma$, with \textit{initial vertex} $v\in V(\Gamma)$ and \textit{terminal vertex} $v'\in V(\Gamma)$, is a sequence $\sigma=(e_1,\dots,e_\ell)$ of edges $e_1,\dots,e_\ell\in E(\Gamma)$ for some integer $\ell\ge0$, with the conditions $\iota(e_1)=v$ and $\tau(e_\ell)=v'$ and $\tau(e_i)=\iota(e_{i+1})$ for $i=1,\dots,\ell-1$. The terminal and initial vertices will often be referred to as \textit{endpoints}. A path $(e_1,\dots,e_\ell)$ is \textbf{reduced} if $e_{i+1}\not=\ol{e}_i$ for $i=1,\dots,\ell-1$. A graph is \textbf{connected} if for every two vertices there is a path going from one to the other. For a connected graph $\Gamma$, we define its \textbf{rank} $\rank{\Gamma}\in\bbN\cup\{+\infty\}$ as the rank of its fundamental group (which is a free group). A graph is a \textbf{tree} if for every pair of vertices, there is a unique reduced path going from one to the other.

\subsection{Graphs of groups}\label{sec:graphs-of-groups}

\begin{defn}[Graph of groups]\label{def:graph-of-groups}
A \textbf{graph of groups} is a quadruple
$$\cG=(\Gamma,\{G_v\}_{v\in V(\Gamma)},\{G_e\}_{e\in E(\Gamma)},\{\psi_e\}_{e\in E(\Gamma)})$$
consisting of a connected graph $\Gamma$, a group $G_v$ for each vertex $v\in V(\Gamma)$, a group $G_e$ for every edge $e\in E(\Gamma)$ with the condition $G_e=G_{\ol{e}}$, and an injective homomorphism $\psi_e:G_e\rar G_{\tau(e)}$ for every edge $e\in E(\Gamma)$.
\end{defn}

\begin{defn}[Universal group]\label{def:universal-group}
Let $\cG=(\Gamma,\{G_v\},\{G_e\},\{\psi_e\})$ be a graph of groups. Define the \textbf{universal group} $\FG{\cG}$ as the quotient of the free product $(*_{v\in V(\Gamma)}G_v)*F(E(\Gamma))$ by the relations
\[
    \ol{e}=e^{-1} \qquad\qquad \psi_{\ol{e}}(g)\cdot e=e\cdot\psi_e(g)
\]
for $e\in E(\Gamma)$ and $g\in G_e$.
\end{defn}

A \textbf{sound writing} for an element $x\in\FG{\cG}$ is given by a path $(e_1,\dots,e_\ell)$ in $\Gamma$ and by elements $g_i\in G_{\tau(e_i)}=G_{\iota(e_{i+1})}$ for $i=0,\dots,\ell$, such that $x=g_0e_1g_1e_2g_2\dots g_{\ell-1}e_\ell g_\ell$ in $\FG{\cG}$. An element $x\in\FG{\cG}$ is called \textbf{sound} if it admits a sound writing; in that case, the vertices $v=\iota(e_1),v'=\tau(e_\ell)\in V(\Gamma)$ do not depend on the writing (nor does the homotopy class relative to the endpoints of the path $(e_1,\dots,e_\ell)$ in $\Gamma$), and we say that $x$ is \textbf{$(v,v')$-sound}. We denote by $\pi_1(\cG,v,v')$ the set of all $(v,v')$-sound elements of $\FG{\cG}$.

Given a vertex $v_0\in V(\Gamma)$, we have a subgroup $\pi_1(\cG,v_0):=\pi_1(\cG,v_0,v_0)\sgr\FG{\cG}$ which we call \textbf{fundamental group} of $\cG$ with basepoint $v_0$; different choices of the basepoint $v_0$ produce conjugate (and in particular isomorphic) subgroups of $\FG{\cG}$. Given an element $g$ inside a vertex group $G_v$, we can take a path $(e_1,\dots,e_\ell)$ from $v_0$ to $v$: we define the conjugacy class $[g]:=[e_1\dots e_\ell g\ol{e}_\ell\dots \ol{e}_1]$ in $\pi_1(\cG,v_0)$, and notice that this does not depend on the chosen path.

\subsection{Bass-Serre correspondence}

An action of a group $G$ on a tree $T$ is called \textbf{without inversions} if $ge\not=\ol{e}$ of any $g\in G$ and $e\in E(T)$. If $G$ is acting on a tree $T$, then we can take the first barycentric subdivision $T^{(1)}$ of $T$, and the action of $G$ on $T^{(1)}$ is without inversions; thus it is not restrictive to limit ourselves to actions without inversions. \textit{From now on, every action of a group $G$ on a tree $T$ is assumed to be without inversions.} Given an action of a group $G$ on a tree $T$, we denote with $G_{\ot v}\sgr G$ the stabilizer of a vertex $\ot{v}\in V(T)$ and with $G_{\ot e}\sgr G$ the stabilizer of an edge $\ot{e}\in E(T)$.

Let $G$ be a group acting on a tree $T$ and let $\Gamma$ be the quotient graph. We make the following choices: for every vertex $v\in V(\Gamma)$ we choose a lifting $\ot v\in V(T)$; for every edge $e\in E(\Gamma)$ we choose a lifting $\ot e\in E(T)$; for every edge $e\in E(\Gamma)$ with $\tau(e)=v$ we choose an element $a_e\in G$ such that $\tau(a_e\cdot\ot e)=\ot v$. We define the graph of groups $\cG$ with underlying graph $\Gamma$; at each vertex $v\in V(\Gamma)$ we put the stabilizer $G_{\ot v}$ of the chosen lifting; at every edge $e\in E(\Gamma)$ we put the stabilizer $G_{\ot e}$ of the chosen lifting; for every edge $e\in E(\Gamma)$ we define the map $\psi_e:G_{\ot e}\rar G_{\ot v}$ given by $x\mapsto a_exa_e^{-1}$. In addition, we define the homomorphism $\Phi:\FG{\cG}\rightarrow G$, consisting of the natural inclusions $G_{\ot v}\rightarrow G$ for $v\in V(\Gamma)$, and given by $\Phi(e)=a_{\ol e}^{-1}a_e$ for $e\in E(\Gamma)$.

If $\cG=(\Gamma,\{G_v\},\{G_e\},\{\psi_e\})$ and $\cG'=(\Gamma,\{G_v'\},\{G_e'\},\{\psi_e'\})$ are graphs of groups obtained from the same action on a tree, but performing different choices, then there are isomorphisms $G_v\rar G_v'$ and $G_e\rar G_e'$ such that all the diagrams with the maps $\psi_e,\psi_e'$ commute up to conjugation. Thus the resulting graph of groups is essentially unique (up to changing the maps $\psi_e:G_e\rar G_{\tau(e)}$ by post-composing with conjugation in the vertex group $G_{\tau(e)}$; the conjugating element is allowed to depend on the chosen edge $e$).

\begin{prop}[Bass-Serre]\label{prop:Bass-Serre}
    For a group $G$, we have the following:
    \begin{itemize}
        \item For every action of $G$ on a tree $T$, the corresponding homomorphism $\Phi:\FG{\cG}\rightarrow G$ induces an isomorphism $\Phi:\pi_1(\cG)\rar G$ on the fundamental group (for all basepoints).
        \item Suppose that we are given a graph of groups $\cG'$ and a homomorphism $\Phi':\FG{\cG'}\rar G$ inducing an isomorphism $\Phi':\pi_1(\cG')\rar G$ on the fundamental group. Then there is an action of $G$ on a tree $T$, inducing a graph of groups $\cG$ and a homomorphism $\Phi:\FG{\cG}\rar G$, such that $\cG$ is isomorphic to $\cG'$ (the isomorphism commuting with the maps $\Phi,\Phi'$). The action is unique up to $G$-equivariant isomorphism.
    \end{itemize}
\end{prop}
\begin{proof}
    See \cite{Ser77}.
\end{proof}

An action of a group $G$ on a tree $T$ will be called a \textbf{splitting} of $G$. In virtue of the Bass-Serre correspondence of Proposition \ref{prop:Bass-Serre}, a morphism $\Phi:\cG\rar G$ inducing an isomorphism at the level of the fundamental groups will also be called a \textbf{splitting} of $G$. A splitting $T$ is \textbf{trivial} if there is a global fixed point, i.e. if there is a vertex $v_0\in V(T)$ with $G_{v_0}=G$. If $\cG$ is the corresponding graph of groups with underlying graph $\Gamma$, being trivial means that $\Gamma$ is a tree and contains a vertex $v_0\in V(\Gamma)$ satisfying the following property: for every $e\in E(\Gamma)$, if $\iota(e)$ is nearer to $v_0$ than $\tau(e)$, then $\psi_e:G_e\rar G_{\tau(e)}$ is an isomorphism.

\subsection{Minimal splittings}

A splitting $T$ of $G$ is \textbf{minimal} if $T$ contains no proper $G$-invariant sub-tree. Given a non-trivial splitting $T$ of $G$, there is a unique $G$-invariant sub-tree $T'$ which is minimal by inclusion (this is given by the intersection of all the $G$-invariant sub-trees); the action of $G$ on $T'$ is minimal. Notice that there is a $G$-equivariant retraction $T\rar T'$.

\begin{lem}
    Let $G$ be a finitely generated group. Suppose that we are given a minimal action of $G$ on a tree $T$. Then $T/G$ is a finite graph {\rm(}i.e. the corresponding graph of groups is finite{\rm)}.
\end{lem}
\begin{proof}
    Let $g_1,\dots,g_n$ be generators for $G$. Take a vertex $v\in V(T)$ and consider the reduced path $\sigma_i$ going from $v$ to $g_i\cdot v$, for $i=1,\dots,n$. The union of the images of the paths $g\cdot\sigma_i$, for $g\in G$ and $i=1,\dots,n$, defines a sub-tree $T'$ which is $G$-invariant and which has finite quotient $T'/G$.
\end{proof}

\subsection{Elliptic subgroups}

Let $G$ be a group acting on a tree $T$. We say that a subgroup $H\sgr G$ is \textbf{elliptic in $T$} if the induced action of $H$ on $T$ has a global fixed point; this means that $H$ is contained in the stabilizer $G_v$ for some vertex $v\in V(T)$. Notice that a subgroup $H$ is elliptic in $T$ if and only if all of its conjugates are elliptic in $T$. If $H$ is finitely generated, then it is elliptic in $T$ if and only if all of its elements are elliptic in $T$, as follows.

\begin{lem}[Serre]\label{lem:Serre-trick}
    Let $G$ be a group acting on a tree $T$ and let $H=\gen{h_1,\dots,h_n}\sgr G$ be a finitely generated subgroup. If $h_i,h_ih_j$ are elliptic in $T$ for all $i,j=1,\dots,n$, then $H$ is elliptic in $T$.
\end{lem}
\begin{proof}
    This is Corollary 2 in Section 6.5 of \cite{Ser77}.
\end{proof}

\begin{lem}\label{lem:elliptic-finite-index}
    Let $G$ be a group acting on a tree $T$, and let $H'\sgr H\sgr G$ be two subgroups such that $H'$ has finite index in $H$. Then $H'$ is elliptic in $T$ if and only if $H$ is elliptic in $T$.
\end{lem}
\begin{proof}
    See \cite{GL17}.
\end{proof}

\subsection{Domination and deformation spaces}\label{sec:deformation-spaces}

Let $G$ be a group acting on two trees $T,T'$. We say that $T$ \textbf{dominates} $T'$ if every vertex stabilizer of $T$ is elliptic in $T'$; this is equivalent to requiring that there is a $G$-equivariant map $f:T\rar T'$ sending each vertex to a vertex and each edge to a (possibly trivial) edge path.

Let $G$ be a group acting on a tree $T$. The \textbf{deformation space} $\fD_G(T)$ is the set of actions of $G$ on a tree $T'$ such that $T$ dominates $T'$ and $T'$ dominates $T$ (considered up to $G$-equivariant isomorphism). The terminology \textit{deformation space} is due to the fact that two trees lie in the same deformation space if and only if it is possible to pass from one to the other by a sequence of \textit{deformations}, as we now explain.

Let $G$ be a group acting on a tree $T$ and let $\cG=(\Gamma,\{G_v\},\{G_e\},\{\psi_e\})$ be the corresponding graph of groups. Let $e\in E(\Gamma)$ be an edge such that $\psi_e:G_e\rar G_{\tau(e)}$ is an isomorphism and $\iota(e)\not=\tau(e)$. Consider the quotient that collapses each edge in the orbit corresponding to $e\in E(\Gamma)$ to a point. Such a collapse is called \textbf{contraction}; the inverse of a contraction is called \textbf{expansion} (see Figure \ref{fig:contraction-expansion} for the effect of such moves on the corresponding graph of groups).

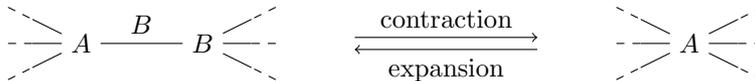
\begin{figure}[H]
\centering
\begin{tikzpicture}[scale=0.8]

\begin{scope}[shift={(0,0)}]
\node (v) at (0,0) {$A$};
\draw[-,dashed] (-1.2,0.6) to (-0.8,0.4);
\draw[-] (-0.8,0.4) to (v);
\draw[-,dashed] (-1.2,0) to (-0.8,0);
\draw[-] (-0.8,0) to (v);
\draw[-,dashed] (-1.2,-0.6) to (-0.8,-0.4);
\draw[-] (-0.8,-0.4) to (v);

\node (w) at (2,0) {$B$};
\draw[-,dashed] (3.2,0.6) to (2.8,0.4);
\draw[-] (2.8,0.4) to (w);
\draw[-,dashed] (3.2,0) to (2.8,0);
\draw[-] (2.8,0) to (w);
\draw[-,dashed] (3.2,-0.6) to (2.8,-0.4);
\draw[-] (2.8,-0.4) to (w);

\draw[-] (v) to node[above]{$B$} (w);
\end{scope}

\draw[->] (4.5,0.1) to node[above]{contraction} (7.5,0.1);
\draw[->] (7.5,-0.1) to node[below]{expansion} (4.5,-0.1);

\begin{scope}[shift={(10,0)}]
\node (v) at (0,0) {$A$};
\draw[-,dashed] (-1.2,0.6) to (-0.8,0.4);
\draw[-] (-0.8,0.4) to (v);
\draw[-,dashed] (-1.2,0) to (-0.8,0);
\draw[-] (-0.8,0) to (v);
\draw[-,dashed] (-1.2,-0.6) to (-0.8,-0.4);
\draw[-] (-0.8,-0.4) to (v);

\draw[-,dashed] (1.2,0.6) to (0.8,0.4);
\draw[-] (0.8,0.4) to (v);
\draw[-,dashed] (1.2,0) to (0.8,0);
\draw[-] (0.8,0) to (v);
\draw[-,dashed] (1.2,-0.6) to (0.8,-0.4);
\draw[-] (0.8,-0.4) to (v);
\end{scope}

\end{tikzpicture}
\caption{The contraction and expansion moves. The two endpoints of the edge to be collapsed are required to be distinct. The edge group $B$ is required to be isomorphic to one of the adjacent vertex groups; on the other side we allow for a proper inclusion $B\sgr A$.}
\label{fig:contraction-expansion}
\end{figure}

\begin{thm}[Forester \cite{For02}]\label{thm:Forester}
    Let $G$ be a group acting on two trees $T,T'$. Then the following are equivalent:
    \begin{itemize}
        \item The two trees $T$ and $T'$ lie in the same deformation space.
        \item There is a finite sequence of contractions and expansions going from $T$ to $T'$.
    \end{itemize}
\end{thm}
\begin{proof}
    This is Theorem 1.1 in \cite{For02}.
\end{proof}

\begin{remark}
    Suppose that we are given two trees $T,T'$ such that $T$ dominates $T'$. This means that there is an (equivariant) map of trees $f:T\rar T'$ sending each edge to an edge path. Up to (equivariantly) subdividing the edges of $T$, we can assume that, for every edge $e\in E(T)$, either $f(e)$ is a single edge of $T'$, or $f(e)$ is a vertex of $T'$ (i.e. $e$ is collapsed to a point). Then we can perform (equivariant) folding moves on $T$: whenever we see two edges $e,e'\in E(T)$ with a common endpoint $\iota(e)=\iota(e')$ and with common image $f(e)=f(e')$, we can fold them to a single edge, and the map to $T'$ remains well-defined. For a detailed analysis of the effect of these folding moves on the graph of groups structure, we refer to \cite{BF91}. Notice that, if we add the condition that $T'$ dominates $T$ (i.e. $T,T'$ lie in the same deformation space), this imposes strong restrictions on which behaviours can occur during the folding sequence.
\end{remark}

\subsection{Peripheral structures and refinements}

A \textbf{peripheral structure} on a group $G$ is a family $\cH$ of subgroups of $G$. A \textbf{splitting relative to the peripheral structure} is an action of $G$ on a tree $T$ such that every subgroup $H\in\cH$ is elliptic in $T$. For an action of $G$ on a tree $T$, every vertex stabilizer $G_v$ for $v\in V(T)$ has a natural peripheral structure, given by the edge stabilizers $G_e$ for $e\in E(T)$ with $\tau(e)=v$. Similarly, for a graph of groups $\cG=(\Gamma,\{G_v\},\{G_e\},\{\psi_e\})$, there is a natural peripheral structure on each vertex group $G_v$ for $v\in V(\Gamma)$, given by the images $\psi_e(G_e)$ for $e\in E(\Gamma)$ with $\tau(e)=v$.

Let $G$ be a group acting on a tree $\oc{T}$. Given a $G$-invariant subset of edges of $\oc{T}$, we can collapse each of those edges to a point; we obtain another tree $T$, together with an action of $G$ on $T$ and a $G$-equivariant quotient map $q:\oc{T}\rar T$. We say that $T$ is obtained from $\oc{T}$ by a \textbf{collapse}. We say that a tree $\oc{T}$ is a \textbf{refinement} of $T$ if there is a collapse $q:\oc{T}\rar T$.

Suppose that we are given a refinement $q:\oc{T}\rar T$. For every vertex $v\in V(T)$, we have an action of $G_v$ on the tree $q^{-1}(v)$: this induces a splitting of $G_v$ relative to the peripheral structure. Conversely, suppose that $\cG=(\Gamma,\{G_v\},\{G_e\},\{\psi_e\})$ is a graph of groups, and suppose that for every vertex $v\in V(\Gamma)$ we are given a splitting of $G_v$ relative to the peripheral structure. Then, for the corresponding action on a tree $T$, it is not hard to construct a refinement $q:\oc{T}\rar T$ inducing those splittings. Therefore, taking a refinement of a tree corresponds to taking the vertex groups of the induced graph of groups, and splitting each of them relative to the respective peripheral structure.

A refinement $q:\oc{T}\rar T$ is \textbf{proper} if there is a vertex $v\in V(T)$ such that the induced action of $G_v$ on $q^{-1}(v)$ is a non-trivial splitting. If $\oc{T}$ is a refinement of $T$, then $\oc{T}$ dominates $T$; moreover, $T$ and $\oc{T}$ are in the same deformation space if and only if the refinement is non-proper. We say that a splitting $S$ is \textbf{reduced} if every collapse $q:S\rar S'$ is a proper refinement. This is equivalent to requiring that no contraction move can be performed on $S$ (see Figure \ref{fig:contraction-expansion}).

\subsection{Accessibility}

\begin{prop}[Dunwoody's accessibility]\label{prop:Dunwoody-accessibility}
    Let $G$ be a finitely presented group. Let $T_0\leftarrow T_1\leftarrow T_2\leftarrow\dots $ be a sequence of refinements of trees. Then there is a tree $S$, with edge and vertex stabilizers finitely generated, such that $S$ dominates $T_k$ for all $k\ge0$.
\end{prop}
\begin{proof}
    See Proposition 2.17 in \cite{GL17}.
\end{proof}

\begin{prop}[Bestvina-Feighn \cite{BF91}]\label{prop:BestvinaFeighn-accessibility}
    Let $G$ be a finitely presented group. Then there is an integer $\gamma(G)\ge0$ satisfying the following property: for every reduced tree $T$ such that no edge stabilizer contains a non-abelian free group, we have that $T/G$ has at most $\gamma(G)$ vertices.
\end{prop}

\subsection{Universally elliptic splittings and JSJ decomposition}

\begin{lem}[Standard refinements]\label{lem:standard-refinement}
    Let $G$ be a group acting on two trees $T,T'$. Suppose that every edge stabilizer of $T$ is elliptic in $T'$. Then there is a refinement $q:\oc{T}\rar T$ such that $\oc{T}$ dominates $T'$. Moreover, it can be chosen in such a way that every edge stabilizer of $\oc{T}$ is contained in an edge stabilizer of either $T$ or $T'$.
\end{lem}
\begin{proof}
    This is Proposition 2.2 in \cite{GL17}.
\end{proof}

Let $\cA$ be a family of subgroups of $G$ which is stable under taking subgroups. An \textbf{$\cA$-splitting} is an action of $G$ on a tree $T$ such that all the edge stabilizers belong to $\cA$. Let $T$ be an $\cA$-splitting: we say that an edge stabilizer $G_e$, for $e\in E(T)$, is \textbf{$\cA$-universally elliptic} if it is elliptic in every $\cA$-splitting; we say that $T$ is \textbf{$\cA$-universally elliptic} if all its edge stabilizers are $\cA$-universally elliptic. Roughly speaking, this means that the splitting $T$ is ``canonical'': it is compatible with every other splitting $T'$, in the sense that we can always find a refinement $\oc{T}$ over $\cA$ that dominates $T'$ (by Lemma \ref{lem:standard-refinement}).

\begin{defn}[JSJ decomposition]
    A \textbf{JSJ decomposition} for $G$ over $\cA$ is an $\cA$-universally elliptic tree $T$ which is maximal by domination {\rm (}i.e. if $T'$ is $\cA$-universally elliptic and $T'$ dominates $T$, then $T$ dominates $T'${\rm)}.
\end{defn}

\begin{thm}[Existence of JSJ]\label{thm:JSJ-existence}
    Let $G$ be a finitely presented group. Let $\cA$ be a family of subgroups of $G$ which is stable under taking subgroups. Then $G$ has a JSJ decomposition over $\cA$. 
\end{thm}
\begin{proof}
    This is Theorem 2.16 in \cite{GL17}. 
\end{proof}

We point out that the proof of Theorem \ref{thm:JSJ-existence} is based on Dunwoody's accessibility (see \cite{Dun85}). Theorem \ref{thm:JSJ-existence} becomes false if we only require $G$ to be finitely generated (see \cite{Dun93}). It is immediate from the definition that, if a JSJ decomposition exists, then the set of all JSJ decompositions is a deformation space.

\begin{defn}[Rigid and flexible vertices]
    Let $J$ be a JSJ decomposition of $G$ over $\cA$. We say that a vertex $v\in V(J)$ is \textbf{rigid} if its stabilizer $G_v$ is $\cA$-universally elliptic, and \textbf{flexible} otherwise.
\end{defn}

If $J$ is a JSJ decomposition of $G$ over $\cA$, then a vertex $v\in V(J)$ is flexible if and only if its stabilizer $G_v$ has a non-trivial $\cA$-splitting relative to its induced peripheral structure. In that case, every non-trivial such splitting has a non-$\cA$-universally elliptic edge stabilizer.

\subsection{JSJ decomposition and maximal splittings}

In Proposition \ref{prop:algorithmic-JSJ}, we are going to algorithmically compute a JSJ decomposition for groups in a given family; in order to do this, we show that we can start with our favourite splitting, then refine it to a maximal splitting, and finally collapse some of its edges (to be precise: the ones which are not universally elliptic).

Let $G$ be a finitely presented group. Let $\cA$ be a family of subgroups of $G$ which is stable under taking subgroups. Denote with $\cA'=\{A'\sgr G : A'$ is contained with infinite index in some $A\in\cA\}$.

\begin{prop}[No proper refinement]\label{prop:get-maximal}
    Consider the following operations on an $\cA$-splitting $T$:
    \begin{enumerate}
        \item\label{itm:contract} If $T$ is not reduced, perform a contraction {\rm(}see {\rm Figure \ref{fig:contraction-expansion})}.
        \item\label{itm:refine} Substitute $T$ with a proper refinement $\oc{T}$ over $\cA$ which has exactly one orbit of edges more than $T$.
    \end{enumerate}
    Suppose that $\cA$ does not contain any non-abelian free group. Suppose that $G$ has no non-trivial $\cA'$-splitting. Then, for every $\cA$-splitting $T$, every sequence of moves {\rm\ref{itm:contract}} and {\rm\ref{itm:refine}} terminates, and the result is a reduced $\cA$-splitting $S$ which has no proper refinement over $\cA$.
\end{prop}

\begin{proof}
    Suppose that the inital graph of groups is reduced. Suppose by contradiction that there is an infinite sequence of moves \ref{itm:contract} and \ref{itm:refine}.

    The rank of the quotient graph $T/G$ is bounded by the number of generators for $G$; in particular, moves of type \ref{itm:refine} which refine a vertex group by an HNN extension can only happen boundedly many times along the process. Without loss of generality, we can assume that they never occur. Suppose that at some point we perform a move of type \ref{itm:contract} that collapses an edge followed by a move of type \ref{itm:refine} that splits a vertex as an amalgam. Then we can choose to perform the move of type \ref{itm:refine} first: this is granted by the fact that $G$ has no non-trivial splitting over $\cA'$, and by Lemma \ref{lem:elliptic-finite-index}. Thus, without loss of generality, we can assume that each move is of type \ref{itm:refine}, splitting some vertex as a non-trivial amalgam.
    
    Since two moves of type \ref{itm:refine} commute, we can assume that the sequence of moves is of the following form: the first move splits a vertex $v_0$ into two vertices $v_1,w_1$; the second move splits the vertex $v_1$ into two vertices $v_2,w_2$; and so on, the $i$-th move splitting the vertex $v_{i-1}$ into two vertices $v_i,w_i$. Let us also call $e_i$ the edge joining $v_i$ and $w_i$.

    Consider the distance of $v_n$ from $w_1$ using only the edges $e_1,\dots,e_n$. If this distance is bounded for $n\rar\infty$, then we can obtain reduced splittings for $G$ with arbitrarily many vertices, which contradicts Proposition \ref{prop:BestvinaFeighn-accessibility}. If this distance goes to infinity as $n\rar\infty$, then we find a subsequence $\ell_1<\ell_2<\ell_3<\dots$ such that, if we go far enough in the sequence of moves, then $e_{\ell_i}$ connects $w_{\ell_i}$ and $w_{\ell_{i+1}}$. Since every move gives a non-trivial amalgam, we must have $G_{e_{\ell_i}}\not=G_{w_{\ell_i}}$. Therefore we must have that $G_{e_{\ell_i}}=G_{w_{\ell_{i+1}}}$ for all but finitely many $i$, as otherwise, we get a contradiction by Proposition \ref{prop:BestvinaFeighn-accessibility}. This way, we obtain an infinite descending chain of edge stabilizers $\dots>G_{e_{\ell_i}}>G_{e_{\ell_{i+1}}}>\dots$. Finally, by Proposition \ref{prop:Dunwoody-accessibility} we find a (minimal) splitting $S$ that dominates all the splittings of the sequence, and in particular, some edge stabilizer of $S$ must contain infinitely many subgroups of the sequence $G_{e_{\ell_i}}$; such edge gives a non-trivial splitting of $G$ over a subgroup in $\cA'$, contradicting the assumptions.
\end{proof}

\begin{prop}[JSJ splitting]\label{prop:JSJ-from-maximal}
    Suppose that $G$ has no non-trivial $\cA'$-splitting. Let $T$ be a reduced $\cA$-splitting which has no proper refinement over $\cA$. Then there is a collapse map $q:T\rar J$ such that $J$ is a JSJ decomposition for $G$ over $\cA$.
\end{prop}
\begin{proof}
    The following argument is a straightforward generalization of \cite{Bar18}. By Theorem \ref{thm:JSJ-existence}, there is a JSJ decomposition $I$ of $G$ over $\cA$. We can assume that $I$ is non-trivial (otherwise the statement is obvious) and minimal (up to taking the minimal invariant sub-tree). We use Lemma \ref{lem:standard-refinement} to build a refinement $\oc{I}$ of $I$ that dominates $T$, and such that $\oc{I}$ is an $\cA$-splitting. Since $I$ is minimal, we can assume that $\oc{I}$ is minimal too (as the minimal invariant sub-tree of $\oc{I}$ must surject onto $I$).

    Since $T$ is minimal, the map from $\oc{I}$ to $T$ is surjective. In particular, every edge stabilizer $G_e$ for $e\in E(T)$ contains an edge stabilizer $G_{e'}$ for some $e'\in E(\oc{I})$. If $G_{e'}$ has infinite index in $G_e$, then by hypothesis $e'$ must induce the trivial splitting of $G$; but this contradicts the minimality of $\oc{I}$. Thus $G_{e'}$ has finite index in $G_e$, yielding that $G_e$ is elliptic in $\oc{I}$ (by Lemma \ref{lem:elliptic-finite-index}). By Lemma \ref{lem:standard-refinement}, there is a refinement $\oc{T}$ of $T$ which dominates $\oc{I}$, and such that $\oc{T}$ is an $\cA$-splitting. By hypothesis, the refinement $\oc{T}$ of $T$ is not a proper refinement, and thus $T$ dominates $\oc{T}$, and in particular, $T$ dominates $I$.
    
    If an edge stabilizer $G_e$, for $e\in E(T)$, is not universally elliptic, then it cannot be contained in any edge stabilizer of $I$, and thus every map $T\rar I$ must collapse $e$ to a point. Let $T\rar J$ be the map that collapses to a point all the edges except the universally elliptic ones. It follows that every map $T\rar I$ factors through a map $J\rar I$. But the edge stabilizers of $J$ are universally elliptic, and $I$ is a JSJ decomposition; thus $J$ must be a JSJ decomposition too. The statement follows.
\end{proof}

\subsection{Flexible vertex groups for cyclic JSJ}

The following Proposition \ref{prop:flexible-vertices} is taken from \cite{GL17} and gives a characterization of the flexible groups for the cyclic JSJ decomposition of a finitely presented torsion-free group.

\begin{prop}[Flexible vertices in cyclic JSJ, see \cite{GL17}]\label{prop:flexible-vertices}
    Let $G$ be a finitely presented torsion-free group, and let $\cA$ be the family given by the trivial subgroup and the infinite cyclic subgroups of $G$. Let $J$ be a JSJ decomposition for $G$ over $\cA$, and let $v\in V(J)$ be a flexible vertex. Then one of the following holds:
    \begin{enumerate}
        \item $G=G_v$ is isomorphic to the fundamental group of a {\rm(}possibly non-orientable{\rm)} closed surface.
        \item There is a compact connected {\rm(}possibly non-orientable{\rm)} surface with non-empty boundary $(\Sigma,\partial\Sigma)$, containing at least one essential simple closed curve {\rm(}i.e. not a sphere, disk, annulus, pair of pants{\rm)}, and an isomorphism $\pi_1(\Sigma)\rar G_v$. Every subgroup in the induced peripheral structure at $G_v$ is contained in $\pi_1(\gamma)$ for some $\gamma\in\partial\Sigma$ {\rm(}up to conjugation{\rm)}. For every $\gamma\in\partial\Sigma$ there is a subgroup in the induced peripheral structure that has finite index in $\pi_1(\gamma)$ {\rm(}up to conjugation{\rm)}.
    \end{enumerate}
\end{prop}
\begin{proof}
    This follows from Theorems 6.2 and 6.6 in \cite{GL17}, see also \cite{RS97}. We denote by $D_\infty=\pres{a,b}{a^2=b^2=1}=\pres{c,t}{tc=\ol{c}t,t^2=1}$ the infinite dihedral group, and by $K=\pres{a,b}{a^2=b^2}=\pres{c,t}{tc=\ol{c}t}$ the Klein bottle group. If a slender group has a non-trivial splitting over $1$, then it must be isomorphic to either $\bbZ$ or $D_\infty$. If a slender group has a non-trivial splitting over $\bbZ$, then it has to be isomorphic to either $\bbZ^2,K$ or $\bbZ\times D_\infty$. Notice that $D_\infty$ has torsion, so it cannot appear under our hypothesis.
\end{proof}

\begin{lem}[Theorem III.2.6 in \cite{MS84}, see also Proposition 5.4 in \cite{GL17}] 
	\label{lem:splitting-surface}
Let $\Sigma$ be a compact hyperbolic surface. Assume that $\pi_1(\Sigma)$ acts on a tree $T$ non-trivially, without inversions, minimally, with cyclic edge stabilizers, and with all boundary subgroups elliptic.

Then $T$ is equivariantly isomorphic to the Bass-Serre tree of the splitting dual to a family of disjoint essential simple closed geodesics of $\Sigma$.

In particular, $\pi_1(\Sigma)$ has a non-trivial splitting relative to the boundary subgroups if and only if $\Sigma$ contains an essential simple closed geodesic.
\end{lem}

\subsection{Algorithmic computation of the JSJ decomposition}

A \textbf{free splitting} of a group $G$ is an action of $G$ on a tree $T$ such that every edge stabilizer is trivial. A group $G$ has a non-trivial free splitting if and only if it can be decomposed as $G=G_1*G_2$ for some $G_1,G_2\not=1$. A group is called \textbf{freely indecomposable} if it has no non-trivial free splitting.

\begin{prop}[Grushko decomposition]\label{prop:Grushko}
    Let $G$ be a finitely generated group. Then we have the following:
    \begin{enumerate}
        \item There is a decomposition $G\cong G_1*\dots *G_r*F_s$ for some $r,s\ge0$, where $G_i$ are freely indecomposable not isomorphic to $1$ or $\bbZ$, and $F_s$ is a free group of rank $s$.
        \item Every such decomposition has the same $r,s$ and the same factors {\rm(}up to permutation and isomorphism{\rm)}.
    \end{enumerate}
    This is called the \textbf{Grushko decomposition} of $G$.
\end{prop}
\begin{proof}
   See \cite{Gru40}.
\end{proof}

The Grushko decomposition of a finitely generated group $G$ is essentially the JSJ decomposition of $G$ over the family of subgroups $\cA=\{1\}$. To be precise, if $r,s\ge0$ and $G_1,\dots,G_r$ are freely indecomposable not isomorphic to $1$ or $\bbZ$, then a JSJ decomposition for the group $G_1*\dots *G_r*F_s$ over $\cA=\{1\}$ is given by the graph of groups in Figure \ref{fig:Grushko}.

\begin{figure}[H]
	\centering

\includegraphics[width=0.35\textwidth]{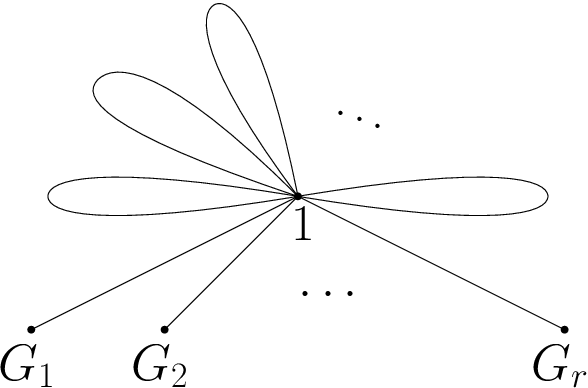}
    
	 \caption{The graph of groups associated with a Grushko decomposition. All the edge groups are trivial.}
	 \label{fig:Grushko}
\end{figure}

We are interested in algorithmically computing the Grushko decomposition of a given group; this can be done in a very general setting, see \cite{Tou18}. However, we prefer to provide a result adapted to our needs; Proposition \ref{prop:Diao-Feighn} below is one of the main theorems in \cite{DF05}.

\begin{defn}[1-algorithmic]\label{def:1algorithmic}
    A family $\fF$ of finitely presented torsion-free groups is called \textbf{$1$-algorithmic} if it satisfies the following conditions:
    \begin{enumerate}
        \item If $G\in\fF$ and $G=G_1*G_2$, then $G_1,G_2\in\fF$.
        \item\label{itm:detect-1splitting} There is an algorithm that, given $G\in\fF$ and $h_1,\dots,h_n\in G$, decides whether $G$ has a non-trivial free splitting relative to the peripheral structure $\cH=\{\gen{h_1},\dots,\gen{h_n}\}$ and, in case it does, computes such a splitting.
        \item\label{itm:detect-Z} There is an algorithm that, given $G\in\fF$, decides whether $G$ it is isomorphic to $1$ or $\bbZ$ and, in case it is, also computes an isomorphism.
    \end{enumerate}
\end{defn}

\begin{prop}[Diao-Feighn \cite{DF05}]\label{prop:Diao-Feighn}
    Let $\fF$ be a family of finitely presented torsion-free groups. Suppose that $\fF$ is $1$-algorithmic. Then we have the following:
    \begin{enumerate}
        \item If $\cG$ is a graph of groups with vertex groups in $\fF$ and infinite cyclic edge groups, then the Grushko decomposition of $\pi_1(\cG)$ is $\pi_1(\cG)=\pi_1(\cG_1)*\dots *\pi_1(\cG_r)*F_s$ for some graphs of groups $\cG_1,\dots,\cG_r$ with vertex groups in $\fF$ and infinite cyclic edge groups.
        \item There is an algorithm that, given $\cG$, computes the integers $r,s\ge0$ and the graphs of groups $\cG_1,\dots,\cG_r$.
    \end{enumerate}
\end{prop}
\begin{proof}
    By Theorem 7.2 in \cite{DF05}, we have that $\pi_1(\cG)$ has a non-trivial free splitting if and only if there is a sequence of I-, II- and III-simplifications, followed by a blow-up, such that the resulting graph of groups has an edge with trivial stabilizer. Since we are working with torsion-free groups, the simplifications and the blow-up must be over the trivial group. Since the edge groups of $\cG$ are isomorphic to $\bbZ$, a non-trivial simplification or blow-up can only happen if there is a vertex group that admits a non-trivial free splitting relative to its induced peripheral structure. This means that $\cG$ has a non-trivial free splitting if and only if there is a vertex group that admits a non-trivial free splitting relative to its induced peripheral structure. The statement follows from the assumption that $\fF$ is $1$-algorithmic.
\end{proof}

A \textbf{$\bbZ$-splitting} of a group $G$ is an action of $G$ on a tree $T$ such that every edge stabilizer is infinite cyclic. A group $G$ has a non-trivial $\bbZ$-splitting if and only if it has a non-trivial decomposition as an amalgamated free product $G=G_1*_\bbZ G_2$ or as an HNN extension $G=G_1*_\bbZ$. For the algorithmic computation of the cyclic JSJ decomposition, there is no result as general as for the Grushko decomposition (but see \cite{KM05,DG11,CM11,Cas16,Bar18,Tou18}). We will need the following result.

\begin{defn}[$\bbZ$-algorithmic]\label{def:Zalgorithmic}
    A family $\fF$ of finitely presented torsion-free groups is called \textbf{$\bbZ$-algorithmic} if it satisfies the following conditions:
    \begin{enumerate}
        \item $\fF$ is $1$-algorithmic, see {\rm Definition \ref{def:1algorithmic}}.
        \item If $G\in\fF$ and $G=G_1*_\bbZ G_2$ {\rm(}resp. $G=G_1*_\bbZ${\rm)}, then $G_1,G_2\in\fF$ {\rm(}resp. $G_1\in\fF${\rm)}.
        \item\label{itm:detect-Zsplitting} There is an algorithm that, given $G\in\fF$ and $h_1,\dots,h_n\in G$, decides whether $G$ has a non-trivial $\bbZ$-splitting relative to the peripheral structure $\cH=\{\gen{h_1},\dots,\gen{h_n}\}$ and, in case it does, computes such a splitting. 
        \item\label{itm:detect-surface} There is an algorithm that, given $G\in\fF$, decides whether $G$ is isomorphic to a free group and, in case it is, computes a basis for $G$.
        \item $\fF$ contains all finitely generated free groups.
    \end{enumerate}
\end{defn}

\begin{prop}[Algorithmic JSJ decomposition]\label{prop:algorithmic-JSJ}
    Let $\fF$ be a family of finitely presented torsion-free groups. Suppose that $\fF$ is $\bbZ$-algorithmic. Then there is an algorithm that, given a freely indecomposable graph of groups $\cG$ with vertex groups in $\fF$ and cyclic edge groups, does the following:
    \begin{enumerate}
        \item Decides whether $\pi_1(\cG)$ is isomorphic to the fundamental group of a closed surface and, in case it is, computes the surface.
        \item In case it is not, computes a graph of groups $\cG'$ with vertex groups in $\fF$ and cyclic edge groups, such that $\pi_1(\cG)\cong\pi_1(\cG')$ and $\cG'$ is a JSJ decomposition for $\pi_1(\cG')$ over the family of its cyclic subgroups.
    \end{enumerate}
\end{prop}
\begin{proof}
    By Proposition \ref{prop:get-maximal}, and using the fact that $\fF$ is $\bbZ$-algorithmic, we can algorithmically compute a reduced graph of groups $\cG'$, with vertex groups in $\fF$ and cyclic edge groups, such that $\cG'$ has no proper refinement over $\bbZ$. By Proposition \ref{prop:JSJ-from-maximal}, in order to obtain a JSJ, we only have to understand which edges of $\cG'$ have to be collapsed.
    
    Since $\fF$ is $\bbZ$-algorithmic, we can recognize which vertex groups are free groups, and for each of them, using Whitehead's algorithm \cite{Whi36}, we can determine whether they are surfaces with boundary given by their respective peripheral structure. Whenever we see two surfaces connected by an edge, and there is no other edge glued on those boundary components of the two surfaces, then we merge them together (as the splitting can not be canonical). By Proposition \ref{prop:flexible-vertices}, all flexible vertices vertices of a JSJ are surfaces; therefore, after this process, we obtain a JSJ decomposition for our initial graph of groups, as desired.
\end{proof}

\section{Generalized Baumslag-Solitar groups}\label{sec:GBSs}

In this section, we recall the definition of generalized Baumslag-Solitar group, which will be the main object of interest along the subsequent Sections \ref{sec:moves}, \ref{sec:controlled-linear-algebra}, \ref{sec:coarse-projection-map}, \ref{sec:sequences-of-moves}. We also present different ways of representing generalized Baumslag-Solitar groups. Firstly, as labeled graphs, which is standard in the literature. Secondly, the \textit{affine representation} (see Section \ref{sec:affine-representation}); this is equivalent to the representation as labeled graphs: it is its ``logarithmic version", which linearizes the operations on the exponents. The affine representation will be instrumental in explaining all the subsequent arguments as it makes several properties more tractable and is amenable for tools in linear algebra. Finally, we include some discussion about conjugacy classes of elliptic elements.

\begin{defn}[GBS group]
A \textbf{GBS graph of groups} is a finite graph of groups such that each vertex group and each edge group is $\bbZ$. A \textbf{Generalized Baumslag-Solitar group} is a group $G$ isomorphic to the fundamental group of some GBS graph of groups.
\end{defn}

\begin{thm}[Forester \cite{For03}]\label{thm:GBS-JSJ}
    Let $\cG$ be a GBS graph of groups and suppose that $\pi_1(\cG)\not\cong\bbZ,\bbZ^2,K$ {\rm(}the Klein bottle group{\rm)}. Then $\cG$ is a JSJ decomposition for $\pi_1(\cG)$ over the family of its cyclic subgroups.
\end{thm}
\begin{proof}
    See \cite{For03} (and also Section 3.5 in \cite{GL17}).
\end{proof}

\subsection{GBS graphs}

In this section, we recall the representation of generalized Baumslag-Solitar groups in terms of labelled graphs.

\begin{defn}[GBS graph]
A \textbf{GBS graph} is a pair $(\Gamma,\psi)$, where $\Gamma$ is a finite graph and $\psi:E(\Gamma)\rar\bbZ\setminus\{0\}$ is a function.
\end{defn}

Given a GBS graph of groups $\cG=(\Gamma,\{G_v\}_{v\in V(\Gamma)},\{G_e\}_{e\in E(\Gamma)},\{\psi_e\}_{e\in E(\Gamma)})$, the map $\psi_e:G_e\rar G_{\tau(e)}$ is an injective homomorphism $\psi_e:\bbZ\rar\bbZ$, and thus coincides with multiplication by a unique non-zero integer $\psi(e)\in\bbZ\setminus\{0\}$. We define the GBS graph associated to $\cG$ as $(\Gamma,\psi)$ associating to each edge $e$ the factor $\psi(e)$ characterizing the homomorphism $\psi_e$, see Figure \ref{fig:GBS-graph}. Giving a GBS graph of groups is equivalent to giving the corresponding GBS graph. In fact, the numbers on the edges are enough to reconstruct the injective homomorphisms and thus the graph of groups.

\begin{figure}[H]
\centering

\includegraphics[width=0.4\textwidth]{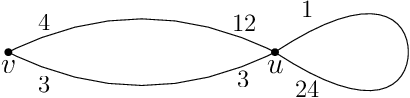}

\caption{In the figure we can see a GBS graph $(\Gamma,\psi)$ with two vertices $v,u$ and three edges $e_1,e_2,e_3$ (and their reverses). The edge $e_1$ goes from $v$ to $u$ and has $\psi(\ol{e}_1)=4$ and $\psi(e_1)=12$. The edge $e_2$ goes from $v$ to $u$ and has $\psi(\ol{e}_2)=\psi(\ol{e}_2)=3$. The edge $e_3$ goes from $u$ to $u$ and has $\psi(\ol{e}_3)=1$ and $\psi(e_3)=24$.}
\label{fig:GBS-graph}
\end{figure}

Let $\cG$ be a GBS graph of groups and let $(\Gamma,\psi)$ be the corresponding GBS graph. The universal group $\FG{\cG}$ has a presentation with generators $V(\Gamma)\cup E(\Gamma)$, the generator $v\in V(\Gamma)$ representing the element $1$ in $\bbZ=G_v$. The relations are given by $\ol{e}=e^{-1}$ and $u^{\psi(\ol{e})}e=ev^{\psi(e)}$ for every edge $e\in E(\Gamma)$ with $\iota(e)=u$ and $\tau(e)=v$.

\subsection{Affine representation of a GBS graph}\label{sec:affine-representation}

In this section, we introduce the affine representation of a generalized Baumslag-Solitar group. This representation will be essential in our study of generalized Baumslag-Solitar groups since it is amenable to methods and tools from linear algebra to both formulate and solve problems, see, for instance, Section \ref{sec:controlled-linear-algebra}.

Consider the abelian group
\begin{equation}\label{defn:group_A}
\bA:=\bbZ/2\bbZ\oplus\bigoplus\limits_{i\ge 1}\bbZ
\end{equation}
whose elements are sequences $\ba=(a_0,a_1,a_2,\dots)$ with $a_0\in\bbZ/2\bbZ$ and $a_i\in\bbZ$ for $i\ge 1$. We denote with $\mathbf{0}\in\bA$ the neutral element $\mathbf{0}=(0,0,0,\dots)$. For an element $\ba=(a_0,a_1,\dots)\in\bA$, we denote $\ba\ge\mathbf{0}$ if $a_i\ge 0$ for all $i\ge1$; notice that we are not requiring any condition on $a_0$. We define the positive cone $\pA:=\{\ba\in\bA : \ba\ge\mathbf{0}\}$.

\begin{defn}[Affine representation]
Let $(\Gamma,\psi)$ be a GBS graph. Define its \textbf{affine representation} to be the graph $\Lambda=\Lambda(\Gamma,\psi)$ given by:
\begin{enumerate}
\item $V(\Lambda)=V(\Gamma)\times\pA$ is the disjoint union of copies of $\pA$, one for each vertex of $\Gamma$.
\item $E(\Lambda)=E(\Gamma)$ is the same set of edges as $\Gamma$, and with the same reverse map.
\item For an edge $e\in E(\Lambda)=E(\Gamma)$, we write the unique prime factorization $\psi(e)=(-1)^{a_0}2^{a_1}3^{a_2}5^{a_3}\dots $ and we define the terminal vertex $\tau_\Lambda(e)=(\tau_\Gamma(e),(a_0,a_1,a_2,\dots))$, see {\rm Figure \ref{fig:aff-rep}}.
\end{enumerate}
For a vertex $v\in V(\Gamma)$ we denote $\pA_v:=\{v\}\times\pA$ the corresponding copy of $\pA$.
\end{defn}

If $\Lambda$ contains an edge going from $p$ to $q$, then we denote $p\edge q$. If $\Lambda$ contains edges from $p_i$ to $q_i$ for $i=1,\dots,m$, then we denote
$$
\begin{cases}
p_1\edge q_1\\
\dots \\
p_m\edge q_m.
\end{cases}
$$
We remark that the above notation does not mean that $p_1\edge q_1,\dots,p_m\edge q_m$ is the set of all the edges of $\Lambda$, but rather that we focus on that subset of edges. If we are focusing on a specific copy $\pA_v$ of $\pA$, for some $v\in V(\Gamma)$, and we have edges $(v,\ba_i)\edge (v,\bb_i)$ for $i=1,\dots,m$, then we say that $\pA_v$ contains edges
$$
\begin{cases}
\ba_1\edge \bb_1\\
\dots \\
\ba_m\edge \bb_m
\end{cases}
$$
omitting the vertex $v$.

\begin{remark}
Components of the abelian group $\bA$ correspond to prime numbers. Since we will work with finite sequences of GBS graphs, at each time we will only have a finite number of prime numbers involved in the factorizations of the numbers on the edges. This means that at each time we will have only a finite number of relevant components in $\bA$. We thus strongly encourage the reader to think of $\bA$ and $\pA$ as finite-dimensional (as if we had only finitely many $\bbZ$ summands). In the figures we will usually only draw the relevant components.
\end{remark}

\begin{figure}[H]
\centering

\includegraphics[width=0.85\textwidth]{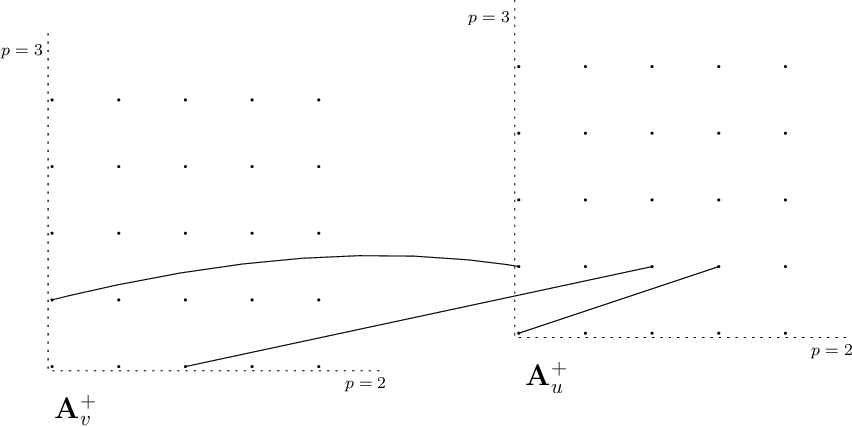}

\caption{The affine representation $\Lambda$ of the GBS graph $(\Gamma,\psi)$ of Figure \ref{fig:GBS-graph}. The set of vertices consists of two copies $\pA_v$ and $\pA_u$ of the positive affine cone $\pA$, associated with the two vertices $v$ and $u$ respectively. The edge $e_1$ going from $v$ to $u$ was labeled with $(\psi(\ol{e}_1),\psi(e_1))=(4,12)=(2^2 3^0,2^2 3^1)$, and thus now it goes from the point $(2,0)$ in $\pA_v$ to the point $(2,1)$ in $\pA_u$. Similarly for $e_2$ and $e_3$.}
\label{fig:aff-rep}
\end{figure}

\subsection{Conjugacy classes}

In this section, we describe the conjugacy problem for elliptic elements in a generalized Baumslag-Solitar group as an equivalence relation on the corresponding vertices in the affine representation. We then recall that the conjugacy problem for elliptic elements is decidable.

Let $(\Gamma,\psi)$ be a GBS graph and let $\Lambda$ be its affine representation. Given a vertex $p=(v,\ba)\in V(\Lambda)$ and an element $\bw\in\pA$, we define the vertex $p+\bw:=(v,\ba+\bw)\in V(\Lambda)$. For two vertices $p,p'\in V(\Lambda)$ we denote $p'\ge p$ if $p'=p+\bw$ for some $\bw\in\pA$; in particular this implies that both $p,p'$ belong to the same $\pA_v$ for some $v\in V(\Gamma)$. 

\begin{defn}[Affine path]\label{defn:affine path}
An \textbf{affine path} in $\Lambda$, with \textit{initial vertex} $p\in V(\Lambda)$ and \textit{terminal vertex} $p'\in V(\Lambda)$, is a sequence $(e_1,\dots,e_\ell)$ of edges $e_1,\dots,e_\ell\in E(\Lambda)$ for some $\ell\ge0$, such that there exist $\bw_1,\dots,\bw_\ell\in\pA$ satisfying the conditions $\iota(e_1)+\bw_1=p$ and $\tau(e_\ell)+\bw_\ell=p'$ and $\tau(e_i)+\bw_i=\iota(e_{i+1})+\bw_{i+1}$ for $i=1,\dots,\ell-1$. 
\end{defn}

\begin{figure}[H]
\centering

\includegraphics[width=0.85\textwidth]{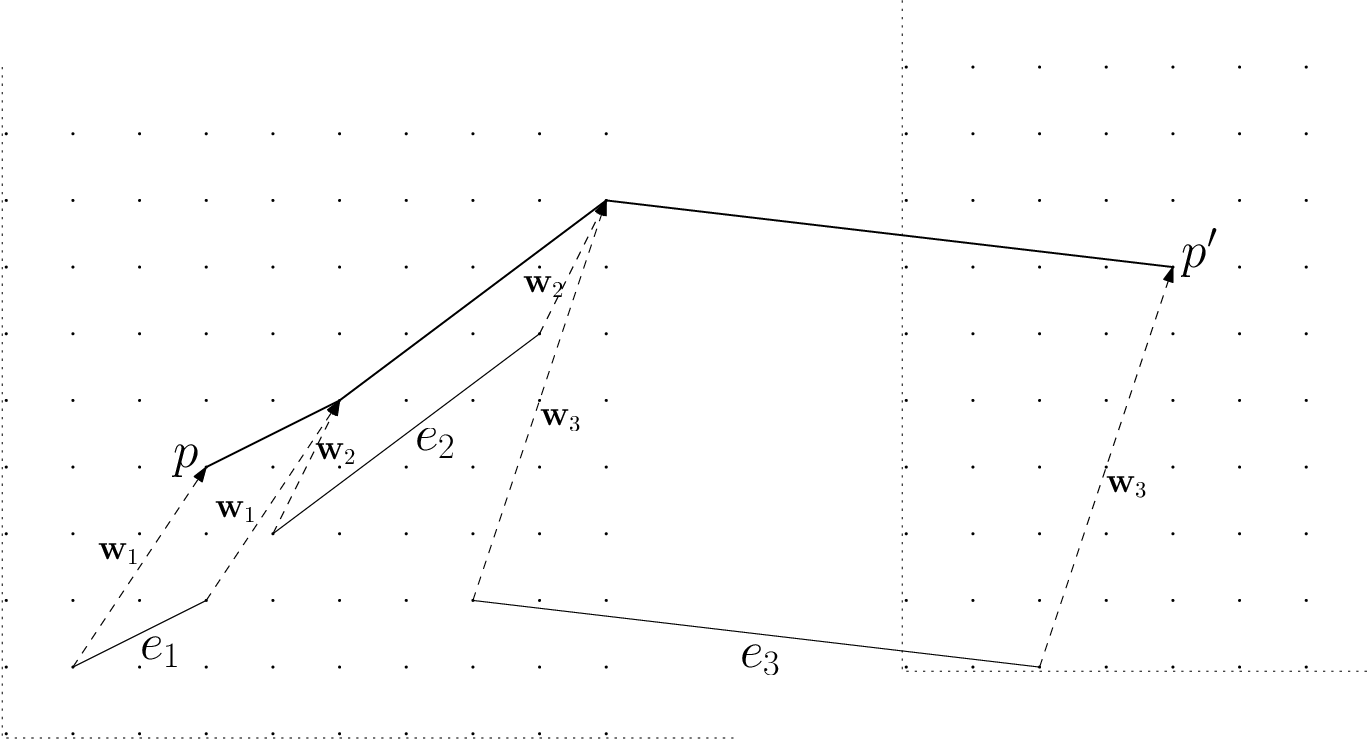}

\caption{An example of an affine path.}
\label{fig:aff-path}
\end{figure}

The elements $\bw_1,\dots,\bw_\ell$ are called \textit{translation coefficients} of the path; if they exist, then they are uniquely determined by the path and the endpoints, and they can be computed algorithmically. They represent the fact that an edge $e\in E(\Lambda)$ connecting $p$ to $q$ also allows us to travel from $p+\bw$ to $q+\bw$ for every $\bw\in\pA$. If $\cG$ is a GBS graph of groups and $(\Gamma,\psi)$ is the corresponding GBS graph, then to the vertex $p=(v,(a_0,a_1,a_2,\dots))$ of the affine representation we associate the element $v^{(-1)^{a_0}2^{a_1}3^{a_2}\dots }$ of the universal group $\FG{\cG}$. An affine path $(e_1,\dots,e_\ell)$ from $p=(v,(a_0,a_1,\dots))$ to $p'=(v',(a_0',a_1',\dots))$ reflects the identity
$$v^{(-1)^{a_0}2^{a_1}3^{a_2}\dots }e_1e_2\dots e_\ell=e_1e_2\dots e_\ell(v')^{(-1)^{a_0'}2^{a_1'}3^{a_2'}\dots }$$
in the universal group $\FG{\cG}$.

\begin{defn}[Conjugate points]\label{def:conjugacy}
Given two vertices $p,p'\in V(\Lambda)$ we say that $p$ is \textbf{conjugate} to $p'$ {\rm(}denoted $p\cnj p'${\rm)} if there is an affine path going from $p$ to $p'$, see {\rm Definition \ref{defn:affine path}}.
\end{defn}

Definition \ref{def:conjugacy} induces an equivalence relation on the set $V(\Lambda)$: we denote by $\cnjclass{p}$ the conjugacy class of $p\in V(\Lambda)$. If $\cG$ is a GBS graph of groups and $(\Gamma,\psi)$ is the corresponding GBS graph, then to the vertex $p=(v,(a_0,a_1,a_2,\dots))$ of the affine representation we associate the conjugacy class $[v^{(-1)^{a_0}2^{a_1}3^{a_2}\dots }]$ in the fundamental group $\pi_1(\cG)$: we have $p\cnj p'$ if and only if the corresponding conjugacy classes coincide. In other words, the conjugacy problem in the group is equivalent to the reachability by affine paths in the affine representation.

The conjugacy problem in generalized Baumslag-Solitar groups was proven to be decidable in \cite{L92, B15}, see also \cite{W16} for complexity bounds. In \cite{ACK-iso2}, 
%%[Corollary 2.29, Remark 2.30]
we give a self-contained proof of the decidability of the conjugacy problem. For completeness, we next explain how to deduce the result using the fact that conjugacy classes are semilinear sets.

Notice that if $p\cnj p'$ then $p+\bw\cnj p'+\bw$ for all $\bw\in\pA$, that is $\cnj$ is a congruence on $\pA$ and so $\pA\slash \cnj$ is a commutative monoid. Furthermore, deciding whether two elements are conjugate is equivalent to solving the word problem in this monoid. Mal'cev \cite{Mal58} and Emelichev \cite{Eme62} showed that the word problem for finitely generated commutative monoids is decidable – even if the congruence is part of the input. In \cite[Theorem II]{ES69} Eilenberg and Schützenberger showed that every congruence on $\mathbb N^m$ is a semilinear subset of $\mathbb N^m \times \mathbb N^m$ (this follows also from the results that congruences are definable by Presburger formulas, see \cite{Tai68}, and that Presburger definable sets are semilinear, see \cite{GS66} – for definition of all these notions we refer to the respective papers). Therefore, the conjugacy class of a point can be computed algorithmically. For the record, we state these results below.

\begin{prop}[Conjugacy classes are semilinear sets]\label{compute-conjugacy-class}
Let $(\Gamma,\psi)$ be a GBS graph and let $p$ be a vertex of its affine representation $\Lambda$. Then there are vertices $p_1,\dots,p_n\in V(\Lambda)$ and elements $\bw_1,\dots,\bw_m\in\pA$ such that, for a vertex $q$ of $\Lambda$, the following are equivalent:
\begin{enumerate}
\item $q\cnj p$
\item $q=p_i+k_1\bw_1+\dots +k_m\bw_m$ for some $i\in\{1,\dots,n\}$ and for some integers $k_1,\dots,k_m\ge0$.
\end{enumerate}
Moreover there is an algorithm that, given $(\Gamma,\psi)$ and the vertex $p$, computes a set of vertices $p_1,\dots,p_n$ and elements $\bw_1,\dots,\bw_m$.
\end{prop}
\begin{proof}
 The relation $\cnj$ is a congruence and so by \cite[Theorem II]{ES69}, we have that the equivalence class is a computable semilinear set. The result follows.
\end{proof}

\begin{lem}[Decidability of conjugacy]\label{compare-conjugacy-class}
There is an algorithm that, given a GBS graph $(\Gamma,\psi)$ and two vertices $p,q$ of its affine representation $\Lambda$, determines whether or not $p\cnj q$, and in case of affirmative answer also outputs an affine path from $p$ to $q$.
\end{lem}
\begin{proof}
To determine whether or not $p\cnj q$ is equivalent to deciding whether or not $q$ belongs the conjugacy class of $p$. By Proposition \ref{compute-conjugacy-class}, this in turn is equivalent to the decidability of the membership problem for a semilinear set. The latter problem is decidable, since semilinear sets are definable Presburguer, see \cite{GS66} and \cite[Theorem 1]{IJCR91}, hence so is the conjugacy problem. If the vertices are conjugate, one can use brute force enumeration to determine an affine path (there are bounds on the length of an affine path and other more efficient methods to describe it, see for instance \cite%[Corollary 2.29]
{ACK-iso2}).
\end{proof}

\section{Moves on GBS graphs}\label{sec:moves}

In this section, we define several moves that, given a GBS graph, produce another GBS graph with isomorphic GBS group. The sign-change, elementary expansion, elementary contraction, slide, and induction moves are standard knowledge in the literature (see \cite{For02,For06,CF08}). On the other hand, the swap move and the connection move are new.

\subsection{Sign-change move}

Let $(\Gamma,\psi)$ be a GBS graph. Let $v\in V(\Gamma)$ be a vertex. Define the map $\psi':E(\Gamma)\rar\bbZ\setminus\{0\}$ such that $\psi'(e)=-\psi(e)$ if $\tau(e)=v$ and $\psi'(e)=\psi(e)$ otherwise. We say that the GBS graph $(\Gamma,\psi')$ is obtained from $(\Gamma,\psi)$ by a \textbf{vertex sign-change}. If $\cG$ is the GBS graph of groups associated to $(\Gamma,\psi)$, then the vertex sign change move corresponds to changing the chosen generator for the vertex group $G_v$; in particular it induces an isomorphism at the level of the universal group $\FG{\cG}$ and of the fundamental group $\pi_1(\cG)$. In terms of the affine representation, the vertex sign-change corresponds to the map $\pA_v \to \pA_v$ that sends $(a_0, a_1, \dots) \mapsto (a_0 + 1, a_1, \dots)$, i.e. the move permutes the hyperplanes $(0, a_1, \dots)$ and $(1, a_1, \dots)$.

Let $(\Gamma,\psi)$ be a GBS graph. Let $d\in E(\Gamma)$ be an edge. Define the map $\psi':E(\Gamma)\rar\bbZ\setminus\{0\}$ such that $\psi'(e)=-\psi(e)$ if $e=d,\ol{d}$ and $\psi'(e)=\psi(e)$ otherwise. We say that the GBS graph $(\Gamma,\psi')$ is obtained from $(\Gamma,\psi)$ by an \textbf{edge sign-change}. If $\cG$ is the GBS graph of groups associated to $(\Gamma,\psi)$, then the vertex sign change move corresponds to changing the chosen generator for the edge group $G_d$; in particular it induces an isomorphism at the level of the universal group $\FG{\cG}$ and of the fundamental group $\pi_1(\cG)$. In the affine representation $\Lambda$ the two endpoints $\iota_\Lambda(d)=(\iota_\Gamma(d),(a_0,a_1,a_2,\dots))$ and $\tau_\Lambda(d)=(\tau_\Gamma(d),(b_0,b_1,b_2,\dots))$ are modified by adding $1$ to the $\bbZ/2\bbZ$ components $a_0$ and $b_0$.

\subsection{Elementary expansion and elementary contraction moves}

Let $(\Gamma,\psi)$ be a GBS graph. Let $v$ be a vertex, let $k$ be a non-zero integer. Define the graph $\Gamma'$ by adding a new vertex $v_0$ and a new edge $e_0$ going from $v_0$ to $v$ (and its reverse $\ol{e}_0$); we set $\psi(\ol{e}_0)=1$ and $\psi(e_0)=k$. We say that the GBS graph $(\Gamma',\psi)$ is obtained from $(\Gamma,\psi)$ by an \textbf{elementary expansion} (see Figure \ref{fig:elem-exp}).

Let $\cG$ be the GBS graph of groups associated with $(\Gamma,\psi)$. At the level of the universal group $\FG{\cG}$ we are adding generators $e_0,v_0$ and the relation $v_0e_0=e_0v^k$. In particular, an elementary expansion induces an isomorphism at the level of the fundamental group $\pi_1(\cG)$.

At the level of the affine representation, we add a new copy $\pA_{v_0}$ of $\pA$ corresponding to the new vertex, together with an edge connecting the origin $(v_0,\mathbf{0})\in\pA_{v_0}$ to another vertex in another copy of $\pA$ (see Figure \ref{fig:elem-exp}).

Let $(\Gamma,\psi)$ be a GBS graph. Let $v_0$ be a vertex, and suppose there is a unique edge $e_0$ with $\iota(e_0)=v_0$; suppose also that $\tau(e_0)\not=v_0$ and $\psi(\ol{e}_0)=1$. Define the graph $\Gamma'$ by removing the vertex $v_0$ and the edge $e_0$ (and its reverse $\ol{e}_0$). We say that the GBS graph $(\Gamma',\psi)$ is obtained from $(\Gamma,\psi)$ by an \textbf{elementary contraction}. The elementary contraction move is the inverse of the elementary expansion move.

At the level of the affine representation, we will see a copy $\pA_{v_0}$ of $\pA$ containing only one endpoint of only one edge $e_0$, and in such a way that $\tau_\Lambda(e_0)=(v_0,\mathbf{0})$. The contraction move has the effect of removing all of $\pA_{v_0}$ and the edge $e_0$ from $\Lambda$.

\begin{figure}[H]
\centering

\includegraphics[width=0.7\textwidth]{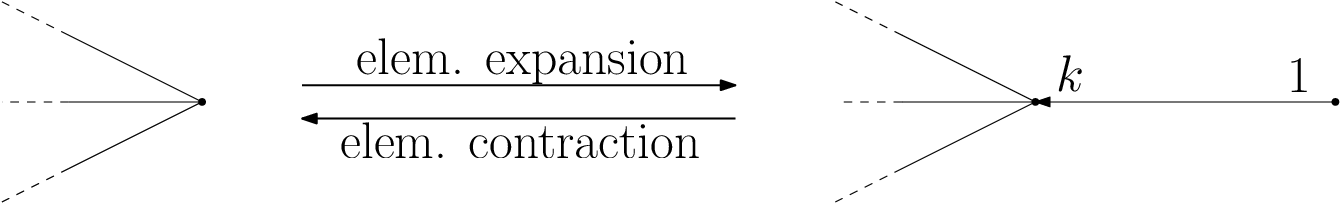}

\vspace{1cm}

\includegraphics[width=\textwidth]{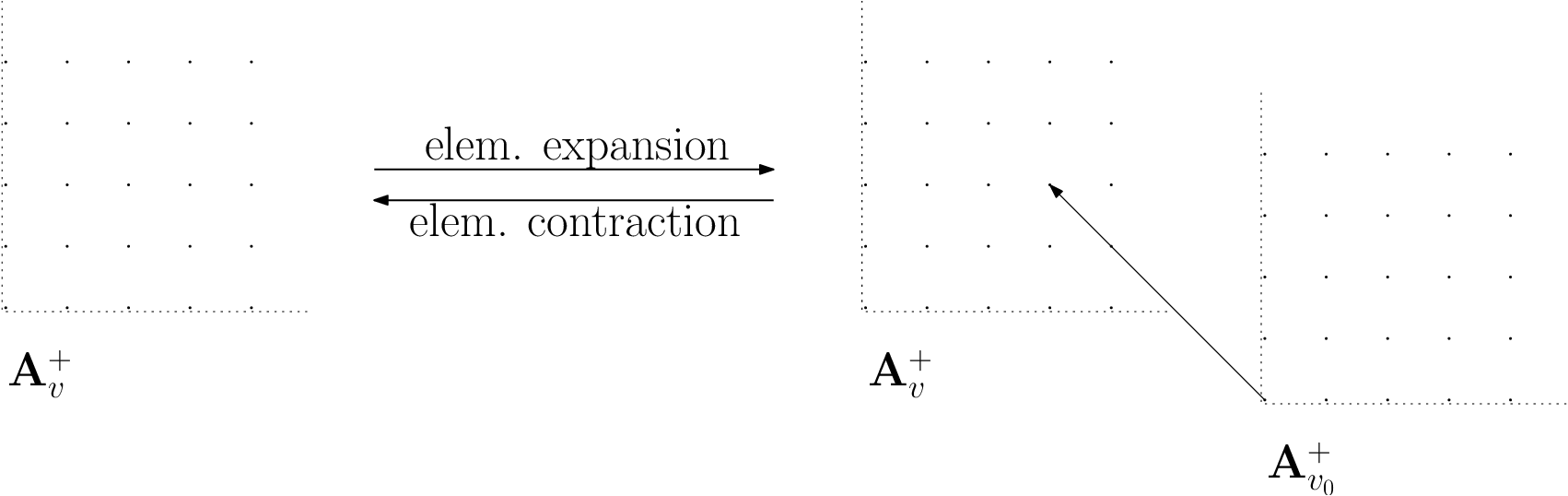}

\caption{An example of an elementary expansion and contraction. Above you can see the GBS graphs. Below you can see the corresponding affine representations.}
\label{fig:elem-exp}
\end{figure}

\subsection{Slide move}\label{subsec:slide}

Let $(\Gamma,\psi)$ be a GBS graph. Let $d,e$ be distinct edges with $\tau(d)=\iota(e)=u$ and $\tau(e)=v$; suppose that $\psi(\ol{e})=n$ and $\psi(e)=m$ and $\psi(d)=\ell n$ for some $n,m,\ell\in\bbZ\setminus\{0\}$ (see Figure \ref{fig:slide}). Define the graph $\Gamma'$ by replacing the edge $d$ with an edge $d'$; we set $\iota(d')=\iota(d)$ and $\tau(d')=v$; we set $\psi(\ol{d}')=\psi(\ol{d})$ and $\psi(d')=\ell m$. We say that the GBS graph $(\Gamma',\psi)$ is obtained from $(\Gamma,\psi)$ by a \textbf{slide}.

Let $\cG,\cG'$ be the GBS graph of groups associated to $(\Gamma,\psi),(\Gamma',\psi)$ respectively. At the level of the universal groups $\FG{\cG},\FG{\cG'}$ we have an isomorphism defined by
$$\begin{cases}
d'=de
\end{cases}
\qquad\text{ with inverse }\qquad
\begin{cases}
d=d'\ol{e}
\end{cases}$$
In particular, a slide move induces an isomorphism at the level of the fundamental group $\pi_1(\cG)$.

At the level of the affine representation, we have an edge $p\edge  q$ and we have another edge with an endpoint at $p+\ba$ for some $\ba\in\pA$. The slide has the effect of moving the endpoint from $p+\ba$ to $q+\ba$ (see Figure \ref{fig:slide}).
$$\begin{cases}
p\edge  q\\
r\edge  p+\ba
\end{cases}
\xrightarrow{\text{slide}}\quad
\begin{cases}
p\edge  q\\
r\edge  q+\ba
\end{cases}$$
\begin{remark}
In the definition of slide move, we also allow for some of the vertices $u,v,\iota(d)$ to coincide.
\end{remark}

\begin{figure}[H]
\centering

\includegraphics[width=0.75\textwidth]{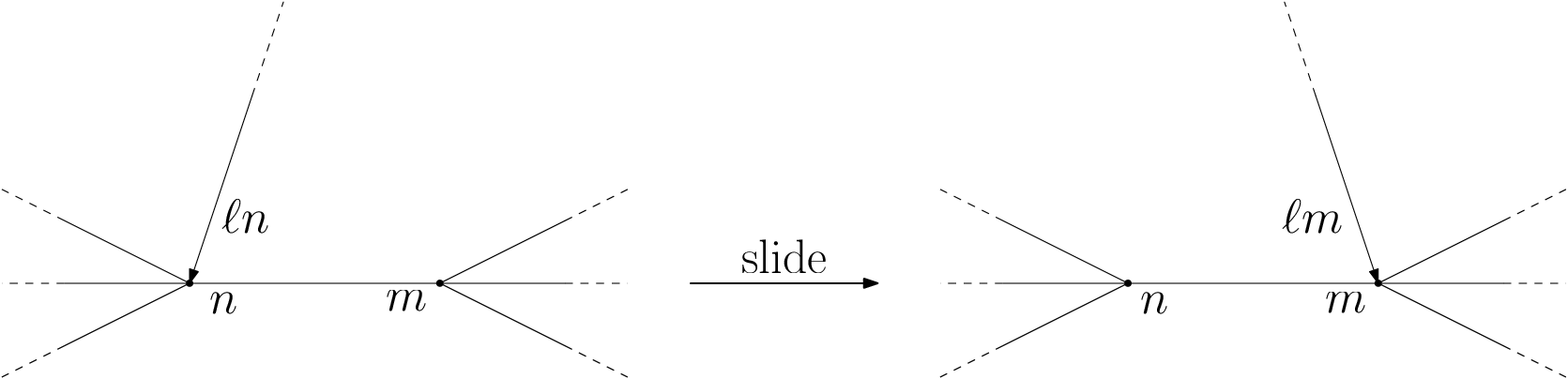}

\vspace{1cm}

\includegraphics[width=\textwidth]{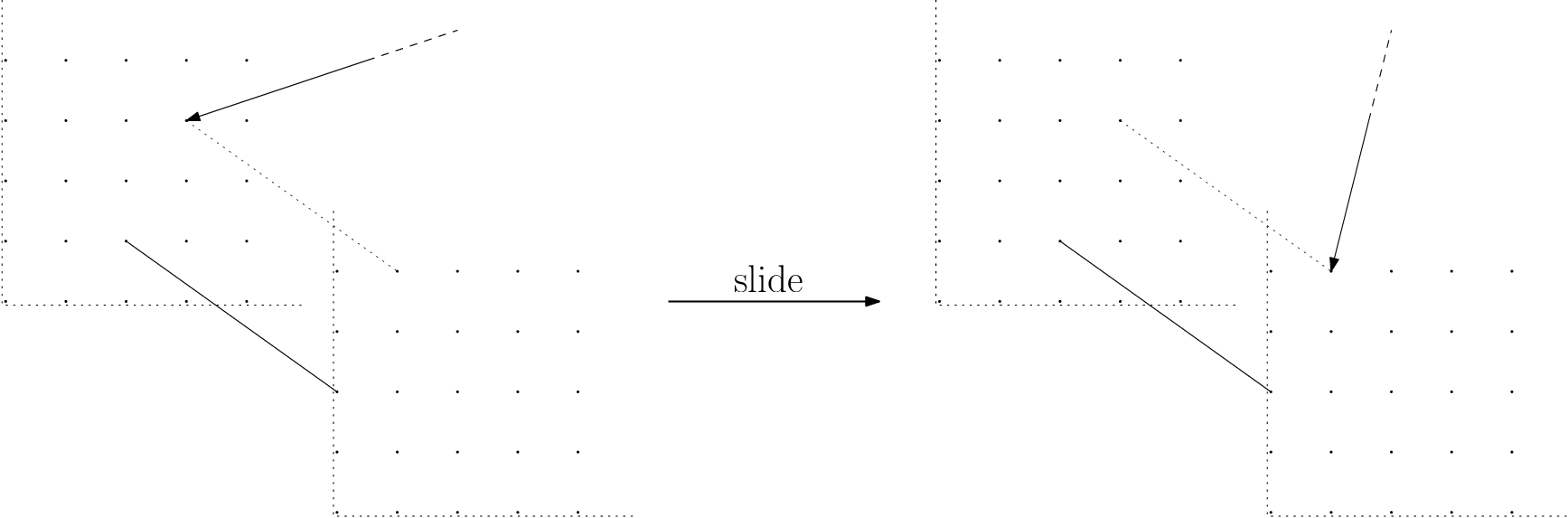}

\caption{An example of a slide move. Above you can see the GBS graphs. Below you can see the corresponding affine representations.}
\label{fig:slide}
\end{figure}

\subsection{Induction move}

Let $(\Gamma,\psi)$ be a GBS graph. Let $e$ be an edge with $\iota(e)=\tau(e)=v$; suppose that $\psi(\ol{e})=1$ and $\psi(e)=n$ for some $n\in\bbZ\setminus\{0\}$, and choose $\ell\in\bbZ\setminus\{0\}$ and $k\in\bbN$ such that $\ell\divides n^k$. Define the map $\psi'$ equal to $\psi$ except on the edges $d\not=e,\ol{e}$ with $\tau(d)=v$, where we set $\psi'(d)=\ell\cdot\psi(d)$. We say that the GBS graph $(\Gamma,\psi')$ is obtained from $(\Gamma,\psi)$ by an \textbf{induction}.

Let $\cG,\cG'$ be the GBS graph of groups associated to $(\Gamma,\psi),(\Gamma,\psi')$ respectively. At the level of the universal groups $\FG{\cG},\FG{\cG'}$ we have an isomorphism defined by
$$\begin{cases}
v'=e^kv^{(n^k/\ell)}\ol{e}^k
\end{cases}
\qquad\text{ with inverse }\qquad
\begin{cases}
v=(v')^\ell
\end{cases}$$
In particular, an induction move induces an isomorphism at the level of the fundamental group $\pi_1(\cG)$.

At the level of the affine representation, we have an edge $(v,\mathbf{0})\edge(v,\bw)$ where $\bw\in\pA$. We choose $\bw_1\in\pA$ such that $\bw_1\le k\bw$ for some $k\in\bbN$; we take all the endpoints of other edges lying in $\pA_v$, and we translate them up by adding $\bw_1$ (see Figure \ref{fig:ind}). Note also that the inverse of an induction move can be obtained by composing an induction move and slide moves.

\begin{figure}[H]
\centering

\includegraphics[width=0.65\textwidth]{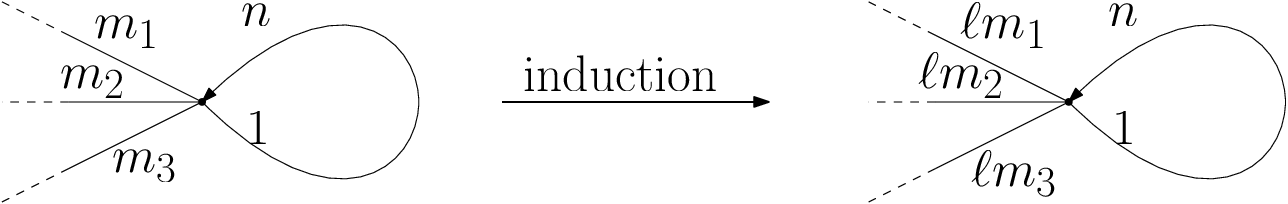}

\vspace{0.5cm}

\includegraphics[width=0.8\textwidth]{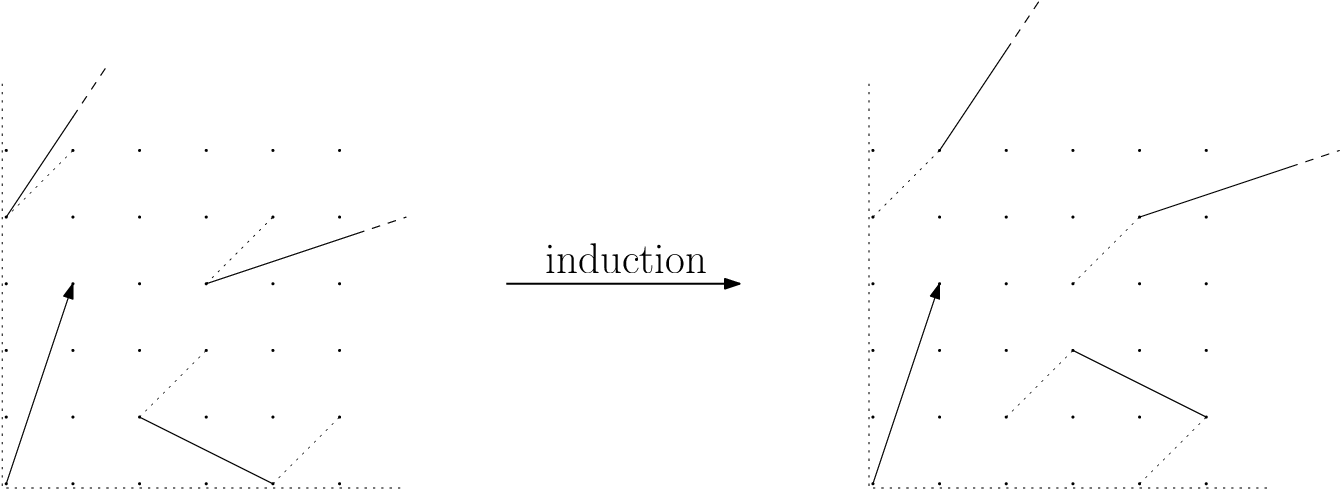}

\caption{An example of an induction move. Above you can see the GBS graphs; here $\ell\divides n^k$ for some integer $k\ge0$. Below you can see the corresponding affine representations.}
\label{fig:ind}
\end{figure}

\subsection{Swap move}

Let $(\Gamma,\psi)$ be a GBS graph. Let $e_1,e_2$ be distinct edges with $\iota(e_1)=\tau(e_1)=\iota(e_2)=\tau(e_2)=v$; suppose that $\psi(\ol{e}_1)=n$ and $\psi(e_1)=\ell_1 n$ and $\psi(\ol{e}_2)=m$ and $\psi(e_2)=\ell_2 m$ and $n\divides m$ and $m\divides \ell_1^{k_1} n$ and $m\divides \ell_2^{k_2} n$ for some $n,m,\ell_1,\ell_2\in\bbZ\setminus\{0\}$ and $k_1,k_2\in\bbN$ (see Figure \ref{fig:swap}). Define the graph $\Gamma'$ by substituting the edges $e_1,e_2$ with two edges $e_1',e_2'$; we set $\iota(e_1')=\tau(e_1')=\iota(e_2')=\tau(e_2')=v$; we set $\psi(\ol{e_1'})=m$ and $\psi(e_1')=\ell_1 m$ and $\psi(\ol{e_2'})=n$ and  $\psi(e_2')=\ell_2 n$. We say that the GBS graph $(\Gamma',\psi)$ is obtained from $(\Gamma,\psi)$ by a \textbf{swap move}.

Let $\cG,\cG'$ be the GBS graph of groups associated to $(\Gamma,\psi),(\Gamma',\psi)$ respectively. At the level of the universal groups $\FG{\cG},\FG{\cG'}$ we have an isomorphism defined by
$$\begin{cases}
e_1'=e_1^{k_1}\ol{e}_2^{k_2}e_1e_2^{k_2}\ol{e}_1^{k_1}\\
e_2'=e_1^{k_1}e_2\ol{e}_1^{k_1}
\end{cases}
\qquad\text{ with inverse }\qquad
\begin{cases}
e_1=(e_2')^{k_2}e_1'(\ol{e_2'})^{k_2}\\
e_2=(e_2')^{k_2}(\ol{e}_1')^{k_1}e_2'(e_1')^{k_1}(\ol{e}_2')^{k_2}
\end{cases}$$
In particular, a swap move induces an isomorphism at the level of the fundamental group $\pi_1(\cG)$.

Let $\bw_1,\bw_2\in\pA$ and $p,q\in V(\Lambda)$ be such that $p\le q\le p+k_1\bw_1$ and $p\le q\le p+k_2\bw_2$ for some $k_1,k_2\in\bbN$. At the level of the affine representation, we have an edge $e_1$ going from $p$ to $p+\bw_1$ and an edge $e_2$ going from $q$ to $q+\bw_2$. The swap has the effect of substituting them with $e_1'$ from $q$ to $q+\bw_1$ and with $e_2'$ from $p$ to $p+\bw_2$ (see Figure \ref{fig:swap}),
$$\begin{cases}
p\edge  p+\bw_1\\
q\edge  q+\bw_2
\end{cases}
\xrightarrow{\text{swap}}\quad
\begin{cases}
q\edge  q+\bw_1\\
p\edge  p+\bw_2
\end{cases}$$

\begin{figure}[H]
\centering

\includegraphics[width=0.65\textwidth]{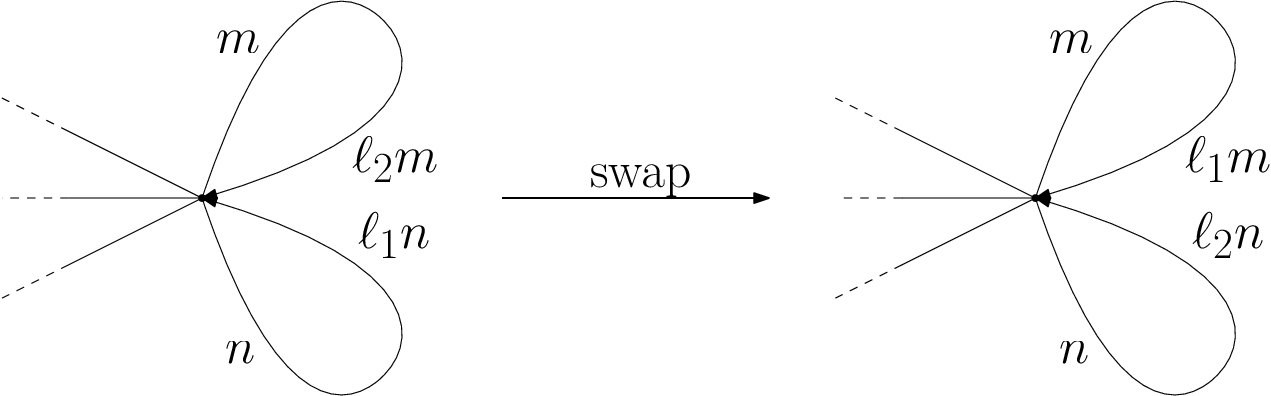}

\vspace{0.5cm}

\includegraphics[width=0.8\textwidth]{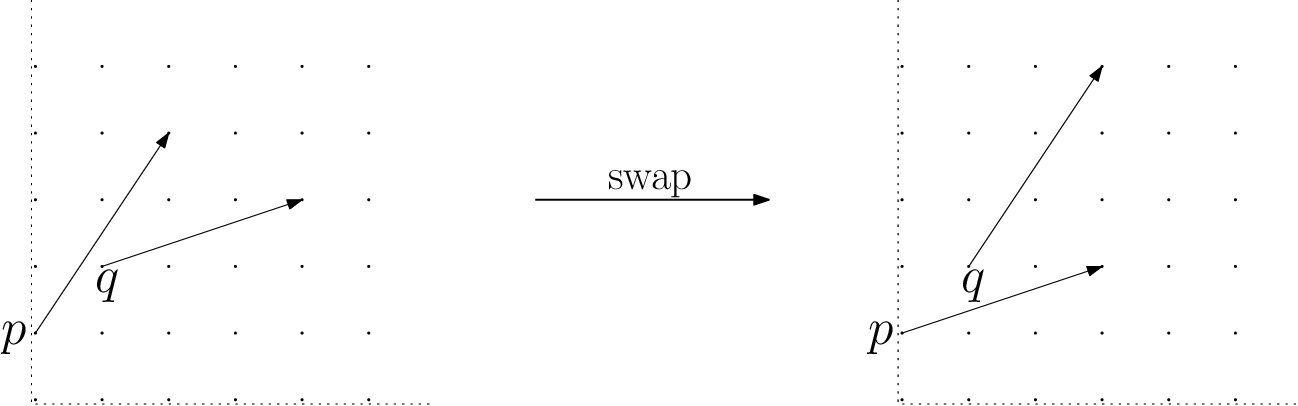}

\caption{An example of a swap move. Above you can see the GBS graphs; here $n\divides m\divides\ell_1^{k_1}n$ and $n\divides m\divides\ell_2^{k_2}n$ for some integers $k_1,k_2\ge0$. Below you can see the corresponding affine representations.}
\label{fig:swap}
\end{figure}

\subsection{Connection move}

Let $(\Gamma,\psi)$ be a GBS graph. Let $d,e$ be distinct edges with $\iota(d)=u$ and $\tau(d)=\iota(e)=\tau(e)=v$; suppose that  $\psi(\ol{d})=m$ and $\psi(d)=\ell_1 n$ and $\psi(\ol{e})=n$ and $\psi(e)=\ell n$ and $\ell_1\ell_2=\ell^k$ for some $m,n,\ell_1,\ell_2,\ell\in\bbZ\setminus\{0\}$ and $k\in\bbN$ (see Figure \ref{fig:connection}). Define the graph $\Gamma'$ by substituting the edges $d,e$ with two edges $d',e'$; we set $\iota(d')=v$ and $\tau(d')=\iota(e')=\tau(e')=u$; we set  $\psi(\ol{d}')=n$ and $\psi(d')=\ell_2 m$ and $\psi(\ol{e}')=m$ and $\psi(e')=\ell m$. We say that the GBS graph $(\Gamma',\psi)$ is obtained from $(\Gamma,\psi)$ by a \textbf{connection move}.

Let $\cG,\cG'$ be the GBS graph of groups associated to $(\Gamma,\psi),(\Gamma',\psi)$ respectively. At the level of the universal groups $\FG{\cG},\FG{\cG'}$ we have an isomorphism defined by
$$\begin{cases}
d'=e^k\ol{d}\\
e'=de\ol{d}
\end{cases}
\qquad\text{ with inverse }\qquad
\begin{cases}
d=(e')^k\ol{d}'\\
e=d'e'\ol{d}'
\end{cases}$$
In particular, a connection move induces an isomorphism at the level of the fundamental group $\pi_1(\cG)$.

Let $\bw,\bw_1,\bw_2\in\pA$ and $k\in\bbN$ be such that $\bw_1+\bw_2=k\cdot\bw$. At the level of the affine representation, we have two edges $q\edge  p+\bw_1$ and $p\edge  p+\bw$. The connection move has the effect of replacing them with two edges $p\edge  p+\bw_2$ and $q\edge  q+\bw$ (see Figure \ref{fig:connection}),
$$\begin{cases}
q\edge  p+\bw_1\\
p\edge  p+\bw
\end{cases}
\xrightarrow{\text{connection}}\quad
\begin{cases}
p\edge  q+\bw_2\\
q\edge  q+\bw.
\end{cases}$$

\begin{remark}
In the definition of connection move, we also allow the vertices $u,v$ to coincide.
\end{remark}

\begin{figure}[H]
\centering

\includegraphics[width=0.8\textwidth]{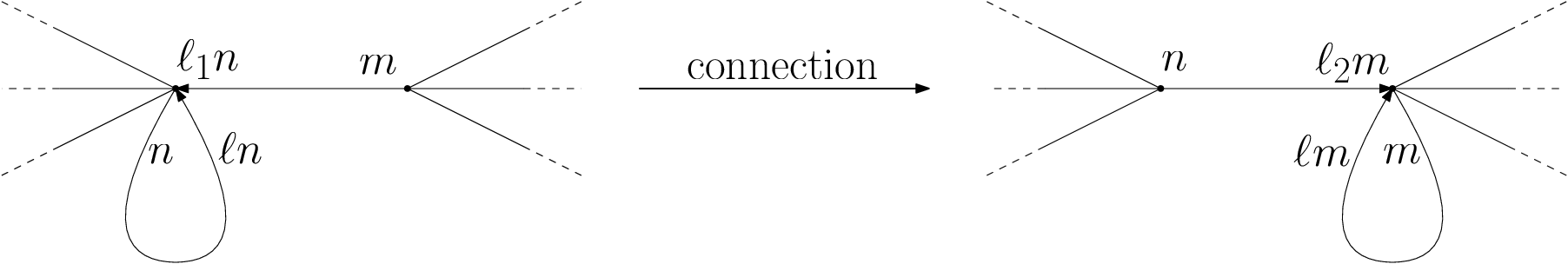}

\vspace{0.5cm}

\includegraphics[width=\textwidth]{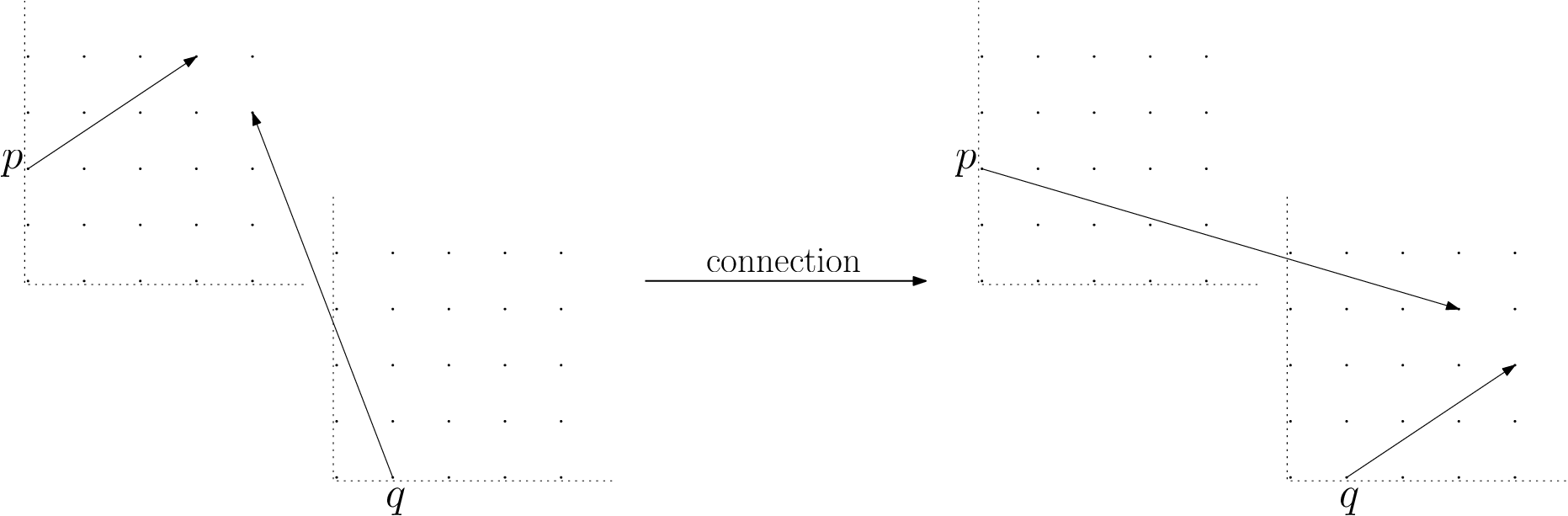}

\caption{An example of a connection move. Above you can see the GBS graphs; here $\ell_1\ell_2=\ell^k$ for some integer $k\ge0$. Below you can see the corresponding affine representations.}
\label{fig:connection}
\end{figure}

\subsection{Relations between the moves}

In this section, we show that the moves that we have introduced, induction, swap, and connection, are \emph{derived moves}, that is they can be obtained as a sequence of elementary contractions, elementary expansions, and slides. We also give explicit bounds on the number of vertices in the graphs that appear along the sequence.

\begin{lem}[Induction is a derived move]\label{induction-from-others}
Let $(\Gamma,\psi)$ be a GBS graph. Then, every induction move can be written as a sequence of elementary contractions, elementary expansions, and slides. This can be done in such a way that every graph along the sequence has at most $\abs{V(\Gamma)}+1$ vertices.
\end{lem}
\begin{proof}
Let $\Lambda$ be the affine representation of $(\Gamma,\psi)$ and let us say that $\Lambda$ contains an edge $e$ given by $(v,\mathbf{0})\edge (v,\bw)$ for some $v\in V(\Gamma)$ and $\bw\in\pA$. We choose $\bw_1\in\pA$ with $\bw_1\le\bw$. We perform an elementary expansion adding a vertex $v'$ and an edge $e'$ given by $(v',\mathbf{0})\edge (v,\bw-\bw_1)$.

For every other edge $f$ with an endpoint at $(v,\ba)$ for some $\ba\in\pA$, we perform a slide (along $e$) changing that endpoint into $(v,\ba+\bw)$ and then another slide (along $e'$) changing that endpoint into $(v',\ba+\bw_1)$.

Finally, we perform slides to change
$$\begin{cases}
(v,\mathbf{0})\edge(v,\bw)\\
(v,\bw-\bw_1)\edge(v',\mathbf{0})
\end{cases}
\xrightarrow{\text{slide}}\quad
\begin{cases}
(v,\mathbf{0})\edge(v',\bw_1)\\
(v,\bw-\bw_1)\edge(v',\mathbf{0})
\end{cases}
\xrightarrow{\text{slide}}\quad
\begin{cases}
(v,\mathbf{0})\edge(v',\bw_1)\\
(v',\bw)\edge(v',\mathbf{0})
\end{cases}$$
and then we perform a contraction move to remove the edge $(v,\mathbf{0})\edge(v',\bw_1)$ and the vertex $v$.

We remain with a vertex $v'$, an edge $(v',\bw)\edge(v',\mathbf{0})$; for every other edge $f$ which had an endpoint at $(v,\ba)$, that endpoint is now at $(v',\ba+\bw_1)$. By reiterating the same procedure, we can get any induction move. The statement follows.
\end{proof}

\begin{lem}[Swap is a derived move]\label{swap-from-others}
Let $(\Gamma,\psi)$ be a GBS graph. Then every swap move can be written as a sequence of elementary contractions, elementary expansions, and slides. This can be done in such a way that every graph along the sequence has at most $\abs{V(\Gamma)}+2$ vertices.
\end{lem}
\begin{proof}
Let $\Lambda$ be the affine representation of $(\Gamma,\psi)$, and suppose $\Lambda$ contains two edges $(v,\ba)\edge(v,\ba+\bw_1)$ and $(v,\bb)\edge(v,\bb+\bw_2)$ with $v\in V(\Gamma)$ and $\ba,\bb,\bw_1,\bw_2\in\pA$ such that $\ba\le \bb\le \ba+k_1\bw_1$ and $\ba\le \bb\le \ba+k_2\bw_2$ for some $k_1,k_2\in\bbN$. We proceed as follows:
$$
\begin{cases}
(v,\ba)\edge (v,\ba+\bw_1)\\
(v,\bb)\edge (v,\bb+\bw_2)
\end{cases}
\xrightarrow{\text{el. exp.}}\quad
\begin{cases}
(v,\ba)\edge (v,\ba+\bw_1)\\
(v,\bb)\edge (v,\bb+\bw_2)\\
(v,\ba)\edge(u,\mathbf{0})
\end{cases}
$$
$$
\xrightarrow{\text{slides}}\quad
\begin{cases}
(u,\mathbf{0})\edge (u,\bw_1)\\
(u,\bb-\ba)\edge (u,\bb-\ba+\bw_2)\\
(v,\ba)\edge(u,\mathbf{0})
\end{cases}
\xrightarrow{\text{slides}}\quad
\begin{cases}
(u,\mathbf{0})\edge (u,\bw_1)\\
(u,\bb-\ba)\edge (u,\bb-\ba+\bw_2)\\
(v,\ba)\edge(u,k_1\bw_1+k_2\bw_2)
\end{cases}
$$
$$
\xrightarrow{\text{induction (using $\bw_1$)}}\quad
\begin{cases}
(u,\mathbf{0})\edge (u,\bw_1)\\
(u,\mathbf{0})\edge (u,\bw_2)\\
(v,\ba)\edge(u,k_1\bw_1+k_2\bw_2-(\bb-\ba))
\end{cases}
$$
$$
\xrightarrow{\text{induction (using $\bw_2$)}}\quad
\begin{cases}
(u,\bb-\ba)\edge (u,\bb-\ba+\bw_1)\\
(u,\mathbf{0})\edge (u,\bw_2)\\
(v,\ba)\edge(u,k_1\bw_1+k_2\bw_2)
\end{cases}
$$
$$
\xrightarrow{\text{slides}}\quad
\begin{cases}
(u,\bb-\ba)\edge (u,\bb-\ba+\bw_1)\\
(u,\mathbf{0})\edge (u,\bw_2)\\
(v,\ba)\edge(u,\mathbf{0})
\end{cases}
\xrightarrow{\text{slides}}\quad
\begin{cases}
(v,\bb)\edge (v,\bb+\bw_1)\\
(v,\ba)\edge (v,\ba+\bw_2)\\
(v,\ba)\edge(u,\mathbf{0})
\end{cases}
$$
$$
\xrightarrow{\text{el. contr.}}\quad
\begin{cases}
(v,\bb)\edge (v,\bb+\bw_1)\\
(v,\ba)\edge (v,\ba+\bw_2)
\end{cases}
$$
and the statement follows from Lemma \ref{induction-from-others}.
\end{proof}

\begin{lem}[Connection is a derived move]\label{connection-from-others}
Let $(\Gamma,\psi)$ be a GBS graph. Then every connection move can be written as a sequence of elementary contractions, elementary expansions, and slides. This can be done in such a way that every graph along the sequence has at most $\abs{V(\Gamma)}+2$ vertices.
\end{lem}
\begin{proof}
Let $\Lambda$ be the affine representation of $(\Gamma,\psi)$, and suppose $\Lambda$ contains two edges $(v,\ba)\edge(v,\ba+\bw)$ and $(v',\bb)\edge(v,\ba+\bw_1)$ with $v,v'\in V(\Gamma)$ and for some $\ba,\bb,\bw,\bw_1,\bw_2\in\pA$ satisfying $\bw_1+\bw_2=k\bw$ for some $k\in\bbN$. We proceed as follows:
$$
\begin{cases}
(v,\ba)\edge (v,\ba+\bw)\\
(v',\bb)\edge (v,\ba+\bw_1)
\end{cases}
\xrightarrow{\text{el. exp.}}\quad
\begin{cases}
(v,\ba)\edge (v,\ba+\bw)\\
(v',\bb)\edge (v,\ba+\bw_1)\\
(v,\ba)\edge (u,\mathbf{0})
\end{cases}
\xrightarrow{\text{slides}}\quad
\begin{cases}
(u,\mathbf{0})\edge (u,\bw)\\
(v',\bb)\edge (u,\bw_1)\\
(v,\ba)\edge (u,\mathbf{0})
\end{cases}
$$
$$
\xrightarrow{\text{induction}}\quad
\begin{cases}
(u,\mathbf{0})\edge (u,\bw)\\
(v',\bb)\edge (u,\bw_1+\bw_2)\\
(v,\ba)\edge (u,\bw_2)
\end{cases}
\xrightarrow{\text{slides}}\quad
\begin{cases}
(u,\mathbf{0})\edge (u,\bw)\\
(v',\bb)\edge (u,\mathbf{0})\\
(v,\ba)\edge (u,\bw_2)
\end{cases}
$$
$$
\xrightarrow{\text{slides}}\quad
\begin{cases}
(v',\bb)\edge (v',\bb+\bw)\\
(v',\bb)\edge (u,\mathbf{0})\\
(v,\ba)\edge (v',\bb+\bw_2)
\end{cases}
\xrightarrow{\text{el. contr.}}\quad
\begin{cases}
(v',\bb)\edge (v',\bb+\bw)\\
(v,\ba)\edge (v',\bb+\bw_2)
\end{cases}
$$
and the statement follows from Lemma \ref{induction-from-others}.
\end{proof}

\begin{figure}[H]
	\centering
    \includegraphics[width=\textwidth]{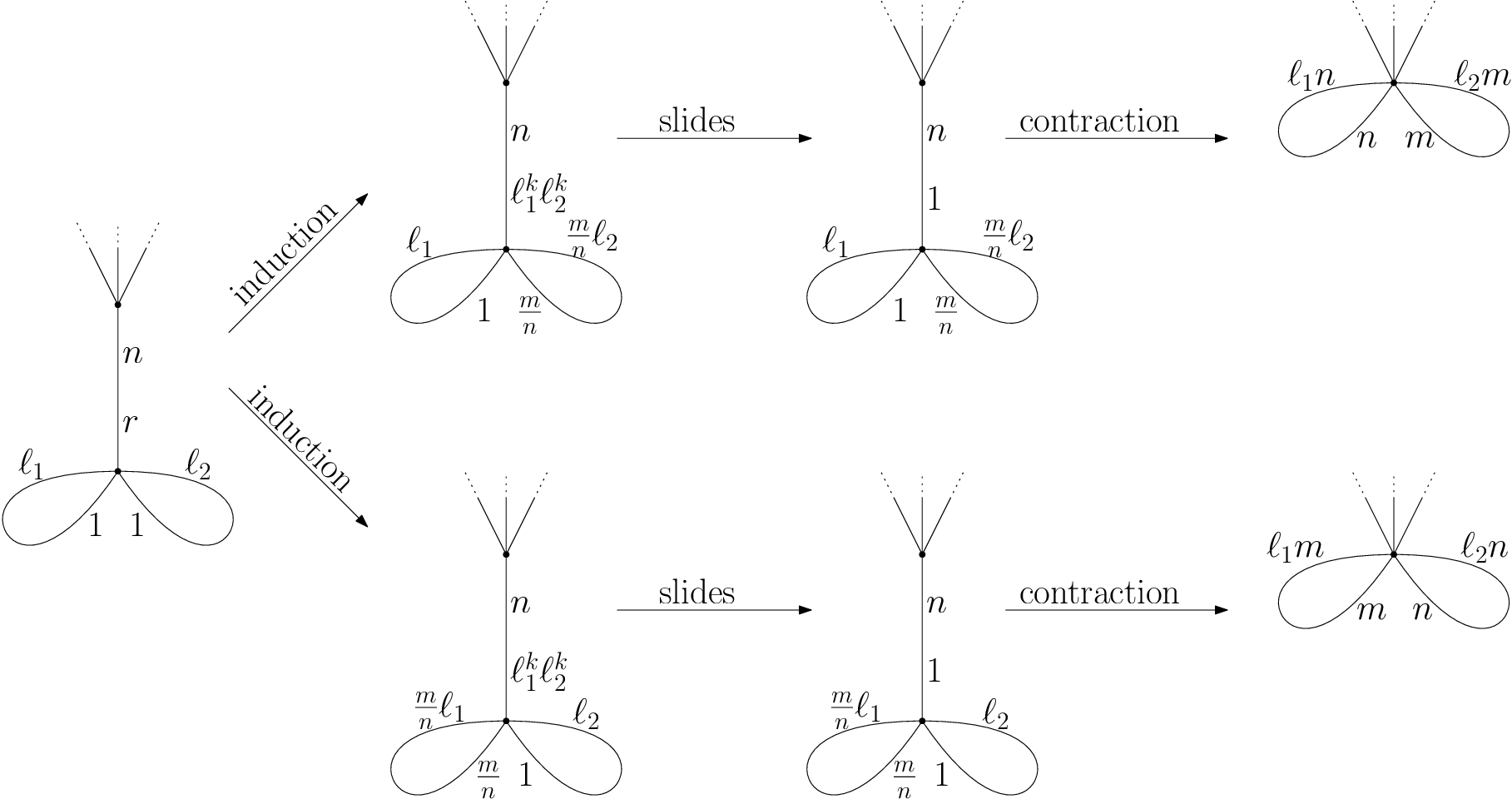}
	\caption{On the left we have a (portion of a) GBS graph, where $\ell_1,\ell_2,r,n \in \bbZ \setminus\{0\}$ 
	are such that $r\divides \ell_1^{k_1}$ and $r\divides\ell_2^{k_2}$ for some $k_1,k_2\ge0$; we also call $m=\ell_1^{k_1}\ell_2^{k_2}n/r$. The vertex below is ``redundant", as its generator can be conjugated inside the other vertex. However, there are two different ways of making this happen. 	 We can use $k_1$ times the edge that adds a factor $\ell_1$, and then move to the other vertex. Or we can use $k_2$ times the edge that gives a factor $\ell_2$, and then move to the other vertex. This gives us two different ways of collapsing the redundant vertex; one of them is the sequence of moves above, and the other is the sequence of moves below. The difference between these two sequences of moves is a swap move.
	}
	\label{fig:collapse-in-two-ways}
\end{figure}

\section{Controlled linear algebra}\label{sec:controlled-linear-algebra}

The main result of this section is Theorem \ref{thm:controlled}, stating that slide and swap moves allow to pass from a configuration to any other configuration that has the ``same linear algebra''. 
This is done under the hypothesis of having an edge that ``controls" all the others (see Definition \ref{def:control} below). The proof is quite technical and relies on the tools developed in \cite[Appendix B]{ACK-iso2}. We point out that the results of \cite[Appendix B]{ACK-iso2} are developed in more generality than what we are using here, as they are thought to be applied to a wider range of configurations. Moreover, the results of this section, combined with the subsequent Theorem \ref{thm:sequence-new-moves}, can be used to provide a complete description of the isomorphism problem for large families of GBSs (see \Cref{cor:iso-controlled}).

\begin{defn}[Support of a point]\label{def:support}
For $\ba=(a_0,a_1,a_2,\dots)\in\pA$ define its \textbf{support} as the set
$$\supp{\ba}:=\{i\ge1 : a_i\not=0\}\subseteq\bbN\setminus\{0\}.$$
\end{defn}
\begin{remark}
Notice that we omit the $\bbZ/2\bbZ$ component from the definition of support.
\end{remark}

\begin{defn}\label{def:control}
Let $\ba,\bb,\bw\in\pA$. We say that $\ba,\bw$ \textbf{controls} $\bb$ if any of the following equivalent conditions hold {\rm(}see also {\rm Figure \ref{fig:control})}:
\begin{enumerate}
\item We have $\ba\le \bb\le \ba+k\bw$ for some $k\in\bbN$.
\item We have $\bb-\ba\ge\mathbf{0}$ and $\supp{\bb-\ba}\subseteq\supp{\bw}$.
\end{enumerate}
\end{defn}

\begin{figure}[H]
	\centering
    
    \includegraphics[width=0.7\textwidth]{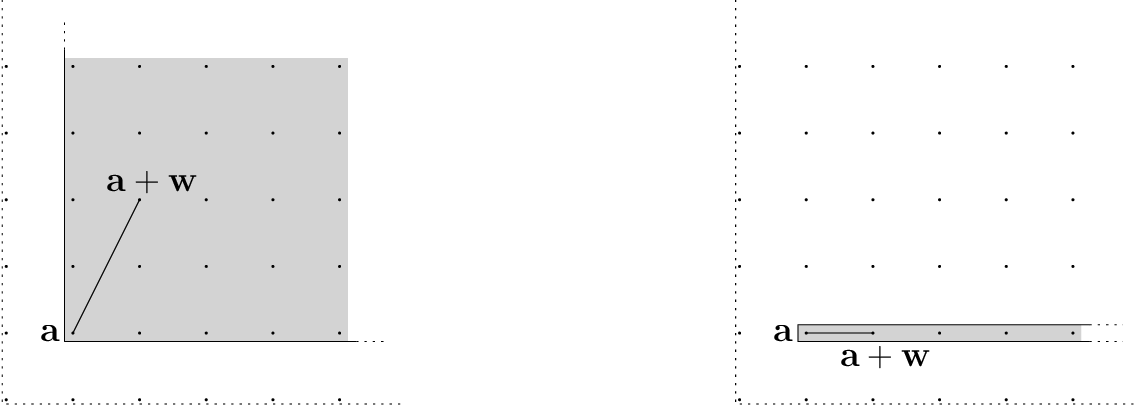}
    
    \caption{The region controlled by an edge $\ba\edge\ba+\bw$, in two different configurations.}
    \label{fig:control}
\end{figure}

\subsection{Generalized slide moves}

Let $(\Gamma,\psi)$ be a GBS graph and let $\Lambda$ be its affine representation. Fix a vertex $v\in V(\Gamma)$ and look at the copy $\pA_v$ inside $\Lambda$. The swap move allows us to change the two edges
$$
\begin{cases}
\ba\edge  \ba+\bw_1\\
\bb\edge  \bb+\bw_2
\end{cases}
\qquad\text{into}\qquad
\begin{cases}
\ba\edge  \ba+\bw_2\\
\bb\edge  \bb+\bw_1
\end{cases}
$$
provided that $\bw_1,\bw_2\in\pA$ and $\ba,\bw_1$ controls $\bb$ and $\ba,\bw_2$ controls $\bb$ (see Figure \ref{fig:controlled-swap}). We now show that actually, we can obtain much more flexibility, through the self-slide move (Lemma \ref{self-slide}, Figure \ref{fig:controlled-self-slide}) and the reverse slide move (Lemma \ref{reverse-slide}, Figure \ref{fig:controlled-reverse-slide}).

\begin{figure}[H]
	\centering

    \includegraphics[width=0.7\textwidth]{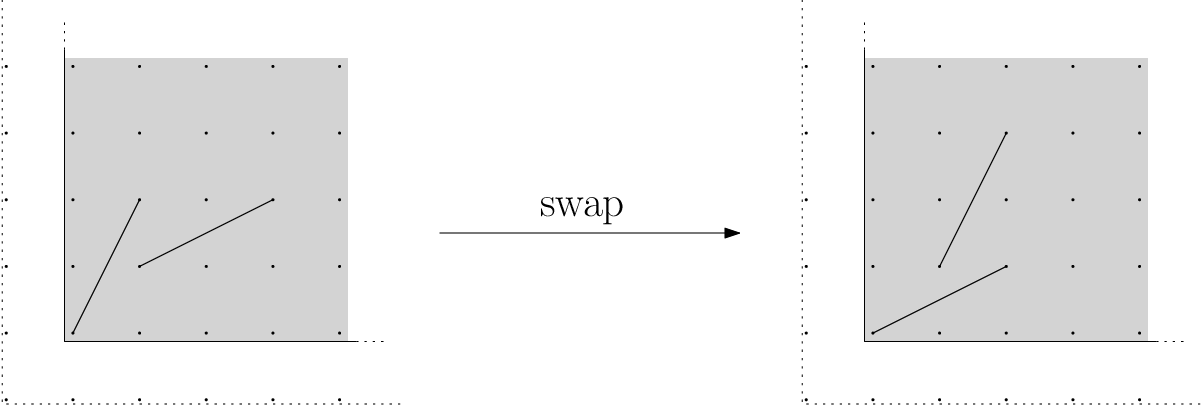}
    
	\caption{A swap move. In evidence the fact that $\ba,\bw_1$ and $\ba,\bw_2$ have to control $\bb$.}
	\label{fig:controlled-swap}
\end{figure}

\begin{lem}[Self-slide]\label{self-slide}
Let $\ba,\bb,\bw\in\pA$ and $\bx\in\bA$. Suppose that $\ba,\bw$ controls $\bb,\bb+2\bx$. Then we can change
$$
\begin{cases}
\ba\edge  \ba+\bw\\
\bb\edge  \bb+\bx
\end{cases}
\qquad\text{into}\qquad
\begin{cases}
\ba\edge  \ba+\bw\\
\bb+\bx\edge  \bb+2\bx
\end{cases}
$$
using a sequence of slides and swaps.
\end{lem}
\begin{proof}
$$
\begin{cases}
\ba\edge  \ba+\bw\\
\bb\edge  \bb+\bx
\end{cases}
\xrightarrow{\text{slides}}\quad
\begin{cases}
\ba\edge  \ba+\bw\\
\bb\edge  \bb+\bx+k\bw
\end{cases}
\xrightarrow{\text{swap}}\quad
\begin{cases}
\ba\edge  \ba+\bx+k\bw\\
\bb\edge  \bb+\bw
\end{cases}
\xrightarrow{\text{slides}}
$$
$$
\xrightarrow{\text{slides}}\quad
\begin{cases}
\ba\edge  \ba+\bx+k\bw\\
\bb+\bx+k\bw\edge  \bb+\bx+k\bw+\bw
\end{cases}
\xrightarrow{\text{swap}}
$$
$$
\xrightarrow{\text{swap}}\quad
\begin{cases}
\ba\edge  \ba+\bw\\
\bb+\bx+k\bw\edge  \bb+2\bx+2k\bw
\end{cases}
\xrightarrow{\text{slides}}\quad
\begin{cases}
\ba\edge  \ba+\bw\\
\bb+\bx\edge  \bb+2\bx
\end{cases}
$$
\end{proof}

\begin{figure}[H]
	\centering

    \includegraphics[width=0.7\textwidth]{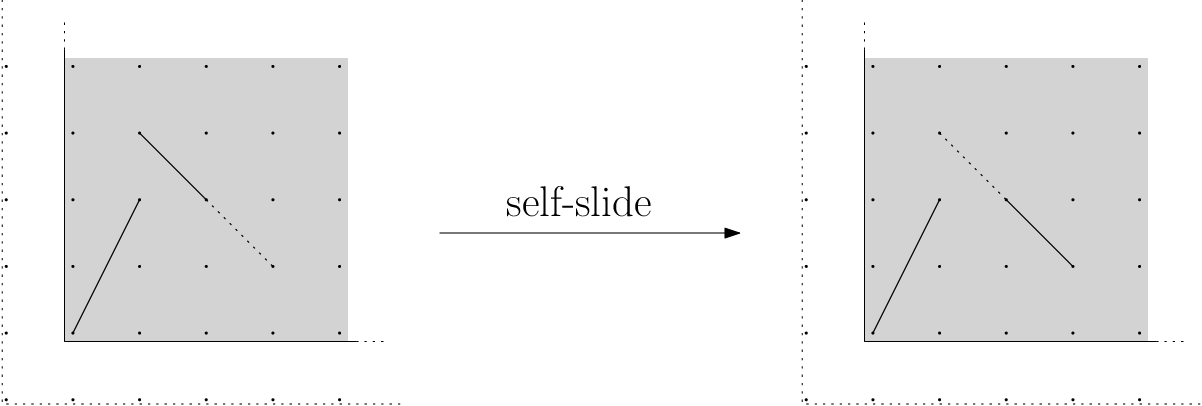}
    
	\caption{A self-slide move. In evidence the fact that $\ba,\bw$ has to control $\bb,\bb+2\bx$.}
	\label{fig:controlled-self-slide}
\end{figure}

\begin{lem}[Reverse slide]\label{reverse-slide}
Let $\ba,\bb,\bw\in\pA$ and $\bx\in\bA$ with $\bw+\bx\ge\mathbf{0}$. Suppose that $\ba,\bw$ controls $\bb,\bb+\bx$ and suppose that $\ba,\bw+\bx$ controls $\bb,\bb+\bx$. If $\pA_v$ contains edges $\ba\edge\ba+\bw$ and $\bb\edge\bb+\bx$, then we can change
$$
\begin{cases}
\ba\edge  \ba+\bw\\
\bb\edge  \bb+\bx
\end{cases}
\qquad\text{into}\qquad
\begin{cases}
\ba\edge  \ba+\bw+\bx\\
\bb\edge  \bb+\bx
\end{cases}
$$
by a sequence of slides and swaps.
\end{lem}
\begin{proof}
Since $\ba,\bw$ controls $\bb+\bx$ we have $-\bx\le\bb-\ba$. Since $\ba,\bw+\bx$ controls $\bb$ we have $\bb-\ba\le k(\bw+\bx)$ for some $k\in\bbN$. It follows that $-\bx\le k(\bw+\bx)$ meaning that $\bw\le (k+1)(\bw+\bx)$, and thus $\ba,\bw+\bx$ controls $\bb+\bw$. We now proceed as follows:
$$
\begin{cases}
\ba\edge  \ba+\bw\\
\bb\edge  \bb+\bx
\end{cases}
\xrightarrow{\text{slide}}\quad
\begin{cases}
\ba\edge  \ba+\bw\\
\bb\edge  \bb+\bw+\bx
\end{cases}
\xrightarrow{\text{swap}}\quad
\begin{cases}
\ba\edge  \ba+\bw+\bx\\
\bb\edge  \bb+\bw
\end{cases}
\xrightarrow{\text{slide}}
$$
$$
\xrightarrow{\text{slide}}\quad
\begin{cases}
\ba\edge  \ba+\bw+\bx\\
\bb+\bw+\bx\edge  \bb+\bw
\end{cases}
\xrightarrow{\text{self-slide}}\quad
\begin{cases}
\ba\edge  \ba+\bw+\bx\\
\bb+\bw+2\bx\edge  \bb+\bw+\bx
\end{cases}
\xrightarrow{\text{slides}}
$$
$$
\xrightarrow{\text{slides}}\quad
\begin{cases}
\ba\edge  \ba+\bw+\bx\\
\bb+\bx\edge  \bb
\end{cases}
\qquad\text{which is the same as}\quad
\begin{cases}
\ba\edge  \ba+\bw+\bx\\
\bb\edge  \bb+\bx
\end{cases}
$$
\end{proof}

\begin{figure}[H]
	\centering
    
    \includegraphics[width=0.7\textwidth]{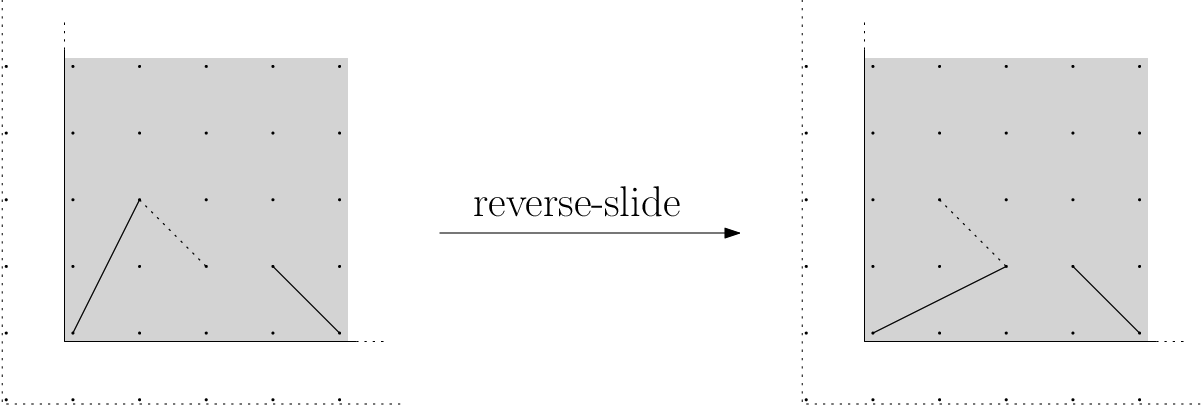}
    
    \caption{A reverse-slide move. In evidence the fact that $\ba,\bw$ and $\ba,\bw+\bx$ have to control $\bb,\bb+\bx$.}
	\label{fig:controlled-reverse-slide}
\end{figure}

\subsection{Controlled linear algebra}

\begin{lem}\label{controlled2}
Let $\ba,\bw,\bw',\bb,\bb'\in\pA$ and $\bx,\bx'\in\bA$. Suppose that we have the following:
\begin{enumerate}
\item $\ba,\bw$ controls $\bb,\bb+\bx$.
\item $\ba,\bw'$ controls $\bb',\bb'+\bx'$.
\item $\gen{\bw,\bx}=\gen{\bw',\bx'}\sgr\bA$.
\item $\bb'-\bb\in\gen{\bw,\bx}$.
\end{enumerate}
Then we can change
$$
\begin{cases}
\ba\edge  \ba+\bw\\
\bb\edge  \bb+\bx
\end{cases}
\qquad\text{into}\qquad
\begin{cases}
\ba\edge  \ba+\bw'\\
\bb'\edge  \bb'+\bx'
\end{cases}
$$
by a sequence of slides and swaps.
\end{lem}
\begin{proof}
Since $\ba,\bw$ controls $\bb,\bb+\bx$, we must have $\supp{\bx}\subseteq\supp{\bw}$, and similarly $\supp{\bx'}\subseteq\supp{\bw'}$. Thus, the condition $\gen{\bw,\bx}=\gen{\bw',\bx'}$ implies that $\supp{\bw}=\supp{\bw'}$. The subgroup $\gen{\bw,\bx}=\gen{\bw',\bx'}$ is isomorphic to one of $1,\bbZ/2\bbZ,\bbZ,\bbZ\oplus(\bbZ/2\bbZ),\bbZ^2$. If it is isomorphic to $1$ or $\bbZ/2\bbZ$, then the statement is trivial. We deal with the other three cases.

CASE 1: Suppose $\gen{\bw,\bx}=\gen{\bw',\bx'}$ is isomorphic to $\bbZ$, and fix a generator $\bz$ so that $\bw=\lambda\bz$ and $\bx=\mu\bz$ and $\bw'=\lambda'\bz$ and $\bx'=\mu'\bz$. We can assume $\lambda,\lambda'>0$ (up to changing $\bz$ with $-\bz$). We can also assume $\mu>0$ (up to substituting $\bx$ with $\bx+k\bw$ with slide moves, for $k\in\bbN$ big enough). We can now perform the Euclidean algorithm on the two natural numbers $(\lambda,\mu)$ as follows. If $\lambda\le\mu$ then we change the edge $\bb\edge \bb+\bx$ into $\bb\edge \bb+\bx-\bw$ by slide, and we get the pair of natural numbers $(\lambda,\mu-\lambda)$. If $\lambda>\mu$ then we change the edge $\ba\edge \ba+\bw$ into $\ba\edge \ba+\bw-\bx$ by reverse slide (Lemma \ref{reverse-slide}), and we get the pair of natural numbers $(\lambda-\mu,\lambda)$. When $\mu=0$ we stop. Since $\gen{\bw,\bx}=\gen{\bz}$ this forces $\lambda=1$. We remain with the two edges $\ba\edge\ba+\bz$ and $\bb\edge\bb$. We proceed in the same way with the two edges $\ba\edge\ba+\bw'$ and $\bb'\edge\bb'+\bx'$ and we arrive at the configuration $\ba\edge\ba+\bz$ and $\bb'\edge\bb'$. Since $\bb'-\bb\in\gen{\bz}$, using slide moves we can change $\bb'$ into $\bb$. The statement follows.

CASE 2: Suppose $\gen{\bw,\bx}=\gen{\bw',\bx'}$ is isomorphic to $\bbZ\oplus\bbZ/2\bbZ$, and fix generators $\bz,\bt$ for $\bbZ\oplus\bbZ/2\bbZ$ where $\bt=(1,0,0,0,\dots)$. We write $\bw=\lambda\bz+\theta\bt$ and $\bx=\mu\bz+\rho\bt$ and we assume $\lambda>0$ (up to changing $\bz$ with $-\bz$). We now proceed as in Case 1, ignoring the components in $\bt$: we can assume $\mu>0$ and we perform the Euclidean algorithm on $(\lambda,\mu)$. Starting from $\ba\edge\ba+\bw$ and $\bb\edge\bb+\bx$ we arrive at the configuration $\ba\edge\ba+\bz$ and $\bb\edge\bb+\bt$. Similarly, from $\ba\edge\ba+\bw'$ and $\bb'\edge\bb'+\bx'$ we arrive at $\ba\edge\ba+\bz$ and $\bb'\edge\bb'+\bt$. Since $\bb'-\bb\in\gen{\bz,\bt}$, using slide moves and self-slide moves (Lemma \ref{self-slide}) we can change $\bb'$ into $\bb$. The statement follows.

CASE 3: Suppose that $\gen{\bw,\bx}=\gen{\bw',\bx'}$ is isomorphic to $\bbZ^2$. We can assume $\bx'\ge\bw'\ge\mathbf{0}$ (up to substituting $\bx'$ with $\bx'+k\bw'$ with slide moves, for $k\in\bbN$ big enough). In particular we have $\supp{\bx'}=\supp{\bw'}$. We write $\bw=\lambda\bw'+\mu\bx'$ for $\lambda,\mu\in\bbZ$ and we observe that at least one of $\lambda,\mu$ is $>0$; we can assume $\lambda>0$ (up to interchanging $\bw'$ and $\bx'$ with a swap move). If $\mu<0$ then we think of the edge $\bb'\edge \bb'+\bx'$ as an edge $(\bb'+\bx')\edge (\bb'+\bx')-\bx'$: we substitute $\bb'$ and $\bx'$ with $\bb'+\bx'$ and $-\bx'$ respectively, and $\mu$ with $-\mu$.

We thus assume $\bw=\lambda\bw'+\mu\bx'$ with $\lambda>0$ and $\mu\ge0$. We now perform the Euclidean algorithm on the two natural numbers $(\lambda,\mu)$ as follows. If $\lambda>\mu$ then we change the edge
$$\bb'\edge \bb'+\bx'$$
by slide into
$$\bb'\edge \bb'+\bx'+\bw'$$
and, since $\bw=(\lambda-\mu)\bw'+\mu(\bx'+\bw')$, we obtain the new pair $(\lambda-\mu,\mu)$. If instead $\lambda\le\mu$, we observe that $\mu(\bw'+\bx')=\bw+(\mu-\lambda)\bw'$ and thus $\bw'+\bx'\ge\mathbf{0}$ and $\supp{\bw'+\bx'}=\supp{\bw'}=\supp{\bw}$: we can change the edge
$$\ba\edge \ba+\bw'$$
using the reverse slide move of Lemma \ref{reverse-slide} into
$$\ba\edge \ba+(\bw'+\bx')$$
and, since $\bw=\lambda(\bw'+\bx')+(\mu-\lambda)\bx'$, we obtain the new pair $(\lambda,\mu-\lambda)$. When $\mu=0$ we stop.

Now we have $\bw=\lambda\bw'$ with $\lambda>0$, but since $\gen{\bw,\bx}=\gen{\bw',\bx'}\cong\bbZ^2$, this forces $\lambda=1$ and $\bw'=\bw$. Now the condition $\gen{\bw,\bx}=\gen{\bw,\bx'}$ forces $\bx=\bx'+\theta\bw$ for some $\theta\in\bbZ$. We take the edge $\bb'\edge \bb'+\bx'$ and we either add $\theta\bw$ to the second endpoint (if $\theta\ge0$) or we add $(-\theta)\bw$ to the first endpoint (if $\theta<0$). This has the effect of changing $\bx'$ into $\bx$. Finally, we observe that the condition $\bb'-\bb\in\gen{\bx,\bw}$ remained true along the whole procedure, and thus we can use slide moves and the self-slide move of Lemma \ref{self-slide} in order to get $\bb'=\bb$, as desired.
\end{proof}

\begin{thm}\label{thm:controlled}
Let $\ba,\bw,\bw',\bb_1,\dots,\bb_m,\bb_1',\dots,\bb_m'\in\pA$ and $\bx_1,\dots,\bx_m,\bx_1',\dots,\bx_m'\in\bA$ for $m\ge1$ integer. Suppose that we have the following:
\begin{enumerate}
\item $\ba,\bw$ controls $\bb_1,\bb_1+\bx_1,\dots,\bb_m,\bb_m+\bx_m$.
\item $\ba,\bw'$ controls $\bb_1',\bb_1'+\bx_1',\dots,\bb_m',\bb_m'+\bx_m'$.
\item $\gen{\bw,\bx_1,\dots,\bx_m}=\gen{\bw',\bx_1',\dots,\bx_m'}\sgr\bA$.
\item $\bb_1'-\bb_1,\dots,\bb_m'-\bb_m\in\gen{\bw,\bx_1,\dots,\bx_m}$.
\end{enumerate}
Then we can change
$$
\begin{cases}
\ba\edge\ba+\bw\\
\bb_1\edge\bb_1+\bx_1\\
\dots \\
\bb_m\edge\bb_m+\bx_m
\end{cases}
\qquad\text{into}\qquad
\begin{cases}
\ba\edge\ba+\bw'\\
\bb_1'\edge\bb_1'+\bx_1'\\
\dots \\
\bb_m'\edge\bb_m'+\bx_m'
\end{cases}
$$
by a sequence of slides and swaps.
\end{thm}
\begin{proof}
The proof is based on the results of \cite[Appendix B]{ACK-iso2}. The result for $m=1$ is the previous Lemma \ref{controlled2}, and thus we can assume $m\ge2$. Notice that $\supp{\bw}=\supp{\bw'}$. For $i=1,\dots,m$, and using slide moves, we can change $\bx_i$ to $\bx_i+C\bw$ for any $C\ge0$; thus we can assume $\supp{\bx_i}=\supp{\bw}$. Similarly we can assume $\supp{\bx_i'}=\supp{\bw'}$ for $i=1,\dots,m$.

Let $I=\supp{\bw}\cup\{0\}=\supp{\bw'}\cup\{0\}$. Consider the abelian group $\bbZ^I$ and consider the natural map $\pi:\bbZ^I\rar\bA$, and let $v=(2,0,0,\dots,0)\in\bbZ^I$ be the generator of $\ker\pi$. We choose liftings $w,w',x_i,x_i'\in\bbZ^I$ such that $\pi(w)=\bw,\pi(w')=\bw',\pi(x_i)=\bx_i,\pi(x_i')=\bx_i'$ for $i=1,\dots,m$; since $\supp{\bw}=\supp{\bw'}=\supp{\bx_i}=\supp{\bx_i'}=I\setminus\{0\}$, we can choose $w,w',x_i,x_i'$ in such a way that they are all $1$-big (according to \cite[Appendix B]{ACK-iso2}). By \cite[Theorem B.1]{ACK-iso2}, there is a sequence of Nielsen moves and Nielsen moves relative to $v$ going from $w,x_1,\dots,x_m$ to $w',x_1',\dots,x_m'$ such that every $(m+1)$-tuple along the sequence is $1$-big.

If there is a (relative) Nielsen move going from $u,y_1,\dots,y_m$ to $u',y_1',\dots,y_m'$, and if both the $(m+1)$-tuples are $1$-big, then we can change
$$
\begin{cases}
\ba\edge\ba+\pi(u)\\
\bb_1\edge\bb_1+\pi(y_1)\\
\dots \\
\bb_m\edge\bb_m+\pi(y_m)
\end{cases}
\quad=\qquad
\begin{cases}
\ba\edge\ba+\bw\\
\bb_1\edge\bb_1+\bx_1\\
\dots \\
\bb_m\edge\bb_m+\bx_m
\end{cases}
$$
into
$$
\begin{cases}
\ba\edge\ba+\pi(u')\\
\bb_1\edge\bb_1+\pi(y_1')\\
\dots \\
\bb_m\edge\bb_m+\pi(y_m')
\end{cases}
\quad=\qquad
\begin{cases}
\ba\edge\ba+\bw'\\
\bb_1\edge\bb_1+\bx_1'\\
\dots \\
\bb_m\edge\bb_m+\bx_m'
\end{cases}
$$
by slide moves and reverse slide moves (Lemma \ref{reverse-slide}). To conclude, for $i=1,\dots,m$ we change $\bb_i$ into $\bb_i'$ in the above writing, using slide moves and self-slide moves (Lemma \ref{self-slide}). The statement follows.
\end{proof}

\begin{figure}[H]
   	\centering

    \includegraphics[width=0.9\textwidth]{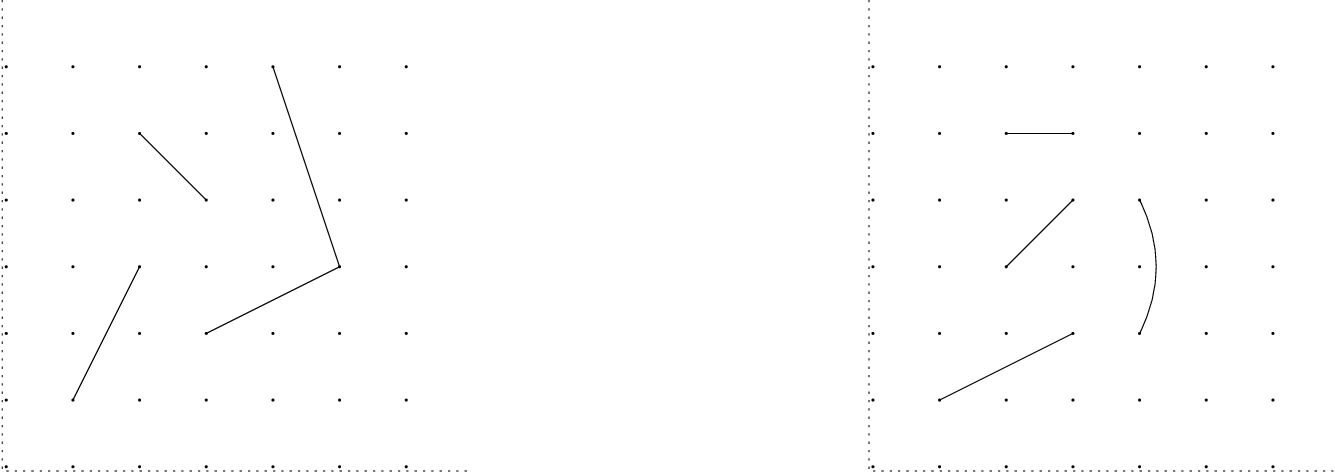}
    
    \caption{The two configurations on the figure are equivalent, provided that the two sets of edges have ``the same linear algebra" (i.e. if they generate the same subgroup, and if there are the same number of edges in each coset of such subgroup). If a GBS graph $(\Gamma,\psi)$ has only one vertex, and if there is one edge that controls all the others (as in the figure), then this gives us a complete description of the isomorphism problem for $(\Gamma,\psi)$; this will be made formal in \Cref{sec:sequences-of-moves}.}
    \label{fig:controlled-linear-algebra}
\end{figure}

\section{The coarse projection map}\label{sec:coarse-projection-map}

As already was observed by Forester \cite{For06}, some vertices in a given GBS graph are unnecessary, as their stabilizer is conjugated into the stabilizer of some other vertex group. Forester calls such vertices ``horizontal", while we refer to them as ``redundant" (see Definition \ref{def:redundant} below), to emphasize that they somehow do not bring any new information. In Section 4 of \cite{For06}, the author describes a procedure to eliminate a redundant vertex from a GBS graph, explaining how to produce a \textit{fully reduced} GBS graph. In Section \ref{sec:projection}, we describe a variant of the procedure that makes use of subsequent swap moves instead of subsequent induction moves - this helps us make the computations cleaner in the subsequent lemmas.

The crucial point is the following one: given a GBS graph, there might be several ways of eliminating a given redundant vertex, that is, different elimination processes leading to different resulting GBS graphs. The main aim of this section is to address the following question: what are all the possible resulting GBS graphs, obtained from a given one by eliminating a certain redundant vertex? The key result of this section is Proposition \ref{prop:indep-data}, which states the following: given a GBS graph $(\Gamma,\psi)$ and a redundant vertex $v\in V(\Gamma)$, all GBS graphs obtained from $(\Gamma,\psi)$ by eliminating $v$ are related by a sequence of sign-changes, slides, swaps and connections. In particular, we do not need to add any vertex to move from one to another; this will be crucial in the proof of the main Theorem \ref{thm:sequence-new-moves} in the next section.

\begin{defn}[Redundant vertex]\label{def:redundant}
Let $(\Gamma,\psi)$ be a GBS graph and let $\Lambda$ be its affine representation. A vertex $v\in V(\Gamma)$ is \textbf{redundant} if the origin $(v,\mathbf{0})\in\pA_v$ is conjugated to some point $p\not\in\pA_v$ {\rm(}see {\rm Definition \ref{def:conjugacy})}.
\end{defn}

\begin{remark}
Suppose $(\Gamma,\psi)$ is the GBS graph associated with the GBS graph of groups $\cG$. Then a vertex $v$ is redundant if and only if in the fundamental group $\pi_1(\cG)$ the conjugacy class $[v]$ coincides with the conjugacy class $[u^k]$ for some other vertex $u\not=v$ and some $k\in\bbZ$.
\end{remark}

\subsection{Sets of collapsing data}

\begin{defn}[Collapsing data]\label{def:collapsing-data}
Let $(\Gamma,\psi)$ be a GBS graph and let $\Lambda$ be its affine representation. A \textbf{set of collapsing data} for a vertex $v\in V(\Gamma)$ is given by pairwise distinct edges $e_1,\dots,e_n,e$ of $\Lambda$ satisfying the following properties:
\begin{enumerate}
\item For $j=1,\dots,n$ we have that $e_j$ goes from $(v,\ba_j)$ to $(v,\bb_j)$.
\item For $j=1,\dots,n$ we have $\supp{\ba_j}\subseteq(\supp{\bb_{j-1}}\cup\ldots \cup\supp{\bb_1})$.
\item The edge $e$ goes from $(v,\bc)$ to $(u,\bd)$ with $u\not=v$.
\item We have $\supp{\bc}\subseteq(\supp{\bb_n}\cup\ldots \cup\supp{\bb_1})$.
\end{enumerate}
\end{defn}

Notice that the element $\ba_1$ of the above Definition \ref{def:collapsing-data} has to be equal to either $(0,0,0,\dots)$ or $(1,0,0,\dots)$, since $\supp{\ba_1}=\emptyset$ (see Figure \ref{fig:collapsing-data}). The relevant edges in a set of collapsing data are the ones that make us gain new components in the support: if an edge $e_j$ is such that $\supp{\bb_j}\subseteq(\supp{\bb_{j-1}}\cup\ldots \cup\supp{\bb_1})$, then we can remove $e_j$ from the set of collapsing data to gain a shorter set of collapsing data $e_1,\dots,e_{j-1},e_{j+1},\dots,e_n,e$. Whenever we gain a new component in the support, we have to gain it with strictly positive increment: if $\ba_j=(a_{j0},a_{j1},a_{j2},\dots)$ and $\bb_j=(b_{j0},b_{j1},b_{j2},\dots)$, then for every $i\in\supp{\bb_j}\setminus(\supp{\bb_{j-1}}\cup\ldots \cup\supp{\bb_1})$ we must have $b_{ji}-a_{ji}\ge1$.

\begin{figure}[H]
	\centering

    \includegraphics[width=0.55\textwidth]{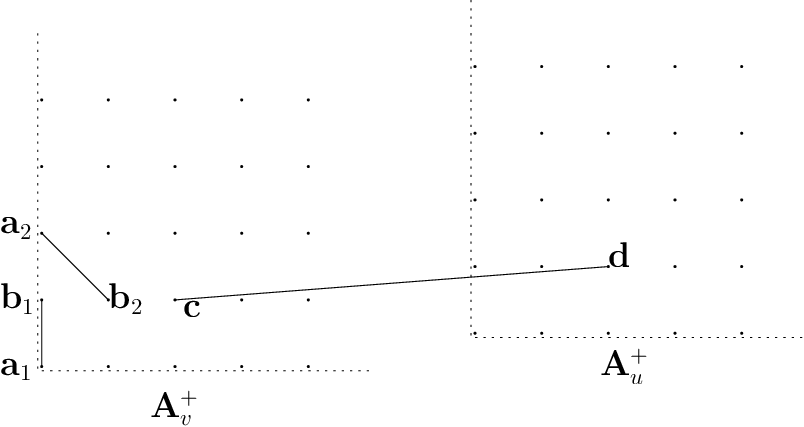}
    
	\caption{An example of a set of collapsing data for a vertex $v$ of a GBS graph.}
	\label{fig:collapsing-data}
\end{figure}

\begin{prop}\label{redundant-collapse}
If a vertex $v\in V(\Gamma)$ is redundant then it has a set of collapsing data.
\end{prop}
\begin{remark}
The converse is also true, see Proposition \ref{collapse-redundant} below.
\end{remark}
\begin{proof}
If $v$ is redundant, then there is an affine path from $(v,\mathbf{0})$ to some point $p\not\in\pA_v$, see Definition \ref{defn:affine path}. Let us take such an affine path $(e_1,\dots,e_n,e)$ of minimum possible length, with translation coefficients
$\bw_1,\dots,\bw_n,\bw\in\pA$; this forces the endpoints of the edges $e_1,\dots,e_n$ to belong to $\pA_v$ (otherwise we can just truncate the path to find a shorter one). For $j=1,\dots,n$, let us call $(v,\ba_j)$ and $(v,\bb_j)$ the endpoints of $e_j$: since the path starts at the origin we must have $\supp{\ba_1}=\supp{\bw_1}=\emptyset$ and by induction we obtain that $\supp{\ba_j},\supp{\bw_j}\subseteq(\supp{\bb_{j-1}}\cup\ldots \cup\supp{\bb_1})$. If $e$ goes from $(v,\bc)$ to $(u,\bd)$ then we must have $\supp{\bc}\subseteq\supp{\bb_n}\cup\supp{\bw_n}\subseteq\supp{\bb_n}\cup\ldots \cup \supp{\bb_1}$. Finally, if an edge appears several times in the sequence $e_1,\dots,e_n$, we remove all the occurrences except the first. The result is a set of collapsing data for $v$.
\end{proof}

\subsection{Sets of collapsing constants}

Let $(\Gamma,\psi)$ be a GBS graph and let $\Lambda$ be its affine representation. Denote by
$$K=\max\{\abs{a_i} : \tau(e)=(v,(a_0,a_1,a_2,\dots)) \text{ for some } e\in E(\Lambda) \text{ and } i\ge1\}.$$

\begin{defn}[Collapsing constants]
Let $v\in V(\Gamma)$ be a redundant vertex and let $e_1,\dots,e_n,e$ be a set of collapsing data. A \textbf{set of collapsing constants} is given by integers $k_1,\dots,k_n$ which are strictly bigger than $K$.
\end{defn}

Let $v$ be a redundant vertex. Let $e_1,\dots,e_n,e$ be a set of collapsing data for $v$, where $e_j$ goes from $(v,\ba_j)$ to $(v,\bb_j)$, and $e$ goes from $(v,\bc)$ to $(u,\bd)$ for some $u\not=v$, as in Definition \ref{def:collapsing-data}. Let $k_1,\dots,k_n$ be a set of collapsing constants. For $j=1,\dots,n$ define inductively $\bw_j$ with $\bw_1=\bb_1-\ba_1$ and $\bw_j=(\bb_j-\ba_j)+k_{j-1}\bw_{j-1}$; this means that
$$\bw_j=(\bb_j-\ba_j)+k_{j-1}(\bb_{j-1}-\ba_{j-1})+k_{j-1}k_{j-2}(\bb_{j-2}-\ba_{j-2})+\dots +k_{j-1}k_{j-2}\dots k_1(\bb_1-\ba_1).$$

\begin{lem}\label{lemmaw}
For $j=1,\dots,n$ we have $\bw_j\ge\mathbf{0}$ and $\supp{\bw_j}\supseteq(\supp{\bb_j}\cup\ldots \cup\supp{\bb_1})$.
\end{lem}
\begin{proof}
Since $\supp{\ba_1}=\emptyset$ we have $\bw_1=\bb_1-\ba_1\ge\mathbf{0}$ and the components of $\bw_1$ in $\supp{\bb_1}$ must be strictly positive. We now proceed by induction on $j$. By definition $\bw_j=(\bb_j-\ba_j)+k_{j-1}\bw_{j-1}$. Since $\supp{\ba_j}\subseteq(\supp{\bb_{j-1}}\cup\ldots \cup\supp{\bb_1})$, the components of $\bw_j$ in $\supp{\bb_j}\setminus(\supp{\bb_{j-1}}\cup\ldots \cup\supp{\bb_1})$ must be strictly positive. Since $k_{j-1}\ge K+1$, and since by inductive hypothesis the components of $\bw_{j-1}$ in $(\supp{\bb_{j-1}}\cup\ldots \cup\supp{\bb_1})$ are strictly positive, we obtain that the components of $\bw_j$ in $(\supp{\bb_{j-1}}\cup\ldots \cup\supp{\bb_1})$ are strictly positive. Thus $\supp{\bw_j}\supseteq(\supp{\bb_j}\cup\ldots \cup\supp{\bb_1})$ as desired.
\end{proof}

\begin{lem}\label{lemmaw-span}
For $j=1,\dots,n$ we have $\gen{\bw_1,\dots,\bw_j}=\gen{\bb_1-\ba_1,\dots,\bb_j-\ba_j}$.
\end{lem}
\begin{proof}
Induction on $j$.
\end{proof}

\begin{prop}\label{collapse-redundant}
If a vertex $v\in V(\Gamma)$ has a set of collapsing data, then it is redundant.
\end{prop}
\begin{proof}
Suppose $e_1,\dots,e_n,e$ is a set of collapsing data for $v$, where $e_j$ goes from $(v,\ba_j)$ to $(v,\bb_j)$ for $j=1,\dots,n$ and where $e$ goes from $(v,\bc)$ to $(u,\bd)$. Then we choose a set of collapsing constants $k_1,\dots,k_n$ and we use them to construct a path going from $(v,\mathbf{0})$ to a vertex outside $\pA_v$, as follows. We start at $(v,\mathbf{0})$, we follow $k_n\dots k_1$ times $e_1$ and we arrive at $(v,k_n\dots k_1\bw_1)$ (which belongs to $\pA$ by Lemma \ref{lemmaw}). Then we follow $k_n\dots k_2$ times $e_2$ and we arrive at $(v,k_n\dots k_2\bw_2)$ (which belongs to $\pA$ by Lemma \ref{lemmaw} again). Then we follow $k_n\dots k_3$ times $e_3$ and we arrive at $(v,k_n\dots k_3\bw_3)$. And so on, until we arrive at $(v,k_n\bw_n)$. Since by Lemma \ref{lemmaw} $k_n\bw_n\ge\bc$, we can now use $e$ to move outside $\pA_v$, and the statement follows.
\end{proof}

\subsection{Projection move}\label{sec:projection}

In this section, we introduce a new move, the projection move, which removes a redundant vertex. The specifics of the move will become clearer in the proof of Proposition \ref{projection}, where we describe it as a sequence of basic moves (that preserve isomorphism).

\begin{defn}[Projection move]\label{defn:Projection_move}
Let $(\Gamma,\psi)$ be a GBS graph and let $\Lambda$ be its affine representation. Let $v\in V(\Gamma)$ be a redundant vertex and let $e_1,\dots,e_n,e$ and $k_1,\dots,k_n$ be a set of collapsing data and a set of collapsing constants. Following the notation of {\rm Definition \ref{def:collapsing-data}}, for $j=1,\dots,n$ we say that $e_j$ goes from $(v,\ba_j)$ to $(v,\bb_j)$, and $e$ goes from $(v,\bc)$ to $(u,\bd)$; moreover, we define $\bw_1,\dots,\bw_n$ as in {\rm Lemma \ref{lemmaw}}, and we call $\bw=k_n\bw_n-\bc\ge\mathbf{0}$.

We substitute the edges $e_1,\dots,e_n$ with edges $e_1',\dots,e_n'$ as follows:
$$
\begin{cases}
(v,\ba_1)\edge  (v,\bb_1)\\
(v,\ba_2)\edge  (v,\bb_2)\\
(v,\ba_3)\edge  (v,\bb_3)\\
\dots \\
(v,\ba_n)\edge  (v,\bb_n)
\end{cases}
\quad\text{becomes}\qquad
\begin{cases}
(u,\bd)\edge  (u,\bd+\bw_n)\\
(u,\bd+\bw+\ba_2)\edge  (u,\bd+\bw+\ba_2+\bw_1)\\
(u,\bd+\bw+\ba_3)\edge  (u,\bd+\bw+\ba_3+\bw_2)\\
\dots \\
(u,\bd+\bw+\ba_n)\edge  (u,\bd+\bw+\ba_n+\bw_{n-1})
\end{cases}
$$

We remove the edge $e$ from the affine representation. We identify $\pA_v$ with a subset of $\pA_u$ by identifying the point $(v,\bx)$ with $(u,\bd+\bw+\bx)$ for all $\bx\in\pA$ and use this identification to transfer the remaining edges in $\pA_v$ to $\pA_u$. The other $\pA_w$ for $w\ne u, v$ in $\Lambda$ remain unchanged. We obtain a new graph $\Lambda'$ which is the affine representation of a GBS graph $(\Gamma',\psi')$. We say that $(\Gamma',\psi')$ is obtained from $(\Gamma,\psi)$ by a \textbf{projection move} {\rm(}relative to the redundant vertex $v$, the collapsing data $e_1,\dots,e_n,e$ and the collapsing constants $k_1,\dots,k_n${\rm)}.
\end{defn}

\begin{prop}\label{projection}
Every projection move can be written as a composition of edge sign-changes, slides, swaps, and inductions, followed by an elementary contraction.
\end{prop}
\begin{proof}
The proof is illustrated in Figure \ref{fig:projection-move}. We first focus on the edge $e_1,e_2$ inside $\pA_v$:
$$
\begin{cases}
\ba_1\edge  \ba_1+\bw_1\\
\ba_2\edge  \bb_2\\
\ba_3\edge  \bb_3\\
\dots 
\end{cases}
\xrightarrow{\text{slides}}\quad
\begin{cases}
\ba_1\edge  \ba_1+\bw_1\\
\ba_2\edge  \bb_2+k_1\bw_1\\
\ba_3\edge  \bb_3\\
\dots 
\end{cases}
\xrightarrow{\text{swap}}\quad
\begin{cases}
\ba_1\edge  \ba_1+\bw_2\\
\ba_2\edge  \ba_2+\bw_1\\
\ba_3\edge  \bb_3\\
\dots 
\end{cases}
$$
We now focus on the edges $e_1,e_3$:
$$
\begin{cases}
\ba_1\edge  \ba_1+\bw_2\\
\ba_2\edge  \ba_2+\bw_1\\
\ba_3\edge  \bb_3\\
\dots 
\end{cases}
\xrightarrow{\text{slides}}\quad
\begin{cases}
\ba_1\edge  \ba_1+\bw_2\\
\ba_2\edge  \ba_2+\bw_1\\
\ba_3\edge  \bb_3+k_2\bw_2\\
\dots 
\end{cases}
\xrightarrow{\text{swap}}\quad
\begin{cases}
\ba_1\edge  \ba_1+\bw_3\\
\ba_2\edge  \ba_2+\bw_1\\
\ba_3\edge  \bb_3+\bw_2\\
\dots 
\end{cases}
$$
We continue like this by induction until we arrive at
$$
\begin{cases}
\ba_1\edge  \ba_1+\bw_n\\
\ba_2\edge  \ba_2+\bw_1\\
\ba_3\edge  \ba_3+\bw_2\\
\dots \\
\ba_n\edge  \ba_n+\bw_{n-1}
\end{cases}
$$
We then apply an induction move, and we translate by $\bw=k_n\bw_n-\bc$ all the endpoints of edges that belong to $\pA_v$, except for the endpoints of the edge $(v,\ba_1)\edge  (v,\ba_1+\bw_n)$. The edge $e$ becomes $(v,k_n\bw_n)\edge  (u,\bd)$ and we slide it to change it into $(v,\mathbf{0})\edge  (u,\bd)$. Finally we use this edge $(v,\mathbf{0})\edge  (u,\bd)$ in order to move everything else from $\pA_v$ to $\pA_u$. The edges $e_1,\dots,e_n$ become
$$
\begin{cases}
(u,\bd+\ba_1)\edge  (u,\bd+\ba_1+\bw_n)\\
(u,\bd+\bw+\ba_2)\edge  (u,\bd+\bw+\ba_2+\bw_1)\\
(u,\bd+\bw+\ba_3)\edge  (u,\bd+\bw+\ba_3+\bw_2)\\
\dots \\
(u,\bd+\bw+\ba_n)\edge  (u,\bd+\bw+\ba_n+\bw_{n-1})
\end{cases}
$$
and any other endpoint $(v,\bx)$ of any other edge is moved to $(u,\bd+\bw+\bx)$. Since $\ba_1$ is either $(0,0,0,\dots)$ or $(1,0,0,\dots)$, we can apply an edge sign-change to remove the $\ba_1$ summands in the endpoints of the first edge. To conclude, we remove the vertex $v$ and the edge $e$ with an elementary contraction move.
\end{proof}

\begin{figure}[H]
	\centering

    \includegraphics[width=\textwidth]{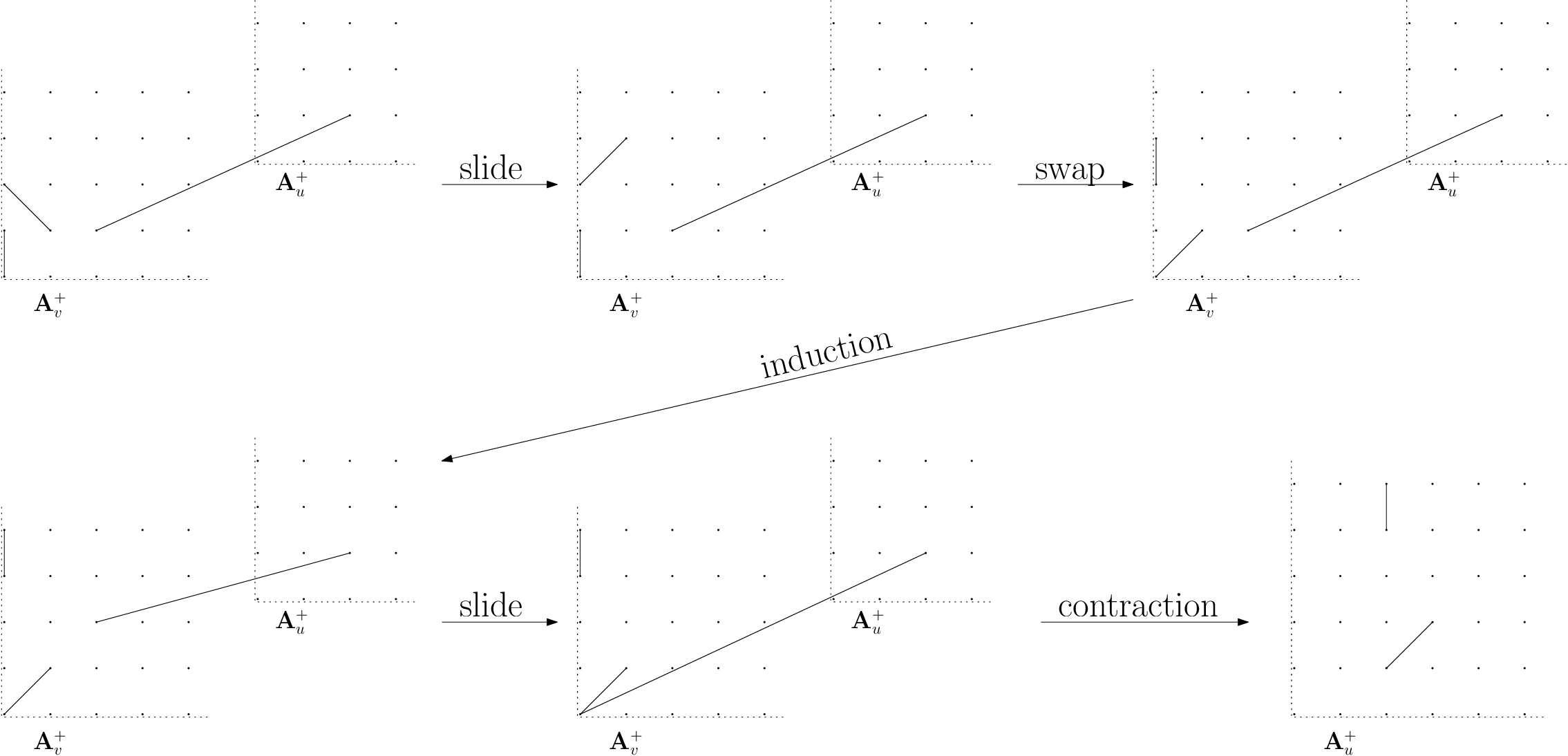}
    
	\caption{How to remove a redundant vertex using a set of collapsing data. From the last graph in the sequence, we can just slide everything out of the redundant vertex, and then perform a collapse.}
	\label{fig:projection-move}
\end{figure}

\subsection{The coarse projection map}

The goal of this section is to show that the projection move does not depend on the choice of the collapsing data and collapsing constants, up to edge sign-changes, slides, swaps, and connections. The proof relies on Theorem \ref{thm:controlled}.

\begin{prop}[Independence of the collapsing constants]\label{indep-constants}
Let $(\Gamma,\psi)$ be a GBS graph and let $v$ be a redundant vertex. Let $e_1,\dots,e_n,e$ be a collapsing data for $v$. Let $k_1,\dots,k_n$ and $k_1',\dots,k_n'$ be two sets of collapsing constants and let $(\Delta,\phi)$ and $(\Delta',\phi')$ be the GBS graphs obtained after applying the two projection moves with collapsing constants $k_1,\dots,k_n$ and $k_1',\dots,k_n'$, respectively. Then, there is a sequence of slides and swaps going from $(\Delta',\phi')$ to $(\Delta,\phi)$.
\end{prop}
\begin{proof}
We follow the notation of Definition \ref{def:collapsing-data} and of Lemma \ref{lemmaw}. For $j=1,\dots,n$, we say that the edge $e_j$ goes from $(v,\ba_j)$ to $(v,\bb_j)$, and that the edge $e$ goes from $(v,\bc)$ to $(u,\bd)$. Let $\bw_1=(\bb_1-\ba_1)$ and $\bw_j=(\bb_j-\ba_j)+k_{j-1}\bw_{j-1}$ and let $\bw=k_n\bw_n-\bc$. Similarly, let $\bw_1'=(\bb_1-\ba_1)$ and $\bw_j'=(\bb_j-\ba_j)+k_{j-1}'\bw_{j-1}'$ and let $\bw'=k_n'\bw_n'-\bc$.

We start with the graph $(\Delta',\psi')$. Since $(\Delta',\psi')$ has been obtained from a projection move, see Definition \ref{defn:Projection_move}, inside $\pA_u$ we have the edges
$$
\begin{cases}
\bd\edge  \bd+\bw_n'\\
\bd+\bw'+\ba_2\edge  \bd+\bw'+\ba_2+\bw_1'\\
\bd+\bw'+\ba_3\edge  \bd+\bw'+\ba_3+\bw_2'\\
\dots \\
\bd+\bw'+\ba_n\edge  \bd+\bw'+\ba_n+\bw_{n-1}'.
\end{cases}
$$
By Lemma \ref{lemmaw-span}, we have that
$$\gen{\bw_1',\dots,\bw_n'}=\gen{\bb_1-\ba_1,\dots,\bb_n-\ba_n}=\gen{\bw_1,\dots,\bw_n}$$
and
$$\bw'-\bw\in\gen{\bw_1,\dots,\bw_n}.$$
Therefore, we can apply Theorem \ref{thm:controlled}: using a sequence of slides and swaps we obtain the configuration
$$
\begin{cases}
\bd\edge  \bd+\bw_n\\
\bd+\bw+\ba_2\edge  \bd+\bw+\ba_2+\bw_1\\
\bd+\bw+\ba_3\edge  \bd+\bw+\ba_3+\bw_2\\
\dots \\
\bd+\bw+\ba_n\edge  \bd+\bw+\ba_n+\bw_{n-1}.
\end{cases}
$$
The projection move from $\Gamma$ to $\Delta'$ identified the point $(v,\bx)$ with $(u,\bd+\bw'+\bx)$ for all $\bx\in\pA$. However, since $\bw-\bw'\in\gen{\bw_n,\dots,\bw_1}$, we can use slide moves to identify $(v,\bx)$ with $(u,\bd+\bw+\bx)$ instead. The statement follows.
\end{proof}

\begin{lem}[Independence of edge ordering]\label{indep-swapedges}
Let $(\Gamma,\psi)$ be a GBS graph and let $v$ be a redundant vertex. Suppose that we are given two sets of collapsing data
$$e_1,\dots,e_{m-1},e_m,e_{m+1},e_{m+2},\dots,e_n,e$$
and
$$e_1,\dots,e_{m-1},e_{m+1},e_m,e_{m+2},\dots,e_n,e$$
for $v$. Let $(\Delta,\phi)$ and $(\Delta',\phi')$ be the GBS graphs obtained by applying two projection moves with the data above respectively. Then, we can go from $(\Delta',\phi')$ to $(\Delta,\phi)$ using a sequence of edge sign-changes, slides, and swaps.
\end{lem}
\begin{proof}
We follow the notation of Definition \ref{def:collapsing-data} and of Lemma \ref{lemmaw}. For $j=1,\dots,n$, we say that $e_j$ goes from $(v,\ba_j)$ to $(v,\bb_j)$, and that $e$ goes from $(v,\bc)$ to $(u,\bd)$. Let $k_1,\dots,k_n$ be a sequence of collapsing constants for $e_1,\dots,e_n,e$, and thus also for $e_1,\dots,e_{m+1},e_m,\dots,e_n,e$. Let $\bw_1=(\bb_1-\ba_1)$ and $\bw_j=(\bb_j-\ba_j)+k_{j-1}\bw_{j-1}$ and let $\bw=k_n\bw_n-\bc$. Define $\bw_m'=(\bb_{m+1}-\ba_{m+1})-k_{m-1}\bw_{m-1}$ and $\bw_{m+1}'=(\bb_m-\ba_m)+k_m\bw_m'$; for $j=m+1,\dots,n$ define $\bw_j'=(\bb_j-\ba_j)+k_{j-1}\bw_{j-1}'$ and set $\bw'=k_n\bw_n'-\bc$.

In $\pA_u$ inside $\Delta'$, we have the edges
$$
\begin{cases}
\bd\edge  \bd+\bw_n'\\
\bd+\bw'+\ba_2\edge  \bd+\bw'+\ba_2+\bw_1\\
\dots \\
\bd+\bw'+\ba_{m-1}\edge  \bd+\bw'+\ba_{m-1}+\bw_{m-2}\\
\bd+\bw'+\ba_{m+1}\edge  \bd+\bw'+\ba_{m+1}+\bw_{m-1}\\
\bd+\bw'+\ba_m\edge  \bd+\bw'+\ba_m+\bw_m'\\
\bd+\bw'+\ba_{m+2}\edge  \bd+\bw'+\ba_{m+2}+\bw_{m+1}'\\
\dots \\
\bd+\bw'+\ba_n\edge  \bd+\bw'+\ba_n+\bw_{n-1}'.
\end{cases}
$$
By Lemma \ref{lemmaw-span}, we have
$$\gen{\bw_1',\dots,\bw_n'}=\gen{\bb_1-\ba_1,\dots,\bb_n-\ba_n}=\gen{\bw_1,\dots,\bw_n}$$
and
$$\bw'-\bw\in\gen{\bw_1,\dots,\bw_n}.$$
Therefore, we can apply Theorem \ref{thm:controlled}: using a sequence of slides and swaps, we obtain the configuration
$$
\begin{cases}
\bd\edge  \bd+\bw_n\\
\bd+\bw+\ba_2\edge  \bd+\bw+\ba_2+\bw_1\\
\dots \\
\bd+\bw+\ba_{m-1}\edge  \bd+\bw+\ba_{m-1}+\bw_{m-2}\\
\bd+\bw+\ba_{m+1}\edge  \bd+\bw+\ba_{m+1}+\bw_{m}\\
\bd+\bw+\ba_m\edge  \bd+\bw+\ba_m+\bw_{m-1}\\
\bd+\bw+\ba_{m+2}\edge  \bd+\bw+\ba_{m+2}+\bw_{m+1}\\
\dots \\
\bd+\bw+\ba_n\edge  \bd+\bw+\ba_n+\bw_{n-1}.
\end{cases}
$$
Notice that, in the particular case $m=1$, we have that both $\ba_1$ and $\ba_2$ are equal to either $(0,0,0,\dots)$ or $(1,0,0,\dots)$, and we need to change $\ba_1$ into $\ba_2$ in the second edge of the above list, and for this we might need to perform an edge sign-change. Finally, the projection move from $\Gamma$ to $\Delta'$ had identified the point $(v,\bx)$ with $(u,\bd+\bw'+\bx)$ for all $\bx\in\pA$. But since $\bw-\bw'\in\gen{\bw_n,\dots,\bw_1}$, we can use slide moves to identify $(v,\bx)$ with $(u,\bd+\bw+\bx)$ instead. The statement follows.
\end{proof}

\begin{lem}[Independence by extension of the collapsing data]\label{indep-uselessedge}
Let $(\Gamma,\psi)$ be a GBS graph and let $v$ be a redundant vertex. Let $e_1,\dots,e_n,e$ and $e_1,\dots,e_n,e_{n+1},e$ be two sets of collapsing data for $v$ and let $(\Delta,\phi)$ and $(\Delta',\phi')$ be the GBS graphs obtained by applying two projection moves with the data above respectively. Then, we can go from $(\Delta',\phi')$ to $(\Delta,\phi)$ using a sequence of slides and swaps.
\end{lem}
\begin{proof}
Choose a set $k_1,\dots,k_n,k_{n+1}$ of collapsing constants for $e_1,\dots,e_n,e_{n+1},e$; in particular we have that $k_1,\dots,k_{n-1},k_n$ is also a set of collapsing constants for $e_1,\dots,e_n,e$. Following the notation of Definition \ref{def:collapsing-data}, for $j=1,\dots,n+1$, we say that $e_j$ goes from $(v,\ba_j)$ to $(v,\bb_j)$, and $e$ goes from $(v,\bc)$ to $(u,\bd)$. Let $\bw_1=(\bb_1-\ba_1)$ and $\bw_j=(\bb_j-\ba_j)+k_{j-1}\bw_{j-1}$ for $j=1,\dots,n+1$, as in Lemma \ref{lemmaw}. Let $\bw'=k_{n+1}\bw_{n+1}-\bc$ and let $\bw=k_n\bw_n-\bc$.

In $\pA_u$ inside $\Delta'$, we have the edges
$$
\begin{cases}
\bd\edge  \bd+\bw_{n+1}\\
\bd+\bw'+\ba_2\edge  \bd+\bw'+\ba_2+\bw_1\\
\bd+\bw'+\ba_3\edge  \bd+\bw'+\ba_3+\bw_2\\
\dots \\
\bd+\bw'+\ba_n\edge  \bd+\bw'+\ba_n+\bw_{n-1}\\
\bd+\bw'+\ba_{n+1}\edge  \bd+\bw'+\ba_{n+1}+\bw_n.
\end{cases}
$$
We observe that $\bw'-\bw=k_{n+1}\bw_{n+1}-k_n\bw_n\in\gen{\bw_1,\dots,\bw_{n+1}}$ and thus, using slide moves and self-slide moves, see Lemma \ref{self-slide}, we obtain
$$
\begin{cases}
\bd\edge  \bd+\bw_{n+1}\\
\bd+\bw+\ba_2\edge  \bd+\bw+\ba_2+\bw_1\\
\bd+\bw+\ba_3\edge  \bd+\bw+\ba_3+\bw_2\\
\dots \\
\bd+\bw+\ba_n\edge  \bd+\bw+\ba_n+\bw_{n-1}\\
\bd+\bw+\ba_{n+1}\edge  \bd+\bw+\ba_{n+1}+\bw_n.
\end{cases}
$$
We now look at the first and the last edge
$$
\begin{cases}
\bd\edge  \bd+\bw_{n+1}\\
\bd+\bw+\ba_{n+1}\edge  \bd+\bw+\ba_{n+1}+\bw_n
\end{cases}
$$
and we perform a swap move to obtain
$$
\begin{cases}
\bd\edge  \bd+\bw_n\\
\bd+\bw+\ba_{n+1}\edge  \bd+\bw+\ba_{n+1}+\bw_{n+1}
\end{cases}
=\ 
\begin{cases}
\bd\edge  \bd+\bw_n\\
\bd+\bw+\ba_{n+1}\edge  \bd+\bw+\bb_{n+1}+k_n\bw_n.
\end{cases}
$$
We next perform slide moves to obtain
$$
\begin{cases}
\bd\edge  \bd+\bw_n\\
\bd+\bw+\ba_{n+1}\edge  \bd+\bw+\bb_{n+1}.
\end{cases}
$$
Finally, if the collapse from $\Gamma$ to $\Delta'$ identified each point $(v,\bx)$ of $\pA_v$ with $(u,\bd+\bw'+\bx)$, we can now identify it with $(u,\bd+\bw+\bx)$ instead using slide moves. The statement follows.
\end{proof}

\begin{lem}[Independence of the connecting edge]\label{indep-lastedge}
Let $(\Gamma,\psi)$ be a GBS graph and let $v$ be a redundant vertex. Let $e_1,\dots,e_n,e$ and $e_1,\dots,e_n,e'$ be two sets of collapsing data for $v$ and let $(\Delta,\phi)$ and $(\Delta',\phi')$ be the GBS graphs obtained by applying two projection moves with the data above respectively. Then we can go from $(\Delta',\phi')$ to $(\Delta,\phi)$ with a sequence of slides, swaps, and connections.
\end{lem}
\begin{proof}
Choose a set of collapsing constants $k_1,\dots,k_n$. Following the notation of Definition \ref{def:collapsing-data}, for $j=1,\dots,n$, we say that $e_j$ goes from $(v,\ba_j)$ to $(v,\bb_j)$, and let $\bw_1=(\bb_1-\ba_1)$ and $\bw_j=(\bb_j-\ba_j)+k_{j-1}\bw_{j-1}$ be as in Lemma \ref{lemmaw}. We say that $e$ goes from $(v,\bc)$ to $(u,\bd)$, and we set $\bw=k_n\bw_n-\bc$. Similarly, we say that $e'$ goes from $(v,\bc')$ to $(u',\bd')$, and we set $\bw'=k_n\bw_n'-\bc'$.

When we apply the projection move to obtain $\Delta'$, the edge $(v,\bc)\edge  (u,\bd)$ becomes $(u',\bd'+\bw'+\bc)\edge (u,\bd)$. We consider the affine representation of $\Delta'$ and perform the following connection move:
$$\begin{cases}
(u',\bd')\edge  (u',\bd'+\bw_n)\\
(u,\bd)\edge (u',\bd'+\bw'+\bc)
\end{cases}
\xrightarrow{\text{connection}}\quad
\begin{cases}
(u',\bd')\edge (u,\bd+\bw+\bc')\\
(u,\bd)\edge (u,\bd+\bw_n).
\end{cases}$$
We now use the edges $(u',\bd')\edge (u,\bd+\bw+\bc')$ and $(u,\bd)\edge (u,\bd+\bw_n)$ to slide the other edges from
$$
\begin{cases}
(u',\bd'+\bw'+\ba_2)\edge  (u',\bd'+\bw'+\ba_2+\bw_1)\\
(u',\bd'+\bw'+\ba_3)\edge  (u',\bd'+\bw'+\ba_3+\bw_2)\\
\dots \\
(u',\bd'+\bw'+\ba_n)\edge  (u',\bd'+\bw'+\ba_n+\bw_{n-1})
\end{cases}
$$
to
$$
\begin{cases}
(u,\bd+\bw+\bc'+\bw'+\ba_2)\edge  (u,\bd+\bw+\bc'+\bw'+\ba_2+\bw_1)\\
(u,\bd+\bw+\bc'+\bw'+\ba_3)\edge  (u,\bd+\bw+\bc'+\bw'+\ba_3+\bw_2)\\
\dots \\
(u,\bd+\bw+\bc'+\bw'+\ba_n)\edge  (u,\bd+\bw+\bc'+\bw'+\ba_n+\bw_{n-1})
\end{cases}
$$
and then to
$$
\begin{cases}
(u,\bd+\bw+\ba_2)\edge  (u,\bd+\bw+\ba_2+\bw_1)\\
(u,\bd+\bw+\ba_3)\edge  (u,\bd+\bw+\ba_3+\bw_2)\\
\dots \\
(u,\bd+\bw+\ba_n)\edge  (u,\bd+\bw+\ba_n+\bw_{n-1}).
\end{cases}
$$
With analogous slide moves, if the collapse from $\Gamma$ to $\Delta'$ identified a point $(v,\bx)$ of $\pA_v$ with $(u',\bd'+\bw'+\bx)$, we can now identify it with $(u,\bd+\bw+\bx)$ instead. The statement follows.
\end{proof}

\begin{prop}[Independence of the collapsing data]\label{prop:indep-data}
Let $(\Gamma,\psi)$ be a GBS graph and let $v$ be a redundant vertex. Let $e_1,\dots,e_n,e$ and $e_1',\dots,e_m',e'$ be two sets of collapsing data for $v$ and let $(\Delta,\phi)$ and $(\Delta',\phi')$ be the GBS graphs obtained by applying two projection moves with the data above respectively. Then, there is a sequence of edge sign-changes, slides, swaps, and connections going from $(\Delta',\phi')$ to $(\Delta,\phi)$.
\end{prop}
\begin{proof}
We start with the set of collapsing data $e_1,\dots,e_n,e$. With an iterated application of Lemma \ref{indep-uselessedge}, we can add the edges $e_1,\dots,e_m$ at the sequence to obtain the set of collapsing data $e_1,\dots,e_n,e_1',\dots,e_m',e$. By Lemma \ref{indep-lastedge}, we can replace $e$ by $e'$ to obtain $e_1,\dots,e_n,e_1',\dots,e_m',e'$. Using Lemma \ref{indep-swapedges}, we permute the vectors and obtain $e_1',\dots,e_m',e_1,\dots,e_n,e'$. Finally, by Lemma \ref{indep-uselessedge}, we can remove the edges $e_1,\dots,e_n$, and we remain with the set of collapsing data $e_1',\dots,e_m',e'$. The statement follows.
\end{proof}

Given a GBS graph $(\Gamma,\psi)$ and a redundant vertex $v$, we can define a GBS graph $(\Gamma',\psi')$ obtained from $(\Gamma,\psi)$ removing the vertex $v$, using the projection move described in Section \ref{sec:projection}. The process would require choosing a set of collapsing data and a set of collapsing constants, but the above Proposition \ref{prop:indep-data} states that different choices produce graphs $(\Gamma',\psi')$ that differ only by a sequence of edge sign-changes, slides, swaps, and connections. When we perform a projection move without making explicit the collapsing data or constants, we call it \emph{a coarse projection map}, and the resulting GBSs, \emph{projections}. In this terminology, two projections (that remove a redundant vertex $v$) differ only by a sequence of edge sign-changes, slides, swaps, and connections.

\subsection{Coarse projection map and moves}

The following Proposition \ref{prop:projection-of-moves} states essentially that the coarse projection map commutes with edge sign-changes, slides, swaps, and connections.

\begin{prop}[Coarse projection and moves essentially commute]\label{prop:projection-of-moves}
Let $(\Gamma_0,\psi_0)$ be a GBS graph and let $v$ be a redundant vertex. Let $(\Gamma_0',\psi_0')$ be the graph obtained by applying a coarse projection map. Let $(\Gamma_1,\psi_1)$ be obtained from $(\Gamma_0,\psi_0)$ by applying an edge sign-change, slide, swap or a connection move; let $v$ be the corresponding redundant vertex in $(\Gamma_1,\psi_1)$, and let $(\Gamma_1',\psi_1')$ be the graph obtained by applying a coarse projection map. 

Then $(\Gamma_1',\psi_1)$ can be obtained from $(\Gamma_0',\psi_0')$ using a sequence of edge sign-changes, slides, swaps, and connections.
\end{prop}
\begin{proof}
If we are performing an edge sign-change from $(\Gamma_0,\psi_0)$ to $(\Gamma_1,\psi_1)$, then $(\Gamma_0',\psi_0')$ and $(\Gamma_1',\psi_1')$ are related by edge sign-change moves only.

Suppose that we are performing a slide move changing one endpoint of an edge $d'$ using another edge $d$. If $d,d'$ are not involved in the collapsing data that we are using, then the slide move induces a slide move on the projections, and we are done. If $d'$ is not involved in the collapsing data, but $d$ is, then the slide move becomes a composition of slide moves (each of them moving the endpoint of $d'$). If $d'$ is involved in the collapsing data, then using Lemmas \ref{indep-uselessedge} and \ref{indep-swapedges}, we can add $d$ to the collapsing data and we can move it before $d$. Thus we can assume that we are using collapsing data $e_1,\dots,e_n,e$ and constants, and that we are sliding either one endpoint of $e$ using the edge $e_\ell$, or one endpoint of $e_m$ using the edge $e_\ell$ for some $\ell<m$. In both cases, the statement is immediate using Theorem \ref{thm:controlled}.

Suppose that we are performing a swap move on the two edges $d$ and $d'$. If $d,d'$ are not involved in the collapsing data that we are using, then the swap move induces a swap move on the projections, and we are done. If one of $d,d'$ is involved in the collapsing data but the other is not, then we can use Lemmas \ref{indep-uselessedge} and \ref{indep-swapedges} to add the other one to the collapsing data, and to move the two edges $d,d'$ until they are adjacent in the collapsing data. Thus we now assume that we are using collapsing data $e_1,\dots,e_n,e$ and that we are performing the swap move on the two edges $e_m,e_{m+1}$. In this case, the statement follows from Theorem \ref{thm:controlled}.

Suppose that we are performing a connection move on the two edges $d,e$. If $d$ and $e$ are not involved in the collapsing data that we are using, then the connection move induces a connection move on the projections, and we are done. If $d$ is involved in the collapsing data but $e$ is not, then we can add $e$ to the collapsing data using Lemma \ref{indep-lastedge}. If $e$ is involved in the collapsing data but $d$ is not, then we can add $d$ to the collapsing data using Lemma \ref{indep-uselessedge}. If $d,e$ are both involved in the collapsing data that we are using, then we can make them consecutive in the collapsing data, using Lemma \ref{indep-swapedges}. Thus we can assume that we are using collapsing data $e_1,\dots,e_n,e$ and that we are performing the connection on either $e_m,e_{m+1}$ for some $1\le m\le n-1$, or on $e_n,e$. If we are performing the connection on $e_m,e_{m+1}$ for some $1\le m\le n-1$, then the statement follows from Theorem \ref{thm:controlled}. If we are performing the connection on $e_n,e$ then the statement  follows again from Theorem \ref{thm:controlled}.
\end{proof}

\subsection{Totally reduced GBS graphs}

We now introduce the notion of \textit{totally reduced} GBS graph, see Definition \ref{def:totally-reduced}. We point out that Condition \ref{itm:fully-reduced} of Definition \ref{def:totally-reduced} coincides with the notion of \textit{fully reduced} GBS graph introduced by Forester \cite{For06} (and it is also equivalent to requiring that the graph does not contain any redundant vertex). However, we slightly strengthen the definition of fully reduced graph by also requiring the (minor) technical Condition \ref{itm:controlling-edge}, which will be necessary to make the main Theorem \ref{thm:sequence-new-moves} hold. In any case, we prove in Proposition \ref{prop:compute-totally-reduced}, that there is an algorithm that brings any GBS graph to a totally reduced one.

\begin{defn}[Totally reduced]\label{def:totally-reduced}
A GBS graph $(\Gamma,\psi)$ with affine representation $\Lambda$ is called \textbf{totally reduced} if the following two conditions hold:
\begin{enumerate}
\item\label{itm:fully-reduced} For every vertex $v\in V(\Gamma)$, if $(v,\mathbf{0})\cnj(u,\bb)$ for $u\in V(\Gamma)$ and $\bb\in\pA$, then $u=v$.
\item\label{itm:controlling-edge} For every vertex $v\in V(\Gamma)$, in $\Lambda$ there is an edge $(v,\mathbf{0})\edge(v,\bw)$ with $\bw\in\pA$ such that, if $(v,\mathbf{0})\cnj(v,\bb)$ for $\bb\in\pA$, then $\mathbf{0},\bw$ controls $\bb$.
\end{enumerate}
\end{defn}

\begin{prop}\label{prop:compute-totally-reduced}
Let $(\Gamma,\psi)$ be a GBS graph. Then, there is a sequence of edge sign-changes, elementary expansions, elementary contractions, and slides changing $(\Gamma,\psi)$ into a totally reduced GBS graph $(\Gamma',\psi')$. Moreover there is an algorithm that, given $(\Gamma,\psi)$, computes the sequence of moves and $(\Gamma',\psi')$.
\end{prop}
\begin{proof}
Let $\Lambda$ be the affine representation of $(\Gamma,\psi)$. Given a vertex $v\in V(\Gamma)$, by Proposition \ref{compute-conjugacy-class}, we can algorithmically compute the conjugacy class of $(v,\mathbf{0})$. In particular, we can determine whether $v$ is redundant or not.

If $v$ is redundant, then using Lemma \ref{compare-conjugacy-class}, we can compute a set of collapsing data and constants for $v$. Using Proposition \ref{projection} and Lemmas \ref{induction-from-others}, \ref{swap-from-others}, and \ref{connection-from-others}, we can thus produce a sequence of edge sign-changes, slides, and elementary expansions and contractions, that perform the projection move of Section \ref{sec:projection}, removing the vertex $v$ and decreasing strictly the number of vertices of $\Gamma$.

If $v$ is not redundant, then by Proposition \ref{compute-conjugacy-class}, we can find a point $(v,\bw)\cnj (v,\mathbf{0})$, for $\bw\in\pA$, with the following property: for all $(v,\bw')\cnj (v,\mathbf{0})$ with $\bw'\in\pA$, we have $\supp{\bw'}\subseteq\supp{\bw}$. By Lemma \ref{compare-conjugacy-class}, we can algorithmically compute a path from $(v,\mathbf{0})$ to $(v,\bw)$, and let $e_1,\dots,e_n$ be the edges that appear in such path, in this order (but without repetitions). We now perform the procedure described in the proof of Proposition \ref{projection} and obtain an edge satisfying Condition \ref{itm:controlling-edge} of Definition \ref{def:totally-reduced}.
\end{proof}

\section{Sequences of moves}\label{sec:sequences-of-moves}

By Forester's results \cite{For02, For03} (see Theorems \ref{thm:Forester} and \ref{thm:GBS-JSJ}), if two GBS graphs $(\Gamma,\psi),(\Delta,\phi)$ represent GBS graphs of groups with isomorphic fundamental group (different from $\bbZ,\bbZ^2$ or the Klein bottle group), then it is possible to go from one to the other by sign-changes, elementary expansion, elementary contractions, and slides. However, a priori, there is no bound on the number of vertices and edges in the graphs that appear in this sequence.

In this section, we address this issue in Theorem \ref{thm:sequence-new-moves}. The theorem states that if two (totally reduced) GBS graphs represent isomorphic groups, then it is possible to go from one to the other by a sequence of sign-changes, inductions, slides, swaps, and connections. The main feature of the theorem is that the set of moves allowed in the sequence does not change the number of vertices of the graphs. This result allows us to (essentially) reduce the isomorphism problem for GBSs to the case of one-vertex graphs of groups. Furthermore, it has proven to be very helpful in deciding the isomorphism between concrete examples and we expect it to be key when addressing the general isomorphism problem for GBSs. 

The proof of Theorem \ref{thm:sequence-new-moves} is based on a systematic elimination of redundant vertices from sequences of moves. More precisely, we take a sequence of sign-changes, slides, and elementary expansions and contractions between the two groups, we choose a redundant vertex that appears in some of the graphs of the sequence, and we eliminate it by applying the projection move described in Section \ref{sec:projection}. This elimination procedure happens at the same time on all the graphs of the sequence: in order to make this formal, we need to introduce the notion of a \textit{story of a vertex along a sequence of moves}. Then we take advantage of Propositions \ref{prop:indep-data} and \ref{prop:projection-of-moves} and argue that, after the elimination procedure, we obtain another sequence of moves between the groups, but with fewer redundant vertices involved.

\subsection{Story of a vertex along a sequence of moves}\label{sec:story-of-a-vertex}

Let $(\Gamma,\psi)$ be a GBS graph. We first recall the behaviour of the moves with respect to the number of vertices in the graph. From the definition of the moves, we have that edge sign-changes, slides, swaps, and connections do not change the set of vertices of the graph. This means that, if $(\Gamma',\psi')$ is obtained from $(\Gamma,\psi)$ by an edge sign-change, a slide, a swap, or a connection, then we have $V(\Gamma')=V(\Gamma)$ and clearly, this extends to a sequence of these moves, i.e. if
$$(\Gamma_0,\psi_0),\dots,(\Gamma_N,\psi_N)$$
is a sequence of GBS graphs such that $(\Gamma_\beta,\psi_\beta)$ is obtained from $(\Gamma_{\beta-1},\psi_{\beta-1})$ by an edge sign-change, slide swap or connection, for $\beta=1,\dots,N$, then we have $V(\Gamma_0)=V(\Gamma_1)=\dots =V(\Gamma_N)$. 

On the other hand, if $(\Gamma',\psi')$ is obtained from $(\Gamma,\psi)$ by an elementary expansion, adding a valence-$1$ vertex $v$ and an edge, then we have $V(\Gamma')=V(\Gamma)\cup\{v\}$. Similarly, if $(\Gamma',\psi')$ is obtained from $(\Gamma,\psi)$ by an elementary contraction, eliminating a valence-$1$ vertex $v$ and an edge, then $V(\Gamma')=V(\Gamma)\setminus\{v\}$.

Similarly, if $(\Gamma',\psi')$ is obtained from $(\Gamma,\psi)$ by applying the projection move of Section \ref{sec:projection}, eliminating a redundant vertex $v$ from the graph, then we have $V(\Gamma')=V(\Gamma)\setminus\{v\}$. According to Proposition \ref{prop:indep-data}, a different choice of collapsing data and collapsing constants can lead to a different graph $(\Gamma'',\psi'')$, but in that case $(\Gamma',\psi')$ and $(\Gamma'',\psi'')$ are related by a sequence of edge sign-changes, slides, swaps, and connections; as argued above, such sequence of moves induces an equality $V(\Gamma')=V(\Gamma'')$: this equality agrees with the equality induced by the inclusions $V(\Gamma'),V(\Gamma'')\subseteq V(\Gamma)$.

\begin{lem}\label{powers-along-story}
Let
$$(\Gamma_0,\psi_0),\dots,(\Gamma_N,\psi_N)$$
be a sequence of GBS graphs such that $(\Gamma_\beta,\psi_\beta)$ is obtained from $(\Gamma_{\beta-1},\psi_{\beta-1})$ by an edge sign-change, elementary expansion, elementary contraction, slide, swap, connection, for $\beta=1,\dots,N$. Let $u,v$ be two vertices that belong to both $V(\Gamma_0)$ and $V(\Gamma_N)$. Then the following are equivalent:
\begin{enumerate}
\item In the fundamental group of the GBS graph of groups associated with $(\Gamma_0,\psi_0)$, $u$ is conjugate to a power of $v$.
\item In the fundamental group of the GBS graph of groups associated with $(\Gamma_N,\psi_N)$, $u$ is conjugate to a power of $v$.
\end{enumerate}
\end{lem}
\begin{proof}
Induction on $N\ge0$.
\end{proof}

\subsection{Ignoring elementary contraction moves}

The goal of this section is to prove Proposition \ref{no-contractions}, which is used to remove elementary contraction moves from the sequences.

\begin{defn}[Elementary subgraph]
Let $(\Gamma,\psi)$ be a GBS graph. An \textbf{elementary subgraph} of $(\Gamma,\psi)$ is a GBS graph $(\Delta,\phi)$ such that
\begin{enumerate}
\item $\Delta$ is a subgraph of $\Gamma$ and $\phi$ is the restriction of $\psi$.
\item The inclusion between the GBS graphs induces an isomorphism between the fundamental groups of the corresponding GBS graphs of groups.
\end{enumerate}
\end{defn}

\begin{lem}\label{elementary-subgraph1}
If $(\Delta,\phi)$ is an elementary subgraph of $(\Gamma,\psi)$, then $\rank{\Gamma}=\rank{\Delta}$.
\end{lem}

\begin{proof}
Let $\cG,\cH$ be the GBS graph of groups corresponding to $(\Gamma,\psi),(\Delta,\phi)$ respectively, and fix a basepoint in $\Delta$. If $\rank{\Gamma}>\rank{\Delta}$ then we can find a reduced path $(e_1,\dots,e_\ell)$ from the basepoint to itself which is not contained in $\Delta$. But then the word $e_1e_2\dots e_\ell$ represents an element of $\pi_1(\cG)$ which does not belong to $\pi_1(\cH)$ - a contradiction.
\end{proof}

\begin{lem}\label{elementary-subgraph2}
If $(\Delta,\phi)$ is an elementary subgraph of $(\Gamma,\psi)$, and $V(\Gamma)\supsetneq V(\Delta)$, then there is a valence-$1$ vertex $v\in V(\Gamma)\setminus V(\Delta)$ such that the unique edge $e$ with $\tau(e)=v$ satisfies $\psi(e)=\pm1$.
\end{lem}
\begin{proof}
Let $\cG,\cH$ be the GBS graph of groups corresponding to $(\Gamma,\psi),(\Delta,\phi)$ respectively, and fix a basepoint in $\Delta$. By Lemma \ref{elementary-subgraph1} we have $\rank{\Gamma}=\rank{\Delta}$, and thus if $V(\Gamma)\supsetneq V(\Delta)$ then $V(\Gamma)\setminus V(\Delta)$ has to contain a valence $1$ vertex $v$. Let $e\in E(\Gamma)$ be the unique edge with $\tau(e)=v$. If $\psi(\ol{e})\not=\pm1$, then we consider a reduced path $(e_1,\dots,e_\ell,e)$ from the basepoint to $v$: the word $e_1\dots e_\ell ev\ol{e}\ol{e}_\ell\dots \ol{e}_1$ represents an element of $\pi_1(\cG)$ which does not belong to $\pi_1(\cH)$, contradiction.
\end{proof}

\begin{lem}\label{elementary-subgraph3}
If $(\Delta,\phi)$ is an elementary subgraph {\rm(}and in particular an induced subgraph{\rm)} of $(\Gamma,\psi)$ and $V(\Gamma)=V(\Delta)$, then $(\Gamma,\psi)=(\Delta,\phi)$.
\end{lem}
\begin{proof}
Follows from Lemma \ref{elementary-subgraph1}.
\end{proof}

\begin{prop}[No elementary contractions]\label{no-contractions}
Let $(\Gamma_0,\psi_0),\dots,(\Gamma_N,\psi_N)$ be a sequence of GBS graphs, for $N\in\bbN$. For $\beta=1,\dots,N$, suppose that $(\Gamma_\beta,\psi_\beta)$ is obtained from $(\Gamma_{\beta-1},\psi_{\beta-1})$ by an edge sign-change, elementary expansion, elementary contraction, slide, swap or connection. Then there is a sequence of GBS graphs $(\Gamma_0',\psi_0'),\dots,(\Gamma_N',\psi_N')$ such that
\begin{enumerate}
\item For $\beta=1,\dots,N$ we have that $(\Gamma_\beta,\psi_\beta)$ is obtained from $(\Gamma_{\beta-1},\psi_{\beta-1})$ by the identity, an edge sign-change, an elementary expansion, a slide, a swap or a connection {\rm(}and notice that we do not allow for elementary contractions{\rm)}.
\item $(\Gamma_0',\psi_0')=(\Gamma_0,\psi_0)$ and $(\Gamma_\beta,\psi_\beta)$ is an elementary subgraph of $(\Gamma_\beta',\psi_\beta')$ for $\beta=1,\dots,N$.
\end{enumerate}
\end{prop}
\begin{proof}
Whenever we would perform an elementary contraction move to eliminate a vertex $v$, we instead do not perform any move, and we leave the vertex $v$ there. The vertex $v$ will appear as an extra vertex in all the subsequent graphs in the sequence (and it does not interfere with any subsequent edge sign-change, elementary expansion, slide, swap, or connection). This defines a new sequence of GBS graphs $(\Gamma_0',\psi_0'),\dots,(\Gamma_N',\psi_N')$ such that $(\Gamma_0',\psi_0')=(\Gamma_0,\psi_0)$ and such that $(\Gamma_\beta',\psi_\beta')$ contains $(\Gamma_\beta,\psi_\beta)$ as a subgraph for $\beta=1,\dots,N$. By induction in $\beta$ we have that $(\Gamma_\beta,\psi_\beta)$ is an elementary subgraph of $(\Gamma_\beta',\psi_\beta')$.
\end{proof}

\subsection{Removing redundant vertices from sequences of moves}

As we mentioned, if two GBS graphs $(\Gamma,\psi),(\Delta,\phi)$ represent GBS graphs of groups with isomorphic fundamental group (different from $\bbZ,\bbZ^2$ or the Klein bottle group), then it is possible to go from one to the other by sign-changes, elementary expansion, elementary contractions, and slides. However, the number of vertices and edges in the graphs in the sequence can grow. In Theorem \ref{thm:sequence-new-moves} we solve this problem.

\begin{thm}[Invariant number of vertices and edges]\label{thm:sequence-new-moves}
Let $(\Gamma,\psi),(\Delta,\phi)$ be totally reduced GBS graphs, and suppose that the corresponding GBS groups are isomorphic and different from $\bbZ,\bbZ^2,K$ {\rm(}the Klein bottle group{\rm)}. Then $\abs{V(\Gamma)}=\abs{V(\Delta)}$ and there is a sequence of slides, swaps, connections, sign-changes, and inductions going from $(\Delta,\phi)$ to $(\Gamma,\psi)$. Moreover, all the sign-changes and inductions can be performed at the beginning of the sequence.
\end{thm}

\begin{remark}
    In particular, since none of slide, swap, connection, sign-change and induction change the number of vertices/edges in the graph, in the sequence of moves provided by the above Theorem \ref{thm:sequence-new-moves}, all graphs will have the same number of vertices/edges.
\end{remark}

\begin{proof}
By Proposition \ref{no-contractions}, we can find a sequence
$$(\Gamma,\psi)=(\Gamma_0,\psi_0),(\Gamma_1,\psi_1),\dots,(\Gamma_N,\psi_N)$$
of GBS graphs such that $(\Gamma_N,\psi_N)$ contains $(\Delta,\phi)$ an elementary subgraph, and such that $(\Gamma_\beta,\psi_\beta)$ is obtained from $(\Gamma_{\beta-1},\psi_{\beta-1})$ by an edge sign-change, elementary expansion, slide, swap, or connection, for $\beta=1,\dots,N$; notice that elementary contractions are not allowed. In particular we have $V(\Gamma)=V(\Gamma_0)\subseteq V(\Gamma_1)\subseteq\dots \subseteq V(\Gamma_N)\supseteq V(\Delta)$.

Choose a vertex $u\in V(\Gamma_N)\setminus(V(\Gamma)\cup V(\Delta))$ and let $(\Gamma_\alpha,\psi_\alpha)$ be the first graph in the sequence where $u$ appears, for some $1\le\alpha\le N$. Notice that $u$ is redundant in $(\Gamma_\alpha,\psi_\alpha)$ (since it is added with an elementary expansion move), and thus by Lemma \ref{powers-along-story} it is redundant also in $(\Gamma_\beta,\psi_\beta)$ for all $\alpha\le\beta\le N$. We perform the projection move described in Section \ref{sec:projection} on the vertex $u$ at the same time along all the graphs $(\Gamma_\beta,\psi_\beta)$ with $\alpha\le\beta\le N$. By Proposition \ref{prop:projection-of-moves}, the projections of $(\Gamma_{\beta-1},\psi_{\beta-1}),(\Gamma_\beta,\psi_\beta)$ are related by a sequence of edge sign-changes, slides, swaps, and connections for all $\alpha\le\beta\le N$. Moreover we have that $(\Delta,\phi)$ is an elementary subgraph of the projection of $(\Gamma_N,\psi_N)$, since $u\not\in V(\Delta)$. Thus we get a new sequence
$$(\Gamma,\psi)=(\Gamma_0',\psi_0'),(\Gamma_1',\psi_1'),\dots,(\Gamma_{N'}',\psi_{N'}')$$
such that $(\Gamma_{N'}',\psi_{N'}')$ contains $(\Delta,\phi)$ as an elementary subgraph, and such that $(\Gamma_\beta',\psi_\beta')$ is obtained from $(\Gamma_{\beta-1}',\psi_{\beta-1}')$ by an edge sign-change, elementary expansion, slide, swap or connection, for $\beta=1,\dots,N'$. Note that, by doing this, the number of vertices in the last graph $(\Gamma_{N'}',\psi_{N'}')$ is one less than the number of vertices in $(\Gamma_N,\psi_N)$. We reiterate the same procedure for all vertices $u\in V(\Gamma_N,\psi_N)\setminus(V(\Gamma)\cup V(\Delta))$, and after a finite number of steps we obtain a sequence of GBS graphs
$$(\Gamma,\psi)=(\Omega_0,\omega_0),(\Omega_1,\omega_1),\dots,(\Omega_L,\omega_L)$$
with $L\in\bbN$, such that $(\Omega_L,\omega_L)$ contains $(\Delta,\phi)$ as an elementary subgraph, for $\beta=1,\dots,L$ we can go from $(\Gamma_{\beta-1},\psi_{\beta-1})$ to $(\Gamma_\beta,\psi_\beta)$ with an edge sign-change, elementary expansion, slide, swap or connection, and the set of vertices of every graph along the sequence is contained in $V(\Gamma)\cup V(\Delta)$. In particular $V(\Omega_L)=V(\Gamma)\cup V(\Delta)$.

Suppose that $V(\Omega_L)\supsetneq V(\Delta)$. Since $(\Delta,\phi)$ is an elementary subgraph of $(\Omega_L,\omega_L)$, by Lemma \ref{elementary-subgraph2}, we can find a valence-$1$ vertex $v_1\in V(\Omega_L)\setminus V(\Delta)$ such that the unique edge $e_1\in E(\Omega_L)$ with $\tau(e_1)=v_1$ satisfies $\omega_L(e_1)=\pm1$; let us call $\iota(e_1)=u_1$ the other endpoint of $e_1$. Since $V(\Omega_L)=V(\Gamma)\cup V(\Delta)$, it follows that $v_1\in V(\Gamma)$; by Lemma \ref{powers-along-story}, and since $(\Gamma,\psi)$ is totally reduced, we must also have $u_1\in V(\Delta)$. We now define $(\Omega_L^1,\omega_L^1)$ obtained from $\Omega_L$ by removing $v_1,e_1$ and we notice that $(\Delta,\phi)$ is an elementary subgraph of $(\Omega_L^1,\omega_L^1)$. If $V(\Omega_L^1)\supsetneq V(\Delta)$, then we reiterate the same reasoning and we find another valence-$1$ vertex $v_2\in V(\Omega_L^1)\setminus V(\Delta)$ such that the unique edge $e_2\in E(\Omega_L^1)$ with $\tau(e_2)=v_2$ satisfies $\psi(e_2)=\pm1$; we also call $u_2=\iota(e_2)$ the other endpoint of $e_2$. As before, we have $v_2\in V(\Gamma)$ and $u_2\in V(\Delta)$. We reiterate the same reasoning until we obtain a GBS graph $(\Omega_L^k,\omega_L^k)$, for some $k\in\bbN$, containing $(\Delta,\phi)$ as an elementary subgraph, and satisfying $V(\Omega_L^k)=V(\Delta)$. But then, by Lemma \ref{elementary-subgraph3}, we have $\Omega_L^k=\Delta$.

This proves that
$$V(\Omega_L)\setminus V(\Delta)=\{v_1,\dots,v_k\}$$
for some $k\in\bbN$ and $v_1,\dots,v_k\in V(\Gamma)$, and that
$$E(\Omega_L)\setminus E(\Delta)=\{e_1,\ol{e}_1,\dots,e_k,\ol{e}_k\}$$
such that $\tau(e_i)=v_i$ and $\iota(e_i)=u_i$ for some $u_i\in V(\Omega_L)\setminus V(\Gamma)$ and $\psi(e_i)=\pm1$, for $i=1,\dots,k$. In particular we can define a map $\rho:V(\Gamma)\rar V(\Delta)$ given by the identity on $V(\Gamma)\cap V(\Delta)$, and sending $v_i$ to $u_i$ for $i=1,\dots,k$. Notice that for $v\in V(\Gamma)$ the conjugacy class $[v]$ is a power of the conjugacy class $[\rho(v)]$. For every $u\in V(\Delta)$, there must be at least one vertex $v\in V(\Gamma)$ such that the conjugacy class $[u]$ is a power of $[v]$; thus we can define a map $\theta:V(\Delta)\rar V(\Gamma)$ sending each $u\in V(\Delta)$ to one such $v\in V(\Gamma)$. Since $(\Gamma,\psi),(\Delta,\phi)$ are totally reduced, we must have that $\rho$ and $\theta$ are inverse of each other. In particular this means that $\abs{V(\Gamma)}=\abs{V(\Delta)}$ and that $V(\Delta)\setminus V(\Gamma)=\{u_1,\dots,u_k\}$. This also means that for $u\in V(\Delta)$ we have that $\theta(u)$ is the \textit{unique} vertex in $V(\Gamma)$ such that the conjugacy class $[u]$ is a power of $[\theta(u)]$.

Since $\Delta$ is totally reduced, for $i=1,\dots,k$ we consider the vertex $u_i\in V(\Delta)\setminus V(\Gamma)$, and we can find an edge $d_i\in E(\Delta)$ satisfying Property \ref{itm:controlling-edge} of Definition \ref{def:totally-reduced}. Then we use the set of collapsing data $d_i,e_i$ for the vertex $u_i$ to perform the projection move of Section \ref{sec:projection} from $u_i$ to $v_i$. We call $(\Delta',\phi')$ the graph obtained performing such projection for all $k=1,\dots,n$. Notice that, since the collapsing data consists of only two edges $d_i,e_i$, in this case, the projection move coincides with the following: applying an induction move (plus potentially a vertex sign-change) at the vertex $u_i$ in $\Delta$, and then removing the edge $e_i$ and identifying $u_i$ and $v_i$, and then performing some slide moves to adjust the remaining edges. In particular, the graph $(\Delta',\phi')$ is obtained from $(\Delta,\phi)$ by performing an induction move (and potentially a vertex sign-change move) at some vertices, and then slide moves.

We now eliminate the vertices $u_1,\dots,u_k$ from the whole sequence of GBS graphs $(\Gamma,\psi)=(\Omega_0,\omega_0),\dots,(\Omega_L,\omega_L)$ using the projection move of Section \ref{sec:projection}. We obtain a sequence of GBS graphs going from $(\Gamma,\psi)$ to $(\Delta',\phi')$ and such that every two consecutive GBS graphs in the sequence are related by an edge sign-change, a slide, a swap or a connection; in fact, we are no longer using any elementary expansion, since we eliminated all the vertices except the ones in $V(\Gamma)$. In particular, we can go from $(\Gamma,\psi)$ to $(\Delta,\phi)$ with a sequence of edge sign-changes, slides, swaps, and connections, followed by inductions and vertex sign-changes at the end of the sequence. To conclude, we observe that edge sign-changes commute with slides, swaps, and connection, and thus they can be moved at the end of the sequence too.
\end{proof}

\begin{cor}\label{cor:sequence-now-moves}
Let $(\Gamma,\psi),(\Delta,\phi)$ be GBS graphs and suppose that the corresponding graphs of groups have isomorphic fundamental group. Then there is a sequence of elementary expansions, elementary contractions, slides, and sign-changes going from $(\Delta,\phi)$ to $(\Gamma,\psi)$. Moreover, the moves can be performed in such a way that every graph along the sequence has at most $\abs{V(\Gamma)}+2$ vertices.
\end{cor}
\begin{proof}
Follows from Theorem \ref{thm:sequence-new-moves} and Lemmas \ref{induction-from-others}, \ref{swap-from-others} and \ref{connection-from-others}.
\end{proof}

\subsection{Reduction to the one-vertex case}

\begin{defn}[Set of primes of a graph]
    For a GBS graph $(\Gamma,\psi)$, define its \textbf{set of primes}
    $$\cP(\Gamma,\psi):=\{p\in\bbN \text{ prime } : p\divides \psi(e) \text{ for some } e\in E(\Gamma)\}.$$
\end{defn}

It is immediate to notice that the set of primes $\cP(\Gamma,\psi)$ is invariant when performing sign-changes, inductions, slides, swaps, and connections. For this reason, and thanks to Theorem \ref{thm:sequence-new-moves}, it makes sense to give the following definition.

\begin{defn}[Set of primes of a GBS]
    For a generalized Baumslag-Solitar group $G$, define its \textbf{set of primes} $\cP(G):=\cP(\Gamma,\psi)$ for any totally reduced GBS graph $(\Gamma,\psi)$ such that the corresponding GBS graph of groups $\cG$ satisfies $\pi_1(\cG)\cong G$.
\end{defn}

\begin{remark}
    For a generalized Baumslag-Solitar group $G$ and a prime $p$, we have that $p\in\cP(G)$ if and only if there are two elliptic elements $a,b\in G$ such that $a$ is not conjugate to a power of $b$ but $a^p$ is conjugate to a power of $b$. Here ``elliptic" means that the element stabilizes a vertex in some (equivalently every) action of $G$ on a tree $T$ with edge stabilizers and vertex stabilizers isomorphic to $\bbZ$. This provides an intrinsic characterization of $\cP(G)$, without relying on Theorem \ref{thm:sequence-new-moves}.
\end{remark}

\begin{remark}
    Given a GBS graph $(\Gamma,\psi)$, there is a homomorphism $q:\pi_1(\Gamma)\rar\bbQ\setminus\{0\}$, called the \textit{modular homomorphism}, from the fundamental group of $\Gamma$ (as a graph) to the multiplicative group $\bbQ\setminus\{0\}$. A loop $(e_1,\dots,e_\ell)$ in $\Gamma$ is sent to the rational number $\frac{\psi(\tau(e_1))}{\psi(\iota(e_1)}\cdots \frac{\psi(\tau(e_\ell))}{\psi(\iota(e_\ell)}$. We can define $\cP'(G)$ to be the set of prime numbers that appear with non-zero exponent in $q(\sigma)$ for some $\sigma\in\pi_1(\Gamma)$. We point out that $\cP'(G)\not=\cP(G)$ in general (as, for example, the former is trivial if $\Gamma$ is a tree, while the latter in general is not).
\end{remark}

\begin{defn}[Induction-free graph]
    A GBS graph $(\Gamma,\psi)$ is called \textbf{induction-free} if $\psi(e)\not=\pm1$ for every $e\in E(\Gamma)$.
\end{defn}

Notice that a GBS graph which is induction-free is in particular totally reduced. It follows from Theorem \ref{thm:sequence-new-moves} that, if two totally reduced GBS graphs represent isomorphic groups and one is induction-free, then the other has to be induction-free too. For induction-free GBS graphs, the isomorphism problem is reduced to deciding the existence of a sequence of slides, swaps, and connections going from a given GBS graph to another given one.

The following Theorem \ref{thm:one-vertex} shows that, up to guessing the initial sign-changes and inductions, the isomorphism problem for generalized Baumslag-Solitar groups can be reduced to the case of induction-free one-vertex GBS graphs.

\begin{thm}\label{thm:one-vertex}
    Let $(\Gamma_0,\psi_0)$ and $(\Gamma_1,\psi_1)$ be totally reduced GBS graphs and let $b:V(\Gamma_1)\rar V(\Gamma_0)$ be a bijection. Then there are induction-free one-vertex GBS graphs $(\Delta_0,\phi_0),(\Delta_1,\phi_1)$ such that the following are equivalent:
    \begin{enumerate}
        \item There is a sequence of slides, swaps, and connections going from $(\Gamma_0,\psi_0)$ to $(\Gamma_1,\psi_1)$ and such that the vertex $v\in V(\Gamma_1)$ corresponds to the vertex $b(v)\in V(\Gamma_0)$.
        \item There is a sequence of slides, swaps, and connections going from $(\Delta_0,\phi_0)$ to $(\Delta_1,\phi_1)$.
    \end{enumerate}
    Moreover, $(\Delta_0,\phi_0),(\Delta_1,\phi_1)$ can be constructed algorithmically from $(\Gamma_0,\psi_0),(\Gamma_1,\psi_1),b$.
\end{thm}
\begin{remark}
    The above \Cref{thm:one-vertex} remains true if we add edge sign-changes to the list of moves that we are allowing.
\end{remark}
\begin{proof}
    Let $\cP=\cP(\Gamma_0,\psi_0)\cup\cP(\Gamma_1,\psi_1)$ be the set of primes that appear in at least one of $(\Gamma_0,\psi_0),(\Gamma_1,\psi_1)$. For every vertex $v\in V(\Gamma_0)$ choose a different prime number $p_v\not\in\cP$. Given a GBS graph $(\Gamma,\psi)$ and a bijection $a:V(\Gamma)\rar V(\Gamma_0)$, define the induction-free one-vertex GBS graph $(\Delta,\phi)$ as follows (see Figure \ref{fig:reduction-one-vertex}):
    \begin{enumerate}
        \item $V(\Delta)=\{*\}$
        \item $E(\Delta)=E(\Gamma)$ with the same reverse map. 
        \item For $e\in E(\Delta)$ we set $\phi(e)=p_{a(\tau(e))}\cdot \psi(e)$.
    \end{enumerate}
    Every label of every label on every edge of $\Delta$ carries exactly one prime $p_v$ for some $v\in V(\Gamma)$; the prime $p_v$ is meant to keep track of the fact that the endpoint of the edge was at $v$.
    
    This behaves well with respect to the moves: if we perform a slide, swap, or connection on two edges in $(\Gamma,\psi)$, then the graph $(\Delta,\phi)$ changes by a slide, swap or connection on the same edges. Conversely, if we perform a slide, swap, or connection on two edges in $(\Delta,\phi)$, then the same move can be performed in $(\Gamma,\psi)$. Thus sequences of slides, swaps, and connections on $(\Gamma,\psi)$ are in bijection with sequences of moves on the corresponding graphs $(\Delta,\phi)$.

    Define $(\Delta_0,\phi_0),(\Delta_1,\phi_1)$ as the graphs obtained from $(\Gamma_0,\psi_0),(\Gamma_1,\psi_1)$ respectively by the above construction. It follows that there is a sequence of slides, swaps, connections going from $(\Gamma_0,\psi_0)$ to $(\Gamma_1,\psi_1)$ if and only if there is a sequence of slides, swaps, connections going from $(\Delta_0,\phi_0)$ to $(\Delta_1,\phi_1)$. The statement  follows.
\end{proof}

\begin{figure}[H]
	\centering
	\begin{tikzpicture}[scale=1]
		
		\begin{scope}[shift={(0,0)}]
		\node (v1) at (0,0) {.};
		\node (v2) at (0,2) {.};
		\node (v3) at (2,1) {.};		
		\draw[-,out=45,in=90,looseness=30] (v3) to (v3);		
	 \node () at (0,-.2) {$v_1$};
	\node () at (0,2.2) {$v_2$};
	\node () at (0.2,1) {$e$};	
	\node () at (2,0.8) {$v_3$};
	\node () at (1.2,1.7) {$f$};
	\node () at (2.7,1.5) {$g$};		
	\node () at (1.2,-1) {$\Gamma$};
		\draw[-] (v1) to (v2);
		\draw[-] (v2) to (v3);
	\end{scope}

\begin{scope}[shift={(8,1)}]		
\node (p1) at (0.2,.2) {.};
\node (p2) at (-.4,.4) {.};
\node (p3) at (0,-0.4) {.};				
\node () at (0,0.1) {$p_1$};
\node () at (-0.6,.3) {$p_2$};
\node () at (-0.2,-0.5) {$p_3$};	
\draw (-.2,0) circle [radius=0.85]; 		
\draw[-,out=40,in=90,looseness=10] (p1) to (p2);
\draw[-,out=170,in=210,looseness=8] (p2) to (p3);	
\draw[-,out=10,in=-60,looseness=50] (p3) to (p3);		
	\node () at (-.1,-2) {$\Delta$};		
\end{scope}
	\end{tikzpicture}
	\caption{A one-vertex GBS graph that mimics the behaviour of another given GBS graph.}
	\label{fig:reduction-one-vertex}
\end{figure}
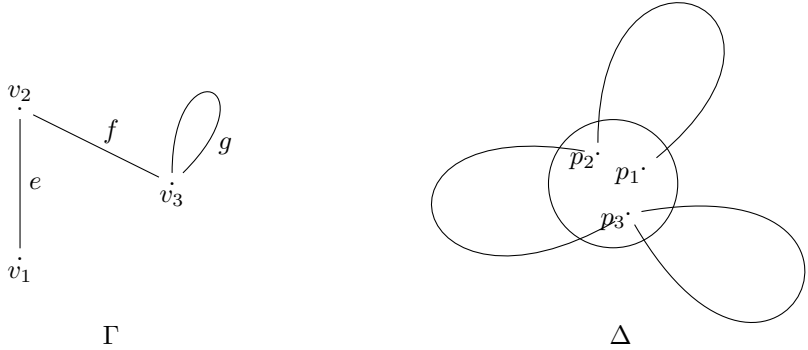

With a similar argument, it is also possible to assume that all the numbers on the edges are positive.

\begin{defn}
    A GBS graph $(\Gamma,\psi)$ is called \textbf{positive} if $\psi(e)>0$ for every $e\in E(\Gamma)$.
\end{defn}

\begin{prop}\label{prop:one-vertex-positive}
    Let $(\Delta_0,\phi_0),(\Delta_1,\phi_1)$ be one-vertex induction-free GBS graphs. Then there are one-vertex induction-free positive GBS graphs $(\Omega_0,\omega_0),(\Omega_1,\omega_1)$ such that the following are equivalent:
    \begin{enumerate}
        \item There is a sequence of slides, swaps, and connections going from $(\Delta_0,\phi_0)$ to $(\Delta_1,\phi_1)$.
        \item There is a sequence of slides, swaps, and connections going from $(\Omega_0,\omega_0)$ to $(\Omega_1,\omega_1)$.
    \end{enumerate}
    Moreover, $(\Omega_0,\omega_0),(\Omega_1,\omega_1)$ can be constructed algorithmically from $(\Delta_0,\phi_0),(\Delta_1,\phi_1)$.
\end{prop}
\begin{remark}
    The above \Cref{prop:one-vertex-positive} remains true if we add edge sign-changes to the list of moves that we are allowing.
\end{remark}
\begin{proof}
    Let $\cP=\cP(\Delta_0)\cup\cP(\Delta_1)$ and choose two primes $q,r\not\in\cP$. Given a one-vertex induction-free GBS graph $(\Delta,\phi)$, and given integers $k_e\ge0$ for $e\in E(\Delta)$, define the one-vertex induction-free positive GBS graph $(\Omega,\omega)$ as follows:
    \begin{enumerate}
        \item $V(\Omega)=\{*\}$
        \item $E(\Omega)=E(\Delta)\cup\{f,\ol{f}\}$ adding one extra edge $f$ (and its reverse $\ol{f}$).
        \item For $e\in E(\Delta)$ we set $\omega(e)=r^2\cdot(q^{2k_e}\phi(e))$ if $\phi(e)>0$ and $\omega(e)=r^2\cdot(-q^{2k_e+1}\phi(e))$ if $\phi(e)<0$.
        \item We set $\omega(f)=r$ and $\omega(\ol{f})=rq^2$.
    \end{enumerate}
    Here the prime $q$ is substituted to the factor $(-1)$. The additional edge $f$ is added, allowing to add or remove $q^2$ from the labels of the other edges at will, by slide moves; this is an artificial way of encoding the fact that $(-1)^2=1$. The GBS graph $(\Omega,\omega)$ is defined up to multiplication by $q^2$ of the labels on each edge; GBS graphs obtained with different choices of the integers $k_e$ only differ by slide of the other edges along $f$. The prime $r$ is needed, in order to ensure that the graph is induction-free, and that no move can change the extra edge $f$.
    
    This behaves well with respect to the moves: if we perform a slide, swap, or connection on two edges in $(\Delta,\phi)$, then the graph $(\Omega,\omega)$ changes by a slide, swap or connection on the same edges (and possibly by a change of choice of the integers $k_e$ for $e\in E(\Delta)$). Conversely, suppose that we perform a slide, swap or connection on two edges in $(\Omega,\omega)$: if the move does not involve $f,\ol{f}$, then the same move can be performed in $(\Gamma,\psi)$; if the move involves $f,\ol{f}$, then it has to be a slide of another edge along $f$ or $\ol{f}$ (since every other label contains a factor $r^2$, both before and after the move), and thus it corresponds to a change in the choice of the integers $k_e$ for $e\in E(\Delta)$.

    Define $(\Omega_0,\omega_0),(\Omega_1,\omega_1)$ as the graphs obtained from $(\Delta_0,\phi_0),(\Delta_1,\phi_1)$ respectively by the above construction. It follows that there is a sequence of slides, swaps, connections going from $(\Delta_0,\phi_0)$ to $(\Delta_1,\phi_1)$ if and only if there is a sequence of slides, swaps, connections going from $(\Omega_0,\omega_0)$ to $(\Omega_1,\omega_1)$. The statement  follows.
\end{proof}

\subsection{Controlled one-vertex GBSs}

As a consequence of Theorem \ref{thm:sequence-new-moves}, and using the results of Section \ref{sec:controlled-linear-algebra}, we can solve the isomorphism problem for a family of flexible one-vertex GBS graphs.

\begin{defn}
    Let $(\Gamma,\psi)$ be a one-vertex GBS graph with affine representation $\Lambda$. We say that $(\Gamma,\psi)$ is \textbf{controlled} if $\Lambda$ contains an edge $\ba\edge\ba+\bw$ with $\ba,\bw\in\pA$, such that $\ba,\bw$ controls $\bb$ for every endpoint $\bb\in\pA$ of every other edge.
\end{defn}

\begin{prop}\label{prop:isomorphism-controlled}
    There is an algorithm that, given in input a one-vertex controlled GBS graphs $(\Gamma,\psi)$ and any GBS graph $(\Delta,\phi)$, decides
    \begin{enumerate}
        \item Whether there is a sequence of slides, swaps, connections going from one to the other.
        \item Whether there is a sequence of sign-changes, inductions, slides, swaps, connections going from one to the other.
    \end{enumerate}
\end{prop}
\begin{proof}
    We first deal with the case of sequences of slides, swaps, connections.

    Let $\Lambda$ be the affine representation of $(\Gamma,\psi)$. Let $\ba\edge\ba+\bw$ with $\ba,\bw\in\pA$ the edge in $\Lambda$ such that $\ba,\bw$ controls the endpoints of all the other edges. Let $\bb_i\edge\bc_i$ for $i=1,\dots,k$ be the other edges in $\Lambda$, so that $\ba,\bw$ controls $\bb_i,\bc_i$.
    
    It is easy to check that slides, swaps, connections preserve the number of edge and the subgroup $H=\gen{\bw,\bb_1-\bc_1,\dots,\bb_k-\bc_k}\le\bA$. It is also easy to check that for each coset $\bu+H$ of $H$ the moves slides, swaps, connections preserve the number of edges $e\in E(\Lambda)$ such that $\iota(e)\in \bu+H$.

    Let now $\Lambda'$ be the affine representation of $(\Delta,\phi)$, and suppose that we can go from $(\Gamma,\psi)$ to $(\Delta,\phi)$ by means of a sequence of slides, swaps, connections. In particular, $(\Delta,\phi)$ must have the same number of vertices and edges as $(\Gamma,\psi)$. As slides, swaps, connections do not affect the conjugacy relation, there must be an affine path in $\Lambda'$ from $\ba$ to $\ba+\bw$, and thus, by applying the same procedure as in Section \ref{sec:projection}, we can assume that $(\Delta,\phi)$ is controlled too. Let us say $\Lambda'$ has edges $\ba'\edge\ba'+\bw'$ and $\bb_i'\edge\bc_i'$ for $i=1,\dots,k$, such that $\ba',\bw'$ controls $\bb_i',\bc_i'$. We must have that $\gen{\bw',\bb_1'-\bc_1',\dots,\bb_k'-\bc_k'}=H$ otherwise there is no sequence of moves going from one to the other. Similarly, up to permuting the $\bb_i'$, we must have that $\bb_i-\bb_i'\in H$. Note that all of this can be checked algorithmically.

    But now Theorem \ref{thm:controlled} allows us to deduce that there is a sequence of slides, swaps, connections going from $(\Gamma,\psi)$ to $(\Delta,\phi)$, as desired.

    Sign changes commute with all the other moves, and can be applied in finitely many ways, so it is enough to make finitely many guesses. If induction can be applied, then we must have that $\ba\in\bbZ/2\bbZ\le\bA$. In this case, the induction will not change the subgroup $H$, but it will translate all the cosets $\bb_i+H$. It is easy to algorithmically determine whether there is such a translation that makes us fall into the previous case. The statement  follows.
\end{proof}

\begin{cor}\label{cor:iso-controlled}
    There is an algorithm that, given in input a one-vertex controlled GBS graphs $(\Gamma,\psi)$ and a GBS graph $(\Delta,\phi)$, decides whether the corresponding GBS groups are isomorphic.
\end{cor}
\begin{proof}
    This follows from Proposition \ref{prop:isomorphism-controlled} and Theorem \ref{thm:sequence-new-moves}.
\end{proof}

\section{Comparing different deformation spaces}\label{sec:comparing-deformation-spaces}

The aim of this section is to prove that if there is an algorithm to solve the isomorphism problem for GBS groups, then the same algorithm can be applied also to much wider families of groups. In order to do this, we generalize the results of the previous sections; since the proofs are completely analogous, we only provide sketches.

\subsection{Moves on graphs of groups}

We recall that the conventions for graphs and graphs of groups have been given in Section \ref{sec:preliminaries}.

\begin{defn}[$\bbZ$-graph] \label{def:Z-graph}
     A \textbf{$\bbZ$-graph} is a triple $(\Gamma,\{G_v\}_{v\in V(\Gamma)},\psi)$ given by:
     \begin{enumerate}
         \item A finite connected graph $\Gamma$.
         \item A group $G_v$ for every $v\in V(\Gamma)$.
         \item A map $\psi:E(\Gamma)\rar\bigcup_{v\in V(\Gamma)}G_v$ sending every edge $e\in E(\Gamma)$ to an infinite-order element $\psi(e)\in G_{\tau(e)}$. 
     \end{enumerate}
\end{defn}

\begin{defn}[Cyclic vertex group]\label{def:cyclic-vertices}
    Let $(\Gamma,\{G_v\},\psi)$ be a $\bbZ$-graph. A vertex $v\in V(\Gamma)$ is called \textbf{cyclic} if $G_v\cong\bbZ$. We denote by $\cyclicV{\Gamma}$ and with $\notcyclicV{\Gamma}$ the set of cyclic and non-cyclic vertices of $\Gamma$ respectively.
\end{defn}

Given a graph of groups $\cG=(\Gamma,\{G_v\}_{v\in V(\Gamma)},\{G_e\}_{e\in E(\Gamma)},\{\psi_e\}_{e\in E(\Gamma)})$ with edge groups isomorphic to $\bbZ$, we choose a generator $z_e$ for every edge group $G_e\cong\bbZ$, and we define the $\bbZ$-graph $(\Gamma,\{G_v\}_{v\in V(\Gamma)},\psi)$ with $\psi(e)=\psi_e(z_e)$. Giving a graph of groups with edge groups isomorphic to $\bbZ$ is equivalent to giving the corresponding $\bbZ$-graph. As for GBS graphs, we can define a set of moves on $\bbZ$-graphs. Let $(\Gamma,\{G_v\}_{v\in V(\Gamma)},\psi)$ be a $\bbZ$-graph: each of the following moves induces an isomorphism at the level of the fundamental group of the corresponding graph of groups.

\vspace{0.1cm}

\textbf{Twist.} Let $d\in E(\Gamma)$ be an edge with $\tau(d)=v$ and let $h\in G_v$ be any element. Define the map $\psi':E(\Gamma)\rar\bigcup_{v\in V(\Gamma)}G_v$ given by $\psi'(d)=h\cdot\psi(d)\cdot\ol{h}$ and $\psi'(e)=\psi(e)$ for $e\not=d$. We say that $(\Gamma',\{G_v\}_{v\in V(\Gamma')},\psi')$ is obtained from $(\Gamma,\{G_v\}_{v\in V(\Gamma)},\psi)$ by a \textbf{twist}.

\vspace{0.1cm}

\textbf{Vertex coordinate-change.} Let $u\in V(\Gamma)$ be a vertex and let $i:G_u\rar H$ be an isomorphism of groups. Define the $\bbZ$-graph $(\Gamma,\{G_v'\}_{v\in V(\Gamma)},\psi')$ given by $G_u'=H$ and $G_v'=G_v$ for $v\not= u$, and by $\psi'(e)=i(\psi(e))$ if $\tau(e)=u$ and $\psi'(e)=\psi(e)$ otherwise. We say that $(\Gamma,\{G_v'\}_{v\in V(\Gamma)},\psi')$ is obtained from $(\Gamma,\{G_v\}_{v\in V(\Gamma)},\psi)$ by a \textbf{vertex coordinate-change}.

\vspace{0.1cm}

\textbf{Edge sign-change.} Let $d\in E(\Gamma)$ be an edge. Define the map $\psi':E(\Gamma)\rar\bigcup_{v\in V(\Gamma)}G_v$ such that $\psi'(e)=(\psi(e))^{-1}$ if $e=d,\ol{d}$ and $\psi'(e)=\psi(e)$ otherwise. We say that $(\Gamma',\{G_v\}_{v\in V(\Gamma')},\psi')$ is obtained from $(\Gamma,\{G_v\}_{v\in V(\Gamma)},\psi)$ by an \textbf{edge sign-change}.

\vspace{0.1cm}

\textbf{Elementary expansion.} Let $v\in V(\Gamma)$ be a vertex and let $g\in G_v$ be an infinite-order element. Define the graph $\Gamma'$ by adding a new vertex $v_0$ and a new edge $e_0$ going from $v_0$ to $v$; we set $G_{v_0}=\bbZ$ with $\psi(\ol{e}_0)=1\in G_{v_0}$ and $\psi(e_0)=g$. We say that $(\Gamma',\{G_v\}_{v\in V(\Gamma')},\psi')$ is obtained from $(\Gamma,\{G_v\}_{v\in V(\Gamma)},\psi)$ by an \textbf{elementary expansion}.

\vspace{0.1cm}

\textbf{Elementary contraction.} Let $v_0\in V(\Gamma)$ be a vertex with $G_{v_0}=\bbZ$ and suppose that there is a unique edge $e_0$ with $\iota(e_0)=v_0$; suppose also that $\tau(e_0)\not=v_0$ and $\psi(\ol{e}_0)$ is a generator for $G_{v_0}=\bbZ$. Define the graph $\Gamma'$ by removing the vertex $v_0$ and the edge $e_0$. We say that $(\Gamma',\{G_v\}_{v\in V(\Gamma')},\psi')$ is obtained from $(\Gamma,\{G_v\}_{v\in V(\Gamma)},\psi)$ by an \textbf{elementary contraction}. The elementary contraction move is the inverse of the elementary expansion move.

\vspace{0.1cm}

\textbf{Slide.} Let $d,e$ be distinct edges with $\tau(d)=\iota(e)=u$ and $\tau(e)=v$ for some $u,v\in V(\Gamma)$; suppose that $\psi(\ol{e})=g^n$ and $\psi(e)=h^m$ and $\psi(d)=g^{\ell n}$ for some $n,m,\ell\in\bbZ\setminus\{0\}$ and $g\in G_u$ and $h\in G_v$. Define the graph $\Gamma'$ by replacing the edge $d$ with an edge $d'$; we set $\iota(d')=\iota(d)$ and $\tau(d')=v$; we set $\psi(\ol{d}')=\psi(\ol{d})$ and $\psi(d')=h^{\ell m}$. We say that $(\Gamma',\{G_v\}_{v\in V(\Gamma')},\psi')$ is obtained from $(\Gamma,\{G_v\}_{v\in V(\Gamma)},\psi)$ by a \textbf{slide}.

\vspace{0.1cm}

\textbf{Induction.} Let $v\in V(\Gamma)$ be such that $G_v\cong\bbZ$ with generator $z_v$; let $e$ be an edge with $\iota(e)=\tau(e)=v$ and suppose that $\psi(\ol{e})=z_v$ and $\psi(e)=z_v^n$ for some $n\in\bbZ\setminus\{0\}$. Choose $\ell\in\bbZ\setminus\{0\}$ and $k\in\bbN$ such that $\ell\divides n^k$; define the map $\psi'$ equal to $\psi$ except on the edges $d\not=e,\ol{e}$ with $\tau(d)=v$, where we set $\psi'(d)=\psi(d)^\ell$. We say that $(\Gamma,\{G_v\}_{v\in V(\Gamma)},\psi')$ is obtained from $(\Gamma,\{G_v\}_{v\in V(\Gamma)},\psi)$ by an \textbf{induction}.

\vspace{0.1cm}

\textbf{Swap.} Let $e_1,e_2$ be distinct edges with $\iota(e_1)=\tau(e_1)=\iota(e_2)=\tau(e_2)=v$; suppose that $\psi(\ol{e}_1)=g^n$ and $\psi(e_1)=g^{\ell_1 n}$ and $\psi(\ol{e}_2)=g^m$ and $\psi(e_2)=g^{\ell_2 m}$ for some $g\in G_v$ and $n,m,\ell_1,\ell_2\in\bbZ\setminus\{0\}$ such that $n\divides m\divides \ell_1^{k_1} n, \ell_2^{k_2} n$ for some $k_1,k_2\in\bbN$. Define the graph $\Gamma'$ by substituting the edges $e_1,e_2$ with two edges $e_1',e_2'$; we set $\iota(e_1')=\tau(e_1')=\iota(e_2')=\tau(e_2')=v$; we set $\psi(\ol{e}_1')=g^m$ and $\psi(e_1')=g^{\ell_1 m}$ and $\psi(\ol{e}_2')=g^n$ and  $\psi(e_2')=g^{\ell_2 n}$. We say that $(\Gamma',\{G_v\}_{v\in V(\Gamma')},\psi')$ is obtained from $(\Gamma,\{G_v\}_{v\in V(\Gamma)},\psi)$ by a \textbf{swap move}.

\vspace{0.1cm}

\textbf{Connection.} Let $d,e$ be distinct edges with $\iota(d)=u$ and $\tau(d)=\iota(e)=\tau(e)=v$; suppose that  $\psi(\ol{d})=g^m$ and $\psi(d)=g^{\ell_1 n}$ and $\psi(\ol{e})=g^n$ and $\psi(e)=g^{\ell n}$ for some $g\in G_v$ and $m,n,\ell_1,\ell_2,\ell\in\bbZ\setminus\{0\}$ such that $\ell_1\ell_2=\ell^k$ for some $k\in\bbN$. Define the graph $\Gamma'$ by substituting the edges $d,e$ with two edges $d',e'$; we set $\iota(d')=v$ and $\tau(d')=\iota(e')=\tau(e')=u$; we set  $\psi(\ol{d}')=g^n$ and $\psi(d')=g^{\ell_2 m}$ and $\psi(\ol{e}')=g^m$ and $\psi(e')=g^{\ell m}$. We say that $(\Gamma',\{G_v\}_{v\in V(\Gamma')},\psi')$ is obtained from $(\Gamma,\{G_v\}_{v\in V(\Gamma)},\psi)$ by a \textbf{connection move}.

\begin{remark}
    Notice that elementary expansion and elementary contraction only allow adding and removing cyclic vertices to $\Gamma$. Similarly, induction can only be performed at cyclic vertices. On the contrary, all the other moves can, in principle, be performed at any vertex. In any case, no move changes the fact that a vertex is cyclic or not.
\end{remark}

\begin{lem}
    Let $(\Gamma,\{G_v\}_{v\in V(\Gamma)},\psi)$ be a $\bbZ$-graph. Then every induction, swap, and connection can be written as a sequence of elementary contraction, expansions, and slides. This can be done in such a way that every $\bbZ$-graph along the sequence has at most $\abs{V(\Gamma)}+1$ vertices for induction, and $\abs{V(\Gamma)}+2$ vertices for swap and connection.
\end{lem}
\begin{proof}
    We use the same sequences of moves as in Lemmas \ref{induction-from-others}, \ref{swap-from-others} and \ref{connection-from-others}.
\end{proof}

\subsection{Collapsing data and the projection move}\label{sec:projection-2}

As in Section \ref{sec:projection}, we can define a projection move that allows us to remove some vertices. The map works exactly in the same way as for GBS graphs; we summarize the results here. In the next Section \ref{sec:redundant-2} we will see that the vertices that can be eliminated are exactly the redundant ones (see Definition \ref{def:redundant-2} below).

\begin{defn}\label{def:collapsing-data-2}
Let $(\Gamma,\{G_v\}_{v\in V(\Gamma)},\psi)$ be a $\bbZ$-graph and let $v\in \cyclicV{\Gamma}$ with $G_v=\gen{z}$. A \textbf{set of collapsing data} for $v$ is given by pairwise distinct edges $e_1,\dots,e_n,e$ of $\Gamma$ satisfying the following properties:
\begin{enumerate}
\item For $j=1,\dots,n$ we have that $e_j$ goes from $v$ to $v$, with $\psi(\ol{e}_j)=z^{a_j}$ and $\psi(e_j)=z^{b_j}$.
\item For $j=1,\dots,n$, every prime that divides $a_j$ also divides $b_{j-1}\cdot b_{j-2}\cdots b_1$.
\item The edge $e$ goes from $v$ to $u$ with $u\not=v$, with $\psi(\ol{e})=z^c$.
\item Every prime that divides $c$ also divides $b_n\cdot b_{n-1}\cdots b_1$.
\end{enumerate}
\end{defn}

Let $(\Gamma,\{G_v\}_{v\in V(\Gamma)},\psi)$ be a $\bbZ$-graph, and let $v\in \cyclicV{\Gamma}$ with $G_v=\gen{z}$. Consider the constant
$$K(v)=\max\{\nu_p(n) : p \text{ prime and } n\in\bbZ\setminus\{0\} \text{ with } \psi(d)=z^n \text{ for some } d\in E(\Gamma) \text{ with } \tau(d)=v\},$$
where $\nu_p(n)$ denotes the exponent of $p$ in the factorization of $n$.

\begin{defn}\label{def:collapsing-constants-2}
Let $(\Gamma,\{G_v\}_{v\in V(\Gamma)},\psi)$ be a $\bbZ$-graph, and let $v\in V(\Gamma)$ be a vertex with $G_v\cong\bbZ$. Let $e_1,\dots,e_n,e$ be a set of collapsing data. A \textbf{set of collapsing constants} is given by integers $k_1,\dots,k_n$ which are strictly bigger than $K(v)$.
\end{defn}

Let $(\Gamma,\{G_v\}_{v\in V(\Gamma)},\psi)$ be a $\bbZ$-graph, and let $v\in \cyclicV{\Gamma}$ with $G_v=\gen{z}$. Let $e_1,\dots,e_n,e$ be a set of collapsing data for $v$, where $\psi(\ol{e}_j)=z^{a_j}$ and $\psi(e_j)=z^{b_j}$ for $j=1,\dots,n$, and where $\psi(\ol{e})=z^c$ and $\tau(e)=u$, for some $u\in V(\Gamma)$ with $u\not=v$ and for integers $a_j,b_j,c\in\bbZ\setminus\{0\}$ as in Definition \ref{def:collapsing-data-2}. Let $k_1,\dots,k_n$ be a set of collapsing constants. We define $w_1=\frac{b_1}{a_1}$ and inductively $w_j=\frac{b_j}{a_j}\cdot w_{j-1}^{k_{j-1}}$ for $j=2,\dots,n$. As in Lemma \ref{lemmaw}, we have by induction that $w_j$ is a non-zero integer for $j=1,\dots,n$, and that every prime that divides $b_j\cdot b_{j-1}\cdots b_1$ also divides $w_j$. We also define the integer $w=w_n^{k_n}/c$.

We now define the $\bbZ$-graph $(\Gamma',\{G_v\}_{v\in V(\Gamma')},\psi')$ as follows. We consider the vertex $u=\tau(e)$ and the element $g=\psi(e)\in G_u$. We define the graph $\Gamma'$ by removing the vertex $v$ and the edge $e$. The groups associated with the other vertices remain unchanged. We substitute the edges $e_1,\dots,e_n$ with edges $e_1',\dots,e_n'$ as follows. We set $\iota(e_1')=\tau(e_1')=u$ and $\psi'(\ol{e}_1')=g$ and $\psi'(e_1')=g^{w_n}$. For $j=2,\dots,n$, we set $\iota(e_j')=\tau(e_j')=u$ and $\psi'(\ol{e}_j')=g^{wa_j}$ and $\psi'(e_j')=g^{wa_jw_{j-1}}$. Finally, for every other edge $d$ with $\tau_{\Gamma}(d)=v$ and $\psi(d)=z^x$, we set $\tau_{\Gamma'}(d)=u$ and $\psi'(d)=g^{wx}$. We say that $(\Gamma',\{G_v\}_{v\in V(\Gamma')},\psi')$ is obtained from $(\Gamma,\{G_v\}_{v\in V(\Gamma)},\psi)$ by a \textbf{projection move} (relative to the vertex $v$, the collapsing data $e_1,\dots,e_n,e$ and the collapsing constants $k_1,\dots,k_n$).

\begin{prop}\label{projection-2}
Every projection move can be written as a composition of edge sign-changes, slides, swaps, and inductions, followed by an elementary contraction.
\end{prop}
\begin{proof}
    We use the same sequence of moves as in Proposition \ref{projection}.
\end{proof}

In particular, a projection move induces an isomorphism at the level of fundamental groups of the corresponding graphs of groups.

\subsection{Redundant vertices and totally reduced graphs}\label{sec:redundant-2}

\begin{defn}\label{def:redundant-2}
    Let $\cG=(\Gamma,\{G_v\}_{v\in V(\Gamma)},\{G_e\}_{e\in E(\Gamma)},\{\psi_e\}_{e\in E(\Gamma)})$ be a graph of groups with edge groups isomorphic to $\bbZ$. A vertex $v\in V(\Gamma)$ is called \textbf{redundant} if $G_v\cong\bbZ$ and the generator $z$ defines the same conjugacy class as some element in some other vertex group {\rm(}i.e. $[z]=[g]$ for some $g\in G_u$ with $u\in V(\Gamma)$ and $u\not=v${\rm)}. For a $\bbZ$-graph $(\Gamma,\{G_v\}_{v\in V(\Gamma)},\psi)$, a vertex $v\in V(\Gamma)$ is called \textbf{redundant} if it is redundant in the corresponding graph of groups.
\end{defn}

\begin{prop}\label{prop:redundant-data-2}
    Let $(\Gamma,\{G_v\}_{v\in V(\Gamma)},\psi)$ be a $\bbZ$-graph, and let $v\in \cyclicV{\Gamma}$. Then $v$ is redundant if and only if it has a set of collapsing data.
\end{prop}
\begin{proof}
     Let $z$ be a generator of $G_v$. Let $\cG$ be the graph of groups corresponding to the $\bbZ$-graph $(\Gamma,\{G_v\}_{v\in V(\Gamma)},\psi)$.
     
     If there is a set of collapsing data $e_1,\dots,e_k,e$ for $v$, then we choose a set of collapsing constants $k_1,\dots,k_n$; we denote with $a_j,b_j,c\in\bbZ\setminus\{0\}$ as in Definition \ref{def:collapsing-data-2}, and $u=\tau(e)$, and with $w_j\in\bbZ\setminus\{0\}$ for $j=1,\dots,n$ as in Section \ref{sec:projection-2}. We observe that in the universal group $\FG{\cG}$ we have $ze_1=e_1z^{w_1}$ and thus $z$ defines the same conjugacy class as $z^{w_1}$, and thus also of $z^{w_1^{k_1}}$. Similarly, for $j=2,\dots,n$ we have $z^{w_{j-1}^{k_{j-1}}}e_j=e_jz^{w_j}$ and thus by induction we have that $z$ defines the same conjugacy class as $z_{w_j^{k_j}}$. In particular $z$ defines the same conjugacy class as $z^{w_n^{k_n}}$; but $z^{w_n^{k_n}}e=eg$ for some $g\in G_u$ and thus $z$ defines the same conjugacy class as $g$, meaning that $v$ is redundant.

     If $v$ is redundant, then there must be a path $(e_1,\dots,e_n)$ in $\Gamma$, going from $v$ to a vertex $u\not=v$, such that in the universal group $\FG{\cG}$ we have $ze_1\dots e_n=e_1\dots e_ng$ for some $g\in G_u$. We take such a path of minimum possible length, forcing $e_1,\dots,e_{n-1}$ to be edges from $v$ to $v$. For $j=1,\dots,n$ we must have that $\psi(\ol{e}_j)\divides\prod_{i=1}^{j-1}\frac{\psi(e_i)}{\psi(\ol{e}_i)}$; in particular, every prime that divides $\psi(\ol{e}_j)$ also divides the product $\prod_{i=1}^{j-1}\psi(e_i)$. Therefore, if we remove all the occurrences of the same edge from the sequence $e_1,\dots,e_n$ (leaving only the first occurrence of each edge), we obtain a set of collapsing data for $v$, as desired.
\end{proof}

\begin{defn}\label{def:totally-reduced-2}
A $\bbZ$-graph $(\Gamma,\{G_v\}_{v\in V(\Gamma)},\psi)$ is called \textbf{totally reduced} if the following conditions hold:
\begin{enumerate}
\item\label{itm:fully-reduced-2} It contains no redundant vertex.
\item\label{itm:controlling-edge-2} For every $v\in \cyclicV{\Gamma}$ with $G_v=\gen{z}$, there is an edge $e$ from $v$ to $v$ with $\psi(\ol{e})=z^{\pm1}$ and $\psi(e)=z^n$, and such that, if $z^{n'}$ defines the same conjugacy class as $z$, then every prime that divides $n'$ also divides $n$.
\end{enumerate}
\end{defn}

\begin{prop}\label{prop:compute-totally-reduced-2}
Let $(\Gamma,\{G_v\}_{v\in V(\Gamma)},\psi)$ be a $\bbZ$-graph. Then there is a sequence of edge sign-changes, elementary expansions, elementary contractions and slides changing $(\Gamma,\{G_v\}_{v\in V(\Gamma)},\psi)$ into a totally reduced $\bbZ$-graph $(\Gamma',\{G_v\}_{v\in V(\Gamma')},\psi')$. Moreover there is an algorithm that, given $(\Gamma,\{G_v\}_{v\in V(\Gamma)},\psi)$, computes the sequence of moves and $(\Gamma',\{G_v\}_{v\in V(\Gamma')},\psi')$.
\end{prop}
\begin{proof}
    We choose a vertex $v\in \cyclicV{\Gamma}$ with $G_v=\gen{z}$, and we restrict our attention to the subgraph $\Gamma'$ of $\Gamma$ given by the vertex $v$ and all the edges going from $v$ to itself. We have that $\Gamma'$ is a GBS graph and by Proposition \ref{compute-conjugacy-class} we can compute the conjugacy class of $z$. We now look at the whole $\bbZ$-graph $\Gamma$ (and in particular at the edges going from $v$ to another vertex $u\not=u$), and we are able to determine whether or not $v$ is redundant, and if it is, a set of collapsing data. If $v$ is redundant, we use the set of collapsing data to remove the vertex with a projection move of Section \ref{sec:projection-2}. If $v$ is not redundant, then by Proposition \ref{prop:compute-totally-reduced-2} we can algorithmically compute a sequence of moves making $\Gamma'$ into a totally reduced GBS graph; the same sequence of moves applied on the $\bbZ$-graph $\Gamma$ will ensure Condition \ref{itm:controlling-edge-2} of Definition \ref{def:totally-reduced-2} for the vertex $v$. We repeat the same procedure for every vertex $v\in \cyclicV{\Gamma}$, and the statement  follows.
\end{proof}

\subsection{The coarse projection map}

\begin{prop}[Coarse projection map]\label{prop:indep-data-2}
Let $(\Gamma,\{G_v\}_{v\in V(\Gamma)},\psi)$ be a $\bbZ$-graph and let $v$ be a redundant vertex. Let $e_1,\dots,e_n,e$ and $e_1',\dots,e_m',e'$ be two sets of collapsing data for $v$ and let $(\Delta,\{G_v\}_{v\in V(\Delta)},\phi)$ and $(\Delta',\{G_v\}_{v\in V(\Delta')},\phi')$ be the $\bbZ$-graphs obtained with the two projection moves. Then there is a sequence of edge sign-changes, slides, swaps, and connections going from $(\Delta,\{G_v\}_{v\in V(\Delta)},\phi)$ to $(\Delta',\{G_v\}_{v\in V(\Delta')},\phi')$.
\end{prop}
\begin{proof}
We use the same sequence of moves as in Proposition \ref{prop:indep-data}.
\end{proof}

\begin{prop}\label{prop:projection-of-moves-2}
Let $(\Gamma_0,\{G_v\}_{v\in V(\Gamma_0)},\psi_0)$ be a $\bbZ$-graph, let $v$ be a redundant vertex, and let $(\Gamma_0',\{G_v\}_{v\in V(\Gamma_0')},\psi_0')$ be the $\bbZ$-graph obtained with the projection move. Let $(\Gamma_1,\{G_v\}_{v\in V(\Gamma_1)},\psi_1)$ be obtained from $(\Gamma_0,\{G_v\}_{v\in V(\Gamma_0)},\psi_0)$ by a twist, edge sign-change, slide, swap or connection; let $v$ be the corresponding redundant vertex in $(\Gamma_1,\{G_v\}_{v\in V(\Gamma_1)},\psi_1)$, and let $(\Gamma_1',\{G_v\}_{v\in V(\Gamma_1')},\psi_1')$ be the graph obtained with the projection move. Then $(\Gamma_1',\{G_v\}_{v\in V(\Gamma_1')},\psi_1')$ can be obtained from $(\Gamma_0',\{G_v\}_{v\in V(\Gamma_0')},\psi_0')$ by twists, edge sign-changes, slides, swaps, and connections.
\end{prop}
\begin{proof}
If the move from $(\Gamma_0,\{G_v\}_{v\in V(\Gamma_0)},\psi_0)$ to $(\Gamma_1,\{G_v\}_{v\in V(\Gamma_1)},\psi_1)$ is not a twist, then we use the same sequence of moves as in Proposition \ref{prop:projection-of-moves}. If it is a twist performed on an edge $d$ with $\tau(d)=v$, then it has no effect on the $\bbZ$-graph, and we are done. If it is a twist performed at the edge $d$ which is not involved in the set of collapsing data that we are using, then it commutes with the projection move, and we are done. Finally, if we are using a set of collapsing data $e_1,\dots,e_n,e$ and the twist is performed on the edge $e$, then $(\Gamma_1',\{G_v\}_{v\in V(\Gamma_1')},\psi_1')$ can be obtained from $(\Gamma_0',\{G_v\}_{v\in V(\Gamma_0')},\psi_0')$ by applying the same twist (i.e. with the same conjugating element) on all the edges which have been changed by the projection move.
\end{proof}

\subsection{Eliminating redundant vertices from sequences of moves}

\begin{lem}\label{lem:powers-along-story-2}
Let $\cG_\beta=(\Gamma_\beta,\{G_v\}_{v\in V(\Gamma_\beta)},\{G_e\}_{e\in E(\Gamma_\beta)},\{\psi_e\}_{e\in E(\Gamma_\beta)})$ be graph of groups with edge groups isomorphic to $\bbZ$ for $\beta=1, \dots, N$, such that the corresponding sequence:
$$(\Gamma_0,\{G_v\},\psi_0),\dots,(\Gamma_N,\{G_v\},\psi_N)$$
of $\bbZ$-graphs satisfies that $(\Gamma_\beta,\{G_v\},\psi_\beta)$ is obtained from $(\Gamma_{\beta-1},\{G_v\},\psi_{\beta-1})$ by a twist, edge sign-change, elementary expansion, elementary contraction, slide, swap, or connection, for $\beta=1,\dots, N$. 

Let $u,v$ be vertices that belong to both $V(\Gamma_0)$ and $V(\Gamma_N)$ {\rm(}i.e. they do not disappear with an elementary contraction along the sequence{\rm)} and such that $G_u\simeq \bbZ=\langle z_u\rangle$ and $G_v\simeq \bbZ=\langle z_v\rangle$. Then the following are equivalent:
\begin{enumerate}
\item $z_v$ is conjugate to a power of $z_u$ in $\cG_0$.
\item $z_v$ is conjugate to a power of $z_u$ in $\cG_N$.
\end{enumerate}
\end{lem}
\begin{proof}
Induction on $N\ge0$.
\end{proof}

\begin{defn}
Let $(\Gamma,\{G_v\}_{v\in V(\Gamma)},\psi)$ be a $\bbZ$-graph. An \textbf{elementary subgraph} is a $\bbZ$-graph  $(\Delta,\{G_v\}_{v\in V(\Delta)},\phi)$ such that
\begin{enumerate}
\item $\Delta$ is a subgraph of $\Gamma$ and $\phi$ is the restriction of $\psi$.
\item The inclusion between the $\bbZ$-graphs induces an isomorphism between the fundamental groups of the corresponding graphs of groups.
\end{enumerate}
\end{defn}

If $(\Delta,\{G_v\}_{v\in V(\Delta)},\phi)$ is an elementary subgraph of $(\Gamma,\{G_v\}_{v\in V(\Gamma)},\psi)$, then $\rank{\Gamma}=\rank{\Delta}$; moreover, if $V(\Gamma)=V(\Delta)$ then $\Gamma=\Delta$.

\begin{lem}\label{lem:elem-Z-subgraph}
If $(\Delta,\{G_v\}_{v\in V(\Delta)},\phi)$ is an elementary subgraph of $(\Gamma,\{G_v\}_{v\in V(\Gamma)},\psi)$, then there is a valence-$1$ vertex $v\in V(\Gamma)\setminus V(\Delta)$ with $G_v\cong\bbZ$ and such that the unique edge $e$ with $\tau(e)=v$ satisfies $\psi(e)=\pm1$.
\end{lem}
\begin{proof}
Analogous to Lemma \ref{elementary-subgraph2}.
\end{proof}

\begin{prop}\label{prop:no-contractions-2}
Let $(\Gamma_0,\{G_v\},\psi_0),\dots,(\Gamma_N,\{G_v\},\psi_N)$ be a sequence of $\bbZ$-graphs, for $N\in\bbN$. For $\beta=1,\dots,N$, suppose that $(\Gamma_\beta,\{G_v\},\psi_\beta)$ is obtained from $(\Gamma_{\beta-1},\{G_v\},\psi_{\beta-1})$ by a twist, edge sign-change, elementary expansion, elementary contraction, slide, swap or connection. Then there is a sequence of $\bbZ$-graphs $(\Gamma_0',\{G_v\},\psi_0'),\dots,(\Gamma_N',\{G_v\},\psi_N')$ such that
\begin{enumerate}
\item For $\beta=1,\dots,N$ we have that $(\Gamma_\beta,\{G_v\},\psi_\beta)$ is obtained from $(\Gamma_{\beta-1},\{G_v\},\psi_{\beta-1})$ by the identity or a twist, edge sign-change, elementary expansion, slide, swap or connection {\rm(}and notice that we do not allow for elementary contractions{\rm)}.
\item $(\Gamma_0',\{G_v\},\psi_0')=(\Gamma_0,\{G_v\},\psi_0)$ and $(\Gamma_\beta,\{G_v\},\psi_\beta)$ is an elementary subgraph of $(\Gamma_\beta',\{G_v\},\psi_\beta')$ for $\beta=1,\dots,N$.
\end{enumerate}
\end{prop}
\begin{proof}
Analogous to Proposition \ref{no-contractions}.
\end{proof}

\begin{thm}\label{thm:sequence-new-moves-2}
Suppose that a group $G$ acts on two trees $S,T$ with every edge stabilizer isomorphic to $\bbZ$ and with the same elliptic subgroups. Let $(\Gamma,\{G_v\}_{v\in V(\Gamma)},\psi),(\Delta,\{H_u\}_{u\in V(\Delta)},\phi)$ be the corresponding $\bbZ$-graphs, and suppose that they are totally reduced. Then there is a sequence of twists, slides, swaps, connections, edge sign-changes, vertex coordinate-changes, and inductions going from $(\Gamma,\{G_v\}_{v\in V(\Gamma)},\psi),(\Delta,\{H_u\}_{u\in V(\Delta)},\phi)$. Moreover, all the edge sign-changes, vertex sign-changes, and inductions can be performed at the beginning of the sequence.
\end{thm}
\begin{proof}
We proceed as in the proof of Theorem \ref{thm:sequence-new-moves}. By Theorem \ref{thm:Forester} and by Proposition \ref{prop:no-contractions-2}, we can find a sequence
$$(\Gamma,\{G_v\},\psi)=(\Gamma_0,\{G_v\},\psi_0),(\Gamma_1,\{G_v\},\psi_1),\dots,(\Gamma_N,\{G_v\},\psi_N)$$
of $\bbZ$-graphs such that $(\Gamma_\beta,\{G_v\},\psi_\beta)$ is obtained from $(\Gamma_{\beta-1},\{G_v\},\psi_{\beta-1})$ by a twist, edge sign-change, elementary expansion, slide, swap, connection, for $\beta=1,\dots,N$; notice that elementary contractions are not allowed. We can also assume that $(\Gamma_N,\{G_v\},\psi_N)$ contains $(\Delta,\{G_v\},\phi)$ an elementary subgraph, up to applying vertex coordinate-changes on $\Delta$. In particular we have $V(\Gamma)=V(\Gamma_0)\subseteq V(\Gamma_1)\subseteq\dots \subseteq V(\Gamma_N)\supseteq V(\Delta)$.

For every vertex $u\in V(\Gamma_N)\setminus(V(\Gamma)\cup V(\Delta))$, we can now eliminate such vertex from the whole sequence of moves, as in the proof of Theorem \ref{thm:sequence-new-moves}, using Proposition \ref{prop:projection-of-moves-2}. We obtain a new sequence of $\bbZ$-graphs
$$(\Gamma,\{G_v\},\psi)=(\Omega_0,\{G_v\},\omega_0),(\Omega_1,\{G_v\},\omega_1),\dots,(\Omega_L,\{G_v\},\omega_L)$$
with $L\in\bbN$, such that $(\Omega_L,\{G_v\},\omega_L)$ contains $(\Delta,\{G_v\},\phi)$ as an elementary subgraph, for $\beta=1,\dots,L$ we can go from $(\Gamma_{\beta-1},\{G_v\},\psi_{\beta-1})$ to $(\Gamma_\beta,\{G_v\},\psi_\beta)$ with a twist, edge sign-change, elementary expansion, slide, swap or connection, and the set of vertices of every graph along the sequence is contained in $V(\Gamma)\cup V(\Delta)$. In particular $V(\Omega_L)=V(\Gamma)\cup V(\Delta)$.

Suppose that $V(\Omega_L)\supsetneq V(\Delta)$. Since $(\Delta,\{G_v\},\phi)$ is an elementary subgraph of $(\Omega_L,\{G_v\},\omega_L)$, by Lemma \ref{elementary-subgraph2} we can find a valence-$1$ vertex $v_1\in V(\Omega_L)\setminus V(\Delta)$ such that $G_{v_1}\cong\bbZ$ and the unique edge $e_1\in E(\Omega_L)$ with $\tau(e_1)=v_1$ satisfies $\omega_L(e_1)=\pm1$; let us also call $\tau(e_1)=u_1$ the other endpoint of $e_1$. Notice that we must have $v_1\in V(\Gamma)$; by Lemma \ref{lem:powers-along-story-2}, and since $(\Gamma,\{G_v\},\psi)$ is totally reduced, we must also have $u_1\in V(\Delta)$. We now define $(\Omega_L^1,\{G_v\},\omega_L^1)$ obtained from $\Omega_L$ by removing $v_1,e_1$ and we notice that $(\Delta,\{G_v\},\phi)$ is an elementary subgraph of $(\Omega_L^1,\{G_v\},\omega_L^1)$. If $V(\Omega_L^1)\supsetneq V(\Delta)$, then we reiterate the same reasoning and we find another valence-$1$ vertex $v_2\in V(\Omega_L^1)\setminus V(\Delta)$ such that $G_{v_2}\cong\bbZ$ and the unique edge $e_2\in E(\Omega_L^1)$ with $\tau(e_2)=v_2$ satisfies $\psi(e_2)=\pm1$; we also call $u_2=\iota(e_2)$ the other endpoint of $e_2$. As before, we have $v_2\in V(\Gamma)$ and $u_2\in V(\Delta)$. We reiterate the same reasoning until we obtain a GBS graph $(\Omega_L^k,\{G_v\},\omega_L^k)$, for some $k\in\bbN$, containing $(\Delta,\{G_v\},\phi)$ as an elementary subgraph, and satisfying $V(\Omega_L^k)=V(\Delta)$. But this implies that $\Omega_L^k=\Delta$.

This proves that
$$V(\Omega_L)\setminus V(\Delta)=\{v_1,\dots,v_k\}$$
for some $k\in\bbN$ and $v_1,\dots,v_k\in V(\Gamma)$ with $G_{v_1}\cong\dots \cong G_{v_k}\cong\bbZ$, and that
$$E(\Omega_L)\setminus E(\Delta)=\{e_1,\ol{e}_1,\dots,e_k,\ol{e}_k\}$$
such that $\tau(e_i)=v_i$ and $\iota(e_i)=u_i$ for some $u_i\in V(\Omega_L)\setminus V(\Gamma)$ and $\psi(e_i)=\pm1$, for $i=1,\dots,k$. In particular we can define a map $\rho:V(\Gamma)\rar V(\Delta)$ given by the identity on $V(\Gamma)\cap V(\Delta)$, and sending $v_i$ to $u_i$ for $i=1,\dots,k$. For every $u\in V(\Delta)$, the group $G_u$ must be contained, up to conjugation, in some group $G_v$ for some $v\in V(\Gamma)$; thus we can define a map $\theta:V(\Delta)\rar V(\Gamma)$ sending each $u\in V(\Delta)$ to one such $v\in V(\Gamma)$. Since $(\Gamma,\{G_v\},\psi),(\Delta,\{G_v\},\phi)$ are totally reduced, we must have that $\rho$ and $\theta$ are inverse of each other. In particular this means that $\abs{V(\Gamma)}=\abs{V(\Delta)}$ and that $V(\Delta)\setminus V(\Gamma)=\{u_1,\dots,u_k\}$.

Since $\Delta$ is totally reduced, for $i=1,\dots,k$ we consider the vertex $u_i\in V(\Delta)\setminus V(\Gamma)$, and we can find an edge $d_i\in E(\Delta)$ satisfying Property \ref{itm:controlling-edge-2} of Definition \ref{def:totally-reduced-2}. Then we use the set of collapsing data $d_i,e_i$ for the vertex $u_i$ to perform the projection move of Section \ref{sec:projection-2} from $u_i$ to $v_i$. We call $(\Delta',\{G_v\},\phi')$ the graph obtained performing such projection for all $k=1,\dots,n$. Notice that, since the collapsing data consists of only two edges $d_i,e_i$, in this case, the projection move coincides with the following: applying an induction move (plus potentially a vertex coordinate-change) at the vertex $u_i$ in $\Delta$, and then removing the edge $e_i$ and identifying $u_i$ and $v_i$, and then performing some slide moves to adjust the remaining edges. In particular, the graph $(\Delta',\{G_v\},\phi')$ is obtained from $(\Delta,\{G_v\},\phi)$ by performing an induction move (and potentially a vertex coordinate-change move) at some vertices, and then slide moves.

We finally eliminate the vertices $u_1,\dots,u_k$ from the whole sequence of $\bbZ$-graphs $(\Gamma,\{G_v\},\psi)=(\Omega_0,\{G_v\},\omega_0),\dots,(\Omega_L,\{G_v\},\omega_L)$ using the projection move of Section \ref{sec:projection-2}. We obtain a sequence of $\bbZ$-graphs going from $(\Gamma,\{G_v\},\psi)$ to $(\Delta',\{G_v\},\phi')$ and such that every two consecutive $\bbZ$-graphs in the sequence are related by a twist, an edge sign-change, a slide, a swap or a connection. The statement  follows.
\end{proof}

\subsection{Comparison of the isomorphism problems}

\begin{defn}
    An element $r\in G$ is called \textbf{root} if we cannot write $r=s^k$ for $s\in G$ and $k\ge2$.
\end{defn}

\begin{defn}\label{def:unique-roots}
    A group $G$ has \textbf{unique roots} if for every $g\in G\setminus\{1\}$ there is a unique root $r_g\in G$ such that $g=r_g^k$ for some $k\ge1$.
\end{defn}

\begin{lem}\label{lem:unique_roots}
    If $G$ has unique roots, then we have the following:
    \begin{enumerate}
        \item $G$ is torsion-free.
        \item If $g\in G\setminus\{1\}$ and $g^i$ is conjugate to $g^j$ for $i,j\in\bbZ$, then $i=j$.
    \end{enumerate}
\end{lem}
\begin{proof}
    If an element has finite order, then its root has also finite order. But a root element cannot have finite order, by definition of root (because it would be a proper power of itself). Thus $G$ is torsion-free.

    Suppose that $g^i=ag^j\ol{a}$ for some $i,j\ge1$ and $a\in G$. Let $g=r_g^k$ for $k\ge1$ and notice that $r_g^{ki}=(ar_g\ol{a})^{kj}$. But $r_g$ and $ar_g\ol{a}$ are both root and $ki,kj\ge1$, thus by the uniqueness of the root we must have $r_g=ar_g\ol{a}$ and so $a$ and $r_g$ commute. It follows that $r_g^{ki}=r_g^{kj}$ and thus $r_g^{k(i-j)}=1$; since $G$ is torsion-free and $k\ge 1$, this implies that $i=j$.
\end{proof}

\begin{defn}
    We say that a $\bbZ$-graph has \textbf{unique roots} if all of its vertex groups do.
\end{defn}

\begin{remark}
    Note that this is not the same as requiring that the fundamental group of the graph of groups has unique roots. For example, $\bbZ$ has unique roots, but GBSs do not (both existence and uniqueness can fail).
\end{remark}

Consider a graph of groups with vertex groups with unique roots and edge groups infinite cyclic. For every edge, we can take the generator of such edge, and look at its image in the adjacent vertex group: we want to choose, inside the vertex group, a root for such element. We do not want this list of roots to contain repetitions: if two edge groups are glued on (powers of) the same root, then we want to list the root only once. Moreover, since edge inclusions can be changed by conjugation without changing the corresponding Bass-Serre tree, we want conjugate roots to be considered the same. Finally, for technical reasons, we only want to consider roots inside non-cyclic vertex groups. This is all summarized in the following definition.

\begin{defn}\label{def:choice-of-roots}
    Let $(\Gamma,\{G_v\}_{v\in V(\Gamma)},\psi)$ be a $\bbZ$-graph with unique roots. 
    A \textbf{choice of non-cyclic roots} is a {\rm(}minimal{\rm)} finite set $\notcyclicR{\Gamma}$ satisfying the following property: for all $e\in E(\Gamma)$ with $\tau(e)\in\notcyclicV{\Gamma}$, there is a unique $\rho\in\notcyclicR{\Gamma}$ such that the root of $\psi(e)$ is conjugate to $\rho^{\pm1}$ in $G_{\tau(e)}$. 
    These are uniquely determined up to taking inverses and conjugates of the roots {\rm(}in their respective vertex groups{\rm)}.
\end{defn}

We notice that a twist, edge sign-change, vertex coordinate-change, slide, induction, swap or connection from $(\Gamma_0,\{G_v\}_{v\in V(\Gamma_0)},\psi_0)$ to $(\Gamma_1,\{G_v\}_{v\in V(\Gamma_1)},\psi_1)$ induces the following:
\begin{itemize}
    \item A bijection $b:\notcyclicV{\Gamma_1}\rar \notcyclicV{\Gamma_0}$.
    \item Isomorphisms $b_v:G_v\rar G_{b(v)}$ for all $v\in \notcyclicV{\Gamma_1}$.
    \item A bijection $\beta:\notcyclicR{\Gamma_1}\rar\notcyclicR{\Gamma_0}$.
\end{itemize}

\begin{thm}\label{thm:GBS-to-Zgraphs}
    Let $(\Gamma_0,\{G_v\}_{v\in V(\Gamma_0)},\psi_0)$ and $(\Gamma_1,\{G_v\}_{v\in V(\Gamma_1)},\psi_1)$ be totally reduced $\bbZ$-graphs with unique roots. Let $b:\notcyclicV{\Gamma_1}\rar \notcyclicV{\Gamma_0}$ be a bijection and let $\beta:\notcyclicR{\Gamma_1}\rar\notcyclicR{\Gamma_0}$ be a bijection which is induced by isomorphisms $b_v:G_v\rar G_{b(v)}$ for $v\in \notcyclicV{\Gamma_1}$. Then there are GBS graphs $(\Delta_0,\phi_0),(\Delta_1,\phi_1)$ such that the following are equivalent:
    \begin{enumerate}
        \item There is a sequence of twists, edge sign-changes, vertex coordinate-changes, slides, inductions, swaps, and connections going from $(\Gamma_0,\{G_v\}_{v\in V(\Gamma_0)},\psi_0)$ to $(\Gamma_1,\{G_v\}_{v\in V(\Gamma_1)},\psi_1)$ and inducing the bijections $b:\notcyclicV{\Gamma_1}\rar \notcyclicV{\Gamma_0}$ and $\beta:\notcyclicR{\Gamma_1}\rar\notcyclicR{\Gamma_0}$.
        \item There is a sequence of edge sign-changes, vertex sign-changes, slides, inductions, swaps, and connections going from $(\Delta_0,\phi_0)$ to $(\Delta_1,\phi_1)$.
    \end{enumerate}
    Moreover, $(\Delta_0,\phi_0),(\Delta_1,\phi_1)$ can be constructed algorithmically from the data $(\Gamma_0,\{G_v\}_{v\in V(\Gamma_0)},\psi_0),$ $(\Gamma_1,\{G_v\}_{v\in V(\Gamma_1)},\psi_1),b,\beta$.
\end{thm}
\begin{proof}
    For every edge $e\in E(\Gamma_0)$ consider the element $\psi_0(e)\in G_{\tau(e)}$ and take the unique root $\psi_0(e)=r_{\psi_0(e)}^{m_e}$ for $m_e\ge1$ integer; let $\cP(\Gamma_0)$ be the set of primes that appear in the factorization of $m_e$ for some $e\in E(\Gamma_0)$. Notice that $\cP(\Gamma_0)$ is invariant under twists, edge sign-changes, vertex coordinate-changes, slides, inductions, swaps, and connections. Define $\cP(\Gamma_1)$ similarly and let $\cP=\cP(\Gamma_0)\cup\cP(\Gamma_1)$. Take a choice of non-cyclic roots $\notcyclicR{\Gamma_0}$ (see Definition \ref{def:choice-of-roots}) and for every $\rho\in\notcyclicR{\Gamma_0}$ choose a different prime $p_\rho\not\in\cP$.

    Given a $\bbZ$-graph with unique roots $(\Gamma,\{G_v\},\psi)$, together with a bijection $a:\notcyclicV{\Gamma}\rar\notcyclicV{\Gamma_0}$ and $\alpha:\notcyclicR{\Gamma}\rar\notcyclicR{\Gamma_0}$, define the GBS graph $(\Delta,\phi)$ as follows:
    \begin{enumerate}
        \item $V(\Delta)=\cyclicV{\Gamma}\sqcup\{*\}$.
        \item $E(\Delta)=E(\Gamma)$ with the same reverse map.
        \item For $e\in E(\Gamma)$ with $\tau_\Gamma(e)\in\notcyclicV{\Gamma}$, we have that $\alpha(\psi(e))$ is conjugate to $\rho^k$ for a unique $\rho\in\notcyclicR{\Gamma_0}$ and $k\in\bbZ\setminus\{0\}$, and we set $\tau_\Delta(e)=*$ and $\phi(e)=p_\rho\cdot k$.
        \item For $e\in E(\Gamma)$ with $\tau_\Gamma(e)\in\cyclicV{\Gamma}$, we set $\tau_\Delta(e)=\tau_\Gamma(e)$ and $\phi(e)=\psi(e)$.
    \end{enumerate}
    This construction behaves well with respect to the moves: if we perform a twist, edge sign-change, vertex-coordinate change, slide, induction, swap, or connection on $(\Gamma,\{G_v\},\psi)$, then the graph $(\Delta,\phi)$ changes by identity, edge sign-change, identity, slide, induction, swap or connection respectively. Conversely, if we perform an edge sign-change, vertex sign-change slide, induction, swap or connection in $(\Delta,\phi)$, then the same move can be performed in $(\Gamma,\{G_v\},\psi)$, possibly after some twist. Thus sequences of moves on $(\Gamma,\{G_v\},\psi)$, considered up to twists and vertex-coordinate changes, are in bijection with sequences of moves on the corresponding graphs $(\Delta,\phi)$.

    Let $(\Delta_0,\phi_0),(\Delta_1,\phi_1)$ be the graphs obtained from $(\Gamma_0,\{G_v\},\psi_0),(\Gamma_1,\{G_v\},\psi_1)$ respectively by the above construction. It follows that there is a sequence of twists, edge sign-changes, vertex-coordinate changes, slides, inductions, swaps, and connections going from $(\Gamma_0,\{G_v\},\psi_0)$ to $(\Gamma_1,\{G_v\},\psi_1)$ if and only if there is a sequence of sign-changes, slides, inductions, swaps, connections going from $(\Delta_0,\phi_0)$ to $(\Delta_1,\phi_1)$. The statement  follows.
\end{proof}

\subsection{Reduction of the isomorphism problem to GBSs}

\begin{defn}\label{def:algorithmic-conjugacy}
    Let $\fF$ be a family of finitely presented torsion-free groups. We say that $\fF$ has \textbf{algorithmic conjugacy problem} if there is an algorithm that, given $G\in\fF$ and $g_1,g_2\in G$, decides whether $g_1$ is conjugated to $g_2$.
\end{defn}

\begin{defn}\label{def:algorithmic-isomorphism}
    Let $\fF$ be a family of finitely presented torsion-free groups. We say that $\fF$ has \textbf{algorithmic isomorphism problem with cyclic peripherals} if there is an algorithm that, given $G,H\in\fF$ and $k\ge0$ and ordered $k$-tuples $([g_1],\dots,[g_k]),([h_1],\dots,[h_k])$ of conjugacy classes in $G,H$ respectively, decides whether or not there is an isomorphism between $G$ and $H$ sending the conjugacy class $[g_i]$ to $[h_i]$ for $i=1,\dots,k$.
\end{defn}

\begin{defn}\label{def:algorithmic-unique-roots}
    Let $\fF$ be a family of finitely presented torsion-free groups. We say that $\fF$ has \textbf{algorithmic unique roots} if every $G\in\fF$ has unique roots and there is an algorithm that, given $G\in\fF$ and $g\in G$, computes the root $r_g\in G$.
\end{defn}

\begin{thm}\label{thm:isomorphism-other-families}
    Let $\fF$ be a family of finitely presented torsion-free groups. Suppose that we have the following conditions:
    \begin{enumerate}
        \item $\fF$ has algorithmic conjugacy problem {\rm(Definition \ref{def:algorithmic-conjugacy})}.
        \item $\fF$ has algorithmic isomorphism problem with cyclic peripherals {\rm(Definition \ref{def:algorithmic-isomorphism})}.
        \item $\fF$ is $\bbZ$-algorithmic {\rm(Definition \ref{def:Zalgorithmic})}.
        \item $\fF$ has algorithmic unique roots {\rm(Definition \ref{def:algorithmic-unique-roots})}.
    \end{enumerate}
    If there is an algorithm to decide the isomorphism problem for GBSs, then there is an algorithm to decide the isomorphism problem for the fundamental group of graphs of groups with vertex groups in $\fF$ and cyclic edge groups.
\end{thm}
\begin{proof}
    Let $\cG_0,\cG_1$ be graphs of groups with vertex groups in $\fF$ and cyclic edge groups. By Proposition \ref{prop:algorithmic-JSJ} we can assume that $\cG_0,\cG_1$ are JSJ decompositions of their respective fundamental groups, and by Proposition \ref{prop:compute-totally-reduced-2} we can assume that the corresponding $\bbZ$-graphs are totally reduced. It is easy to algorithmically check whether $\pi_1(\cG_0),\pi_1(\cG_1)$ are isomorphic to $\bbZ,\bbZ^2,K$ (the Klein bottle group), thus we assume that this is not the case. Let $(\Gamma_0,\{G_v\}_{v\in V(\Gamma_0)},\psi_0),(\Gamma_1,\{G_v\}_{v\in V(\Gamma_1)},\psi_1)$ be the corresponding $\bbZ$-graphs.

    There are finitely many bijections $b:\notcyclicV{\Gamma_1}\rar \notcyclicV{\Gamma_0}$; for each such $b$, there are finitely many bijections $\beta:\notcyclicR{\Gamma_1}\rar\notcyclicR{\Gamma_0}$, and for each of them, we can check algorithmically whether it is induced by isomorphisms $b_v:G_v\rar G_{b(v)}$. For each such pair $(b,\beta)$, we can use Theorem \ref{thm:GBS-to-Zgraphs}, plus the assumption that the isomorphism problem for GBS groups is decidable, to algorithmically check whether there is a sequence of moves going from $(\Gamma_0,\{G_v\}_{v\in V(\Gamma_0)},\psi_0)$ to $(\Gamma_1,\{G_v\}_{v\in V(\Gamma_1)},\psi_1)$ and inducing those bijections.

    If we find such a sequence of moves, then the fundamental groups of $\cG_0,\cG_1$ must be isomorphic. Conversely, if the fundamental groups of $\cG_0,\cG_1$ are isomorphic, then by Theorem \ref{thm:Forester} there must be such a sequence of moves, which will induce some bijections $(b,\beta)$. Thus if we do not find such a sequence of moves for any choice of $(b,\beta)$, then the fundamental groups of $\cG_0,\cG_1$ are not isomorphic. The statement  follows.
\end{proof}

In particular, we can choose the family of groups $\fF$ as follows:

\begin{prop}
The class $\fF$ of finitely presented, torsion-free relatively hyperbolic groups with nilpotent parabolic groups satisfies the conditions of {\rm \Cref{thm:isomorphism-other-families}}.  
\end{prop}
\begin{remark}
    In particular this includes all finitely presented torsion-free toral relatively hyperbolic groups.
\end{remark}
\begin{proof}
By \cite{O06}, the conjugacy problem for $\fF$ is reduced to the conjugacy problem for the family of parabolic subgroups. But for nilpotent groups, the conjugacy problem is decidable, see \cite{Bl65}.

Similarly, by \cite{O06}, finding algorithmic unique roots for groups in $\fF$ is reduced to the same problem in the parabolic subgroups. But it is well-known that in torsion-free nilpotent groups roots are unique, see for instance \cite[Theorem 16.2.8]{KM79} and the proof can be made algorithmic (we refer the reader to \cite{MMNV22} for complexity results).

The isomorphism problem for $\fF$ is done in \cite{DT19}. For effective Grushko and JSJ decompositions, and thus for the property of being $\bbZ$-algorithmic, we refer the reader to \cite{Tou18}.
\end{proof}

\begin{cor}\label{cor:reduction_iso_relative_hyp_nilpotent}
If there is an algorithm to decide the isomorphism problem for GBSs, then there is an algorithm to decide the isomorphism problem for graphs of groups with vertex groups that are finitely presented, torsion-free relatively hyperbolic groups with nilpotent parabolic subgroups, and cyclic edge groups. 
\end{cor}

\begin{remark}
As both free and free abelian groups can be easily included in the family $\fF$, it would be interesting to know whether right-angled Artin groups can be included in $\fF$ too. Given any presentation for a group $G$, if we are guaranteed that $G$ is RAAG, then we can algorithmically find the (unique) canonical RAAG presentation for $G$ (by just brute-force enumerating all possible presentations for $G$, applying sequences of Tietze transformations to the given presentation).

Once that is done, the decidability of the conjugacy problems is well-known, see for instance \cite{Crisp2009TheCP}. Algorithmic unique roots are well-known too, and we refer the reader to \cite{Bau77}. Since the canonical RAAG presentation is unique, the isomorphism problem with prescribed cyclic peripherals is reduced to the automorphism orbit decidability for $k$-tuples of conjugacy classes, which is done in \cite{Day14}.

The Grushko decomposition of a RAAG $G_\Delta$ can be described in terms of the connected components of the graph $\Delta$, and thus RAAGs are $1$-algorithmic. An alternative argument can be given using Touikan's work \cite{Tou18}. The cyclic JSJ decomposition of a one-ended RAAG $G_\Delta$ can be described in terms of the cut-vertices of $\Delta$, see \cite{Clay14}. However a non-one-ended RAAG $G$ can be endowed with an arbitrary peripheral structure $\cP$ (so that $G$ one-ended relative to $\cP$), and deciding the existence of a $\bbZ$-splitting relative to such peripheral structure is, to the best of our knowledge, not known (for free groups it is already a hard problem, see \cite{Cas16}).
\end{remark}

\section{Additional considerations and questions}\label{sec:questions}

In this section, we formulate two questions for future research. The first question is about the study of the automorphism group of a GBS. The second is about the possibility of generalizing the results of this paper to wider families of groups.

\subsection{Automorphisms of GBSs}

Given two GBSs $G_1,G_2$, we consider the following two problems:

\begin{enumerate}
    \item\label{itm:iso-weak} Understand whether there exists an isomorphism from $G_1$ to $G_2$.
    \item\label{itm:iso-strong} List/characterize all the possible isomorphisms from $G_1$ to $G_2$.
\end{enumerate}

Problem \ref{itm:iso-weak} is the isomorphism problem, and it only requires a yes or a no as an answer; the object that we want to understand in this case is the quotient $\fD_G/\out{G}$ of the deformation space by the action of the outer automorphism group. On the other hand, Problem \ref{itm:iso-strong} is strictly stronger, and it is equivalent to studying the automorphism group of a given GBS; the object that we want to study for this problem is the deformation space $\fD_G$ (without taking any quotient). With Theorem \ref{thm:sequence-new-moves} we are showing that the new set of moves that we introduced is suitable for the study of Problem \ref{itm:iso-weak}; we are interested in knowing whether it can be used to study Problem \ref{itm:iso-strong} as well, as we now explain.

Let $(\Gamma,\psi)$ be a GBS graph with associated graph of groups $\cG$. Suppose that we apply a sign-change, elementary expansion, elementary contraction, slide, induction, swap or connection in order to obtain a new GBS graph $(\Gamma',\psi')$ with associated graph of groups $\cG'$. Each one of those moves induces an isomorphism $\pi_1(\cG)\rar\pi_1(\cG')$, given by the formulas provided in Section \ref{sec:moves}. Moreover, we can apply a twist move to the GBS graph $(\Gamma,\psi)$: this does not change the GBS graph, but it induces a non-identity isomorphism from $\pi_1(\cG)$ to itself (which consists essentially in multiplying and edge by an element in the adjacent vertex group). Finally, given two GBS graphs $(\Gamma,\psi),(\Delta,\phi)$ with corresponding graphs of groups $\cG,\cH$, a label preserving isomorphism $\Gamma\rar\Delta$ induces an isomorphism $\pi_1(\cG)\rar\pi_1(\cH)$.

Let now $(\Gamma_0,\psi_0),(\Gamma_1,\psi_1)$ be two GBS graphs, with corresponding graphs of groups $\cG_0,\cG_1$ (and suppose that they are not isomorphic to $\bbZ,\bbZ^2,K$). Forester's result \cite{For02}, together with the standard Bass-Serre correspondence, implies that for every isomorphism $\theta:\pi_1(\cG_0)\rar\pi_1(\cG_1)$, there is a sequence of moves (sign-changes, twists, expansions, contractions, slides), followed by a label-preserving isomorphism of GBS graphs, inducing exactly the isomorphism $\theta$. This means that one can try and use Forester's set of moves to study Problem \ref{itm:iso-strong}, as every isomorphism can be encoded as a sequence of moves.

On the other hand, it is not clear at all whether the new set of moves that we introduced satisfies the same property. Let $(\Gamma,\psi)$ be a GBS graph with a redundant vertex $v$, and let $(\Delta,\phi),(\Delta',\phi')$ be two GBS graphs obtained by eliminating $v$ with two projection moves with respect to different sets of data/constants (Section \ref{sec:projection}); let also $\cG,\cH,\cH'$ be the corresponding graphs of groups. By the sequence of moves described in Proposition \ref{projection}, we obtain two isomorphisms $\pi_1(\cG)\rar\pi_1(\cH)$ and $\pi_1(\cG)\rar\pi_1(\cH')$. On the other hand, Proposition \ref{prop:indep-data} tells us that there is a sequence of edge sign-changes, slides, swaps, and connections going from $(\Delta,\phi)$ to $(\Delta',\phi')$, inducing an isomorphism $\pi_1(\cH)\rar\pi_1(\cH')$. In general the diagram

\begin{center}
\begin{tikzpicture}[scale=0.6]
\node (G) at (0,3) {$\pi_1(\cG)$};
\node (H) at (-2,0) {$\pi_1(\cH)$};
\node (H') at (2,0) {$\pi_1(\cH')$};
\draw[->] (G) to (H);
\draw[->] (G) to (H');
\draw[->] (H) to (H');
\end{tikzpicture}
\end{center}

\noindent does not commute, and it is natural to wonder about the following:

\begin{question}
    Can the sequence of moves of {\rm Proposition \ref{prop:indep-data}} be chosen in such a way that the diagram commutes? Do we need to allow for twists? Do we need to define and allow for some other move {\rm(}e.g. some twisted version of swap and/or connection{\rm)}?
\end{question}

\begin{remark}
    The reader interested in the structure of the automorphism group of a GBS should look at \cite{Cla09,Lev07}. We point out that the automorphism group of a GBS is not finitely generated in general.
\end{remark}

\subsection{Deformation spaces with non-cyclic edge groups}

We gave evidence that generalized Baumslag-Solitar groups encode most of the structure of many other deformation spaces, and the techniques that we developed for GBSs also work in general, as long as the edge groups are isomorphic to $\bbZ$. One might be tempted to try and generalize the strategy adopted in this paper to the study of deformation spaces with different edge groups. As a test case, we are interested in graphs of $\bbZ^2$ (i.e. graphs of groups with all vertex groups and edge groups isomorphic to $\bbZ^2$); among torsion-free groups, $\bbZ^2$ can be considered the ``easiest" group after $\bbZ$.

\begin{remark}\label{rmk:isomorphism-undecidable}
    We point out that the isomorphism problem is known to be undecidable for graphs of $\bbZ^4$ (see \cite{Zim90,Lev08}): this is due to the fact that $\GL{4}{\bbZ}$ contains a finitely generated subgroup with unsolvable membership problem (which follows from the fact that $\GL{4}{\bbZ}$ contains $F_2\times F_2$). However, the same does not hold for $\GL{2}{\bbZ}$, and the isomorphism problem for graphs of $\bbZ^2$ is still open. Moreover, even if the isomorphism problem for graphs of $\bbZ^2$ would be undecidable, it might still be possible to generalize Theorem \ref{thm:sequence-new-moves}.
\end{remark}

\begin{remark}
    Graphs of $\bbZ^2$ also have a symbolic importance, as $\bbZ^2$ is the fundamental group of the torus, and this kind of splittings appears naturally in the study of $3$-manifolds; however, we point out that the JSJ decomposition for a $3$-manifolds is rigid in the strongest possible sense - the JSJ tori in an irreducible orientable compact manifold are uniquely determined up to isotopy, see Theorem 1.9 in \cite{Hat} - and thus none of our techniques is of any use there.
\end{remark}

The strategy would consist of the following steps: define a notion of redundant vertex, find a way of eliminating redundant vertices, and investigate what happens if a redundant vertex is eliminated in two different ways (trying to define a new set of moves - the analogue of swap and connection - and to generalize Theorem \ref{thm:sequence-new-moves}). During this process,ss, we encounter some obstructions: we now outline one of them.

A key step is the definition of an induction move. Suppose that we have a vertex $v$ and an edge $e$ with $\iota(e)=\tau(e)=v$. Suppose also that the vertex group at $v$ and the edge group at $e$ coincide, and the map $\psi_{\ol{e}}:G\rar G$ is the identity, while the map $\psi_e:G\rar G$ is given by an injective endomorphism $\varphi:G\rar G$. A simple version of the induction move can be obtained by taking any subgroup $\varphi(G)\sgr G'\sgr G$ and substituting $G$ with $G'$, as shown in Figure \ref{fig:new-induction}. A nested version of the induction move consists of substituting $G$ with a subgroup $G'$ such that $\varphi^k(G)\sgr G'\sgr G$ for some $k\in\bbN$; however, in order for the above move to make sense, we are forced to require an extra condition, namely $\varphi(G')\sgr G'$, as otherwise, the right-hand graph of the nested move in Figure \ref{fig:new-induction} does not even make sense. Notice that this extra condition was trivial in the case of GBSs, while it is non-trivial now.

This leads us to the following intriguing example. Consider $G=\bbZ^2$ with basis $(1,0),(0,1)$; let $\varphi:G\rar G$ be the injective endomorphism given by $\varphi((1,0))=(8,0)$ and $\varphi((0,1))=(0,4)$, and let $G'\sgr G$ be given by $G'=\gen{(16,0),(0,16),(2,2)}$. Consider the graph of groups as in Figure \ref{fig:counterexample}, where $G,G',\varphi$ are as above, and $K=\bbZ^2$ with the inclusion $G'\rar K$ being any map which is not an isomorphism. Since $\varphi^2(G)\sgr G'\sgr G$, it is easy to see that $G$ is conjugated to a subgroup of $K$; however, since $\varphi(G')\not\sgr G'$, we are not able to use induction to eliminate the vertex $G$ from the graph. Thus it seems natural to ask the following questions:

\begin{question}\label{quest:counterexample}
    Is there a sequence of expansions, contractions, and slides going from the graph of groups of {\rm Figure \ref{fig:counterexample}} to a $1$-vertex graph?
\end{question}

We suspect that the answer to the above Question \ref{quest:counterexample} is negative. This would mean that the notion of redundant vertex, which plays such a crucial role in this work, is being split into two non-equivalent notions. We say that a vertex is \textbf{weakly redundant} if its vertex group is contained, up to conjugation, in the vertex group at another vertex, and we say that it is \textbf{strongly redundant} if there is a sequence of expansions, contractions and slides eliminating that vertex. It is clear that if a vertex is strongly redundant then it is weakly redundant.

\begin{question}
    Are the two notions of being weakly redundant and strongly redundant equivalent for graphs of $\bbZ^2$? In case of a negative answer, is there some other intrinsic characterization for the strongly redundant vertices?
\end{question}

\begin{question}\label{quest:algorithmic-Z2}
    Is being weakly redundant {\rm(}resp. strongly redundant{\rm)} algorithmically decidable in a graph of $\bbZ^2$?
\end{question}

When trying to deal with this last Question \ref{quest:algorithmic-Z2}, the reader will very quickly end up studying the structure of the monoid of the endomorphisms of $\bbZ^2$. On the other hand, the undecidability of the isomorphism problem for graphs of $\bbZ^4$ is based exactly on the undecidability of problems inside the group of automorphisms of $\bbZ^4$ (see Remark \ref{rmk:isomorphism-undecidable}). 

\begin{figure}[H]
\centering
\begin{tikzpicture}[scale=1]

\begin{scope}[shift={(0,0)}]
\node (v) at (0,0) {$G$};
\draw[-,dashed] (-1.8,0.9) to (-1.2,0.6);
\draw[-] (-1.2,0.6) to (v);
\draw[-,dashed] (-1.8,0) to (-1.2,0);
\draw[-] (-1.2,0) to (v);
\draw[-,dashed] (-1.8,-0.9) to (-1.2,-0.6);
\draw[-] (-1.2,-0.6) to (v);
\node () at (-0.4,0.5) {$H$};
\node () at (-0.8,0.17) {$J$};
\node () at (-0.4,-0.5) {$K$};
\draw[->,out=-40,in=40,looseness=25] (v) to node[left]{$\varphi$} (v);
\node () at (0.6,-0.15) {$G$};
\node () at (0.5,0.8) {$\varphi(G)$};
\end{scope}

\draw[->] (4,0) to node[above]{induction} node[below]{(simple)} (6.5,0);

\begin{scope}[shift={(9,0)}]
\node (v) at (0,0) {$G'$};
\draw[-,dashed] (-1.8,0.9) to (-1.2,0.6);
\draw[-] (-1.2,0.6) to (v);
\draw[-,dashed] (-1.8,0) to (-1.2,0);
\draw[-] (-1.2,0) to (v);
\draw[-,dashed] (-1.8,-0.9) to (-1.2,-0.6);
\draw[-] (-1.2,-0.6) to (v);
\node () at (-0.7,0.7) {$\varphi(H)$};
\node () at (-1,0.17) {$\varphi(J)$};
\node () at (-0.7,-0.7) {$\varphi(K)$};
\draw[->,out=-40,in=40,looseness=25] (v) to node[left]{$\varphi$} (v);
\node () at (0.6,-0.15) {$G'$};
\node () at (0.5,0.8) {$\varphi(G')$};
\end{scope}

\begin{scope}[shift={(0,-3.5)}]
\node (v) at (0,0) {$G$};
\draw[-,dashed] (-1.8,0.9) to (-1.2,0.6);
\draw[-] (-1.2,0.6) to (v);
\draw[-,dashed] (-1.8,0) to (-1.2,0);
\draw[-] (-1.2,0) to (v);
\draw[-,dashed] (-1.8,-0.9) to (-1.2,-0.6);
\draw[-] (-1.2,-0.6) to (v);
\node () at (-0.4,0.5) {$H$};
\node () at (-0.8,0.17) {$J$};
\node () at (-0.4,-0.5) {$K$};
\draw[->,out=-40,in=40,looseness=25] (v) to node[left]{$\varphi$} (v);
\node () at (0.6,-0.15) {$G$};
\node () at (0.5,0.8) {$\varphi(G)$};
\end{scope}

\draw[->] (4,-3.5) to node[above]{induction} node[below]{(nested)} (6.5,-3.5);

\begin{scope}[shift={(9,-3.5)}]
\node (v) at (0,0) {$G'$};
\draw[-,dashed] (-1.8,0.9) to (-1.2,0.6);
\draw[-] (-1.2,0.6) to (v);
\draw[-,dashed] (-1.8,0) to (-1.2,0);
\draw[-] (-1.2,0) to (v);
\draw[-,dashed] (-1.8,-0.9) to (-1.2,-0.6);
\draw[-] (-1.2,-0.6) to (v);
\node () at (-0.7,0.7) {$\varphi^k(H)$};
\node () at (-1,0.17) {$\varphi^k(J)$};
\node () at (-0.7,-0.7) {$\varphi^k(K)$};
\draw[->,out=-40,in=40,looseness=25] (v) to node[left]{$\varphi$} (v);
\node () at (0.6,-0.15) {$G'$};
\node () at (0.5,0.8) {$\varphi(G')$};
\end{scope}
\end{tikzpicture}
\caption{The induction move in the general setting. Above you can see the simple version of induction, with $\varphi(G)\sgr G'\sgr G$. Below you can see the nested version of induction, with $\varphi^k(G)\sgr G'\sgr G$ and $\varphi(G')\sgr G'$.}\label{fig:new-induction}
\end{figure}
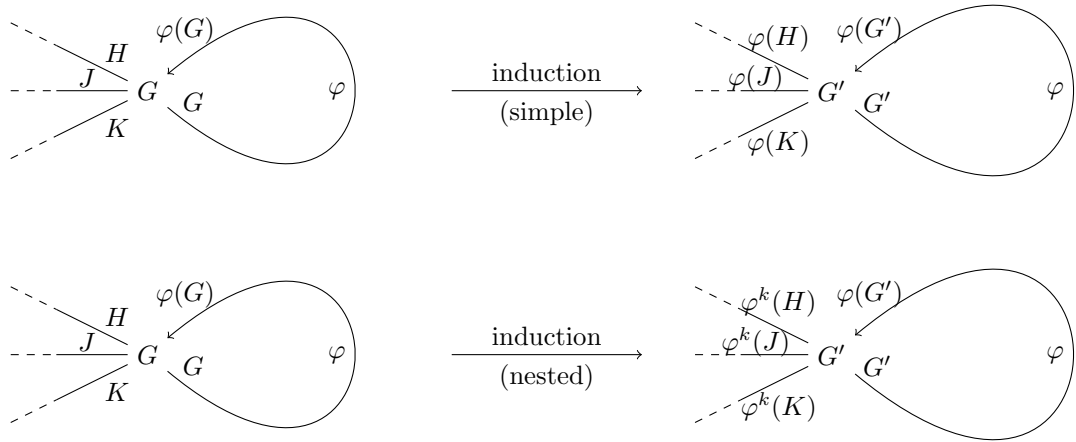

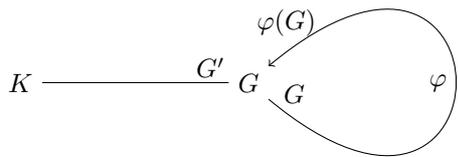
\begin{figure}[H]
\centering
\begin{tikzpicture}[scale=1]

\node (v) at (0,0) {$G$};
\node (u) at (-3,0) {$K$};
\draw[-] (u) to (v);
\node () at (-0.5,0.2) {$G'$};
\draw[->,out=-40,in=40,looseness=25] (v) to node[left]{$\varphi$} (v);
\node () at (0.6,-0.15) {$G$};
\node () at (0.5,0.8) {$\varphi(G)$};

\end{tikzpicture}
\caption{An example of a graph of $\bbZ^2$ with an obstruction to performing induction. Here $G=\bbZ^2$, the homomorphism $\varphi:G\rar G$ is given by $\varphi((1,0))=(8,0)$ and $\varphi((0,1))=(0,4)$, and $G'=\gen{(16,0),(0,16),(2,2)}$.}\label{fig:counterexample}
\end{figure}

\bibliographystyle{alpha}

\end{document}